 \documentclass[11pt,letterpaper]{amsart}
\usepackage{amssymb}
\usepackage{enumitem}
\usepackage{bbm}
\usepackage[usenames, dvipsnames]{color}

\usepackage{hyperref}
\usepackage{amsmath}

\usepackage{verbatim}
\usepackage{todonotes}
\usepackage{pxfonts}
\usepackage{tikz}
\usetikzlibrary{arrows, automata}

\allowdisplaybreaks[4]

\theoremstyle{plain}
\newtheorem{theorem}{Theorem}[section]
\newtheorem{lemma}[theorem]{Lemma}

\newtheorem{proposition}[theorem]{Proposition}

\theoremstyle{definition}

\newtheorem{assumption}[theorem]{Assumption}

\theoremstyle{remark}
\newtheorem{remark}[theorem]{Remark}

\newcommand{\Div}{\operatorname{div}}

\numberwithin{equation}{section}

\newcommand{\bR}{\mathbb{R}}

\newcommand\cD{\mathcal{D}}

\newcommand\cQ{\mathcal{Q}}

\makeatletter
\def\dashint{\operatorname%
{\,\,\text{\bf--}\kern-.98em\DOTSI\intop\ilimits@\!\!}}
\makeatother

\begin{document}
	
\title[Higher regularity for equations from composite materials]{Higher regularity for solutions to equations arising from composite materials}

\author[H. Dong]{Hongjie Dong}
\address[H. Dong]{Division of Applied Mathematics, Brown University, 182 George Street, Providence, RI 02912, USA}
\email{Hongjie\_Dong@brown.edu }
\thanks{H. Dong was partially supported by the Simons Foundation, grant no. 709545, a Simons fellowship, grant no. 007638, the NSF under agreement DMS-2055244, and the Charles Simonyi Endowment at the Institute for Advanced Study.}

\author[L. Xu]{Longjuan Xu}
\address[L. Xu]{Department of Mathematics, National University of Singapore, 10 Lower Kent Ridge Road, Singapore 119076}
\email{ljxu311@163.com}
\thanks{}


\begin{abstract}
We consider parabolic systems in divergence form with piecewise $C^{(s+\delta)/2,s+\delta}$ coefficients and data in a  bounded domain consisting of a finite number of cylindrical subdomains with interfacial boundaries in $C^{s+1+\mu}$, where $s\in\mathbb N$, $\delta\in (1/2,1)$, and $\mu\in (0,1]$. We establish piecewise $C^{(s+1+\mu')/2,s+1+\mu'}$ estimates for weak solutions to such parabolic systems, where $\mu'=\min\big\{1/2,\mu\big\}$, and the estimates are independent of the distance between the interfaces. In the elliptic setting, our results answer an open problem (c) in Li and Vogelius (Arch. Rational Mech. Anal. 153 (2000), 91--151).
\end{abstract}

\maketitle

\section{Introduction and main results}
\subsection{Introduction}
In this paper, we study higher derivative estimates of solutions to divergence form parabolic systems
\begin{equation}\label{systems}
-u_t+D_\alpha(A^{\alpha\beta}D_\beta u)=D_\alpha f^\alpha=:\Div f\quad\mbox{in}~\cQ:=(-T,0)\times\cD,
\end{equation}
where $T\in(0,\infty)$, $\cD:=\cup_{j=1}^{M}\overline{\cD}_j\setminus\partial\cD\subset\mathbb R^{d}$ is a bounded domain containing a finite number of subdomains $\cD_j$, $j=1,\ldots,M$, and $d\geq2$.  Here
$$
u=(u_{1},\ldots,u_{n})^{\top},\quad f^{\alpha}=(f^{\alpha}_{1},\ldots,f^{\alpha}_{n})^{\top}
$$
are (column) vector-valued functions, $A^{\alpha\beta}$  are $n\times n$ matrices which  satisfy the strong ellipticity condition: there exists a number $\nu>0$ such that for any $\xi=(\xi_{\alpha}^{i})\in\mathbb R^{n\times d}$,
$$\nu|\xi|^{2}\leq A_{ij}^{\alpha\beta}\xi_{\alpha}^{i}\xi_{\beta}^{j},\quad|A^{\alpha\beta}|\leq\nu^{-1},$$
and are assumed to be piecewise smooth in each subdomain\footnote{Even though we assume the strong ellipticity condition in the proofs below, our main result still holds under a weaker ellipticity condition: $\int_{\cD} A_{ij}^{\alpha\beta}D_\alpha\phi^i D_\beta\phi^j\ge \nu \int_{\cD} |D\phi|^2$ for any $\phi\in H^1_0(\cD)$, and in particular for the linear systems of elasticity. See \cite[pp. (1.4)--(1.6)]{ln03} for details.}.
Throughout this paper, we use the Einstein summation convention over repeated indices. We will denote $A^{\alpha\beta}$ by $A$ for abbreviation.

In the elliptic setting (i.e., the time-independent case), this problem arises from the stress analysis in composite materials consisting of  inclusions (subdomains) embedded in the background medium (the matrix), where the inclusions have material properties different from that of the matrix. The mathematical problem is formulated in terms of a finite number of disjoint bounded subdomains $\cD_{j}$, $j=1,\ldots,M$. The subdomains are assumed to be appropriately smooth. With these conditions, the physical properties of the composite media are described by a divergence form partial differential equations (PDEs)  with coefficients that are smooth in each subdomain but have discontinuities across the interface separating the subdomains. From an  engineering point of view, the most  important quantity is the stress represented by the gradient of the solution to PDEs. There is massive  literature in this direction.  We will only mention the most relevant works in this paper.

A special case  when $\cD_1$ and $\cD_2$ are two touching disks in a bounded domain $\cD\subset\mathbb R^2$ was studied by Bonnetier and Vogelius \cite{bv00}. They considered the scalar equation
\begin{equation}\label{homoscalar}
D_\alpha(a(x)D_\alpha u)=0\quad\mbox{in}~\cD,
\end{equation}
where the coefficient $a(x)$ is given by
\begin{align*}
a(x)=a_0\mathbbm{1}_{\cD_1\cup\cD_2}+\mathbbm{1}_{\cD\setminus(\cD_1\cup\cD_2)},
\end{align*}
with $0<a_0<\infty$ and $\mathbbm{1}_{\bullet}$ is the indicator function. They showed that $|Du|$ is bounded by using a M\"{o}bius  transformation and the maximum principle. The numerical analysis also indicates such result holds for certain elliptic systems; see the work \cite{basl99} for  the
equations of elasticity. The elliptic equation with nonhomogeneous terms
\begin{equation}\label{nonhomo}
D_\alpha(A^{\alpha\beta}D_\beta u)=\Div f\quad\mbox{in}~\cD
\end{equation}
was  first studied by Li and Vogelius in \cite{lv00}. They showed that  any weak solution $u$ to \eqref{nonhomo} is piecewise $C^{1,\delta}$ with $\delta\in\big(0,\frac{\mu}{d(1+\mu)}\big]$, under the assumption that the interface is in $C^{1,\mu}$, the coefficients $A^{\alpha\beta}$ and the data $f$ are piecewise $C^{\delta}$, where $\mu\in(0,1]$. Such  Schauder estimates were improved to $C^{1,\delta}$ with $\delta\in\big(0,\frac{\mu}{2(1+\mu)}\big]$ in \cite{ln03}, and to $C^{1,\delta}$ with $\delta\in\big(0,\frac{\mu}{1+\mu}\big]$ in \cite{dx2019}, where  second-order elliptic system in divergence form was considered. The main feature of \cite{lv00,ln03,dx2019} is that more than two components are allowed to touch and, interestingly, these estimates are independent of the distance between the subdomains.  See also \cite{ckvc86,cf2012,d2012,xb2013,dl2019} and the references therein.
The corresponding results for parabolic equations and systems were studied in \cite{fknn13,ll17}, where the subdomains are assumed to be cylindrical and coefficients satisfy some smooth regularity assumptions, and in \cite{dx2021}, where the subdomains are allowed to be non-cylindrical, and the interfacial boundaries are assumed to be $C^{1,\text{Dini}}$ in the spatial variables and $C^{\gamma_{0}}$ in the time variable with $\gamma_{0}>1/2$.

In \cite[Page 94]{lv00}, Li and Vogelius also studied higher regularity of solutions to \eqref{homoscalar} in the special case when $\cD_1$ and $\cD_2$ are two touching unit disks centered at $(0,-1)$ and $(0,1)$ in $\mathbb R^2$, and $\cD$ is a disk $B_{R_0}:=\{x: |x|<R_0\}$ with $R_0>2$. By using known results, it is easily seen that away from the origin $u$ is smooth in each subdomain up to the interfacial boundaries. The regularity issue near the origin is subtle because of the geometry of the subdomains. In \cite{lv00}, by using conformal mappings it was proved that in the 2D case, for sufficiently large $R_0$, $u$ is piecewise smooth up to interfacial boundaries near the origin. Among other thought-provoking questions raised in \cite{lv00}, the authors asked (1) whether the condition that $R_0$ being sufficiently large can be dropped, (2) whether a uniform estimate holds when the inclusions are close to each other but do not touch, and ultimately (3) whether a similar result holds true for general subdomains in any dimensions.

The first question was answered affirmatively by the first named author and H. Zhang in \cite{dz2016} using an explicit construction of Green's function, which is represented as an infinite series of logarithmic functions composed with conformal mappings. The second question was answered a bit later by the first named author and H. Li in \cite{dl2019} by also using a Green function method. In both papers,
the authors considered more general non­homogeneous equations and obtained Schauder type estimates as well as optimal derivative estimates by showing the explicit dependence of the coefficients and the distance between interfacial boundaries of inclusions. In particular, when $a_0 = 0$ or $\infty$, it was shown that for any positive integer $j$, $|D^j u|$ blows up at the rate $\varepsilon^{-j/2}$ with $\varepsilon$ being the distance between two disks, which agrees with the known results in \cite{BC84, M96} when $j = 1$.
See also a recent interesting paper \cite{jk21} for related results about higher derivative estimates in dimension two with circular inclusions.

In this paper, we address the third question mentioned above. We consider more general divergence form parabolic systems with piecewise $C^{(s+\delta)/2,s+\delta}$ coefficients and data in a  bounded domain consisting of a finite number of cylindrical subdomains with $C^{s+1+\mu}$ interfacial boundaries, where $s\in\mathbb N$, $\delta\in (1/2,1)$, and $\mu\in (0,1]$. We establish  piecewise $C^{(s+1+\mu')/2,s+1+\mu'}$ estimates for weak solutions to such parabolic systems, where $\mu'=\min\big\{\frac 1 2,\mu\big\}$, and the estimates are independent of the distance between the interfaces.  In the time-independent case, the corresponding result for elliptic systems follows.

It is worth mentioning that in the case of two subdomains, the problem is also closely related to the transmission problem. We refer the reader to the work \cite{xb2013, z2021} for results about sharp regularity in various spaces and interior higher-order Schauder estimates for weak solutions to the transmission problem.

\subsection{Main results}
For $\varepsilon>0$ small, we set
$$\cD_{\varepsilon}:=\{x\in \cD: \mbox{dist}(x,\partial \cD)>\varepsilon\}.$$

\begin{assumption}\label{assump}
The domain, coefficients, and data satisfy the following conditions, respectively:
	
\begin{enumerate}
\item The domain $\cD$ contains $M$ disjoint subdomains $\cD_{j},j=1,\ldots,M,$ and the interfacial boundaries are $C^{s+1+\mu}$, where $s\in\mathbb N$ and $\mu\in(0,1]$. We also assume that any point $x\in\cD$ belongs to the boundaries of at most two of the $\cD_{j}$'s.
\item The coefficients $A^{\alpha\beta}$ and the data $f^{\alpha}$ are of class $C^{(s+\delta)/2,s+\delta}((-T+\varepsilon,0)\times (\cD_{\varepsilon}\cap\overline{{\cD}_{j}})),j=1,\ldots,M$, where $\delta\in\big(\frac{1}{2},1\big)$.
\end{enumerate}
\end{assumption}

Here is the main result of the paper.

\begin{theorem}\label{Mainthm}
Let $\cQ=(-T,0)\times\cD$, $\varepsilon\in (0,1)$, $p\in(1,\infty)$,  $A^{\alpha\beta}$ and $f^{\alpha}$ satisfy Assumption \ref{assump}. Let $u\in \mathcal{H}_{p}^{1}(\cQ)$ be a weak solution to \eqref{systems} in $\cQ$. Then $u\in C^{(s+1+\mu')/2,s+1+\mu'}((-T+\varepsilon,0)\times (\cD_{\varepsilon}\cap\overline{{\cD}_{j_0}}))$ and satisfies
\begin{align*}
&|u|_{(s+1+\mu')/2,s+1+\mu';(-T+\varepsilon,0)\times (\cD_{\varepsilon}\cap\overline{{\cD}_{j_0}})}\\
&\leq N\Big(\|Du\|_{L_{1}((-T,0)\times \cD)}+\sum_{j=1}^{M}|f|_{(s+\delta)/2,s+\delta;(-T,0)\times \overline{\cD_{j}}}\Big)
\end{align*}
for any $j_0=1,\ldots,M$, where $\mu'=\min\big\{\frac{1}{2},\mu\big\}$, $N$ depends on $n$, $d$, $M$, $p$, $\nu$, $\varepsilon$, $|A|_{(s+\delta)/2,s+\delta;(-T,0)\times\overline{\cD_{j}}}$, and the $C^{s+1+\mu}$ characteristic of $\cD_{j}$.
\end{theorem}

\begin{remark}
The proof of Theorem \ref{Mainthm} actually gives the following better estimate in the time variable:
\begin{align*}
[D^su]_{t,(1+\delta)/2;(-T+\varepsilon,0)\times (\cD_{\varepsilon}\cap\overline{{\cD}_{j_0}})}\leq N\Big(\|Du\|_{L_{1}((-T,0)\times \cD)}+\sum_{j=1}^{M}|f|_{(s+\delta)/2,s+\delta;(-T,0)\times \overline{\cD_{j}}}\Big),
\end{align*}
where $j_0=1,\ldots,M$. See Lemmas \ref{lemut} and \ref{lemmaDDDu} below.
\end{remark}

\begin{remark}
When $s=0$, piecewise H\"{o}lder-regularity of $Du$ was proved in \cite{dx2021}. When $s=1$, to obtain piecewise regularity of $u$, we  first prove the  regularity of $u_t$ since we need to use the equation \eqref{systems} to solve for $D^{2}u$. For this, we  employ finite difference quotient argument to get
$$u_t\in C^{1/4,1/2}((-T+\varepsilon,0)\times(\cD_\varepsilon\cap\overline{{\cD}_{j}})).$$
See the proof of Lemma \ref{lemut} for the details.
When $s\geq2$, we can differentiate the equation \eqref{systems} with respect to $t$  and obtain
$$u_t\in C^{(s-1+\mu')/2,s-1+\mu'}((-T+\varepsilon,0)\times (\cD_{\varepsilon}\cap\overline{{\cD}_{j}})),~\mu'=\min\big\{\frac{1}{2},\mu\big\}.$$
See the proof of Lemma \ref{lemmaDDDu} for more details.
\end{remark}

\begin{remark}
As explained in \cite[Corollary 2.2]{dx2021},  any $\mathcal{H}_{1}^{1}(\cQ)$ weak solution to \eqref{systems} is in $\mathcal{H}_{p,loc}^{1}(\cQ)$, thus the results in Theorem \ref{Mainthm} hold under the assumption that $u\in\mathcal{H}_{1}^{1}(\cQ)$.
\end{remark}

Now let us end this section with the organization of the paper and an outline of the proof. In Section \ref{preliminary}, we first present the basic notation and definitions. For simplicity, we assume that $\cD$ is a unit ball $B_1$ and take $\cQ$ to be a unit cylinder $Q_1^-:=(-1,0)\times B_1$. Then we represent the interfacial boundaries in $B_1$ by $x^d=h_j(x')$, where $j=1,\ldots,m+1$, and $h_j\in C^{s+1+\mu}(B'_{1})$. Using the functions $h_j$, we define a vector field $\ell^{k,0}=(0,\ldots,0,1,0,\ldots,\ell_d^{k,0})$ near $0$ which is a tangential direction on each interfacial boundary, where $k=1,\ldots,d-1$, the $k$-th component is 1, and $\ell_d^{k,0}$ is defined in \eqref{elld}.  From the vector field $\ell^{k,0}$, we get a unit vector field $\ell^k$ which is orthogonal to each other; see Figure \ref{coordinates}. For the properties of $\ell^k$, we refer the reader to Lemma \ref{lemell} below.
	
We give a complete proof of Theorem \ref{Mainthm} with $s=1$ in Section  \ref{C2delta}. The main idea is to consider the directional derivatives of $u$ along the vector field $\ell^k$ defined in Section \ref{preliminary}. We first derive a new parabolic system in divergence form
$$
-\tilde u_t+D_\alpha(A^{\alpha\beta}D_\beta \tilde u)=D_\alpha \tilde f^\alpha+g,$$
where $\tilde f^\alpha$ and $g$ are defined in \eqref{def-tildefalpha} and \eqref{defg}, respectively, $\tilde u:=u_{\ell}-{\mathfrak u}$, ${\mathfrak u}$ is defined in \eqref{def-v} which is piecewise  smooth,
\begin{align*}
u_{\ell}=D_{\ell}u-\sum_{j=1}^{m+1}\tilde\ell_{i,j}D_iu(t_0,P_jx_0),
\end{align*}
$P_jx_0$ is defined in \eqref{Pjx},  $\tilde\ell_{,j}$ is a smooth extension of $\ell|_{\cD_j}$ to $\cup_{k=1,k\neq j}^{m+1}\cD_k$, and $(t_0,x_0)\in Q_1^-$ is a fixed point.  Then  piecewise H\"{o}lder regularity of $\tilde u$ is derived by adapting Campanato's approach in  \cite{c1963,g1983}, and further developed in \cite{d2012,dk17,dx2021} and the references therein. The key point in employing Campanato's method is to show the mean oscillation of $D\tilde u$ in cylinders vanishes in a certain order as the radii of the  cylinders go to zero. However, we cannot apply this method to $D\tilde u$ directly since $D\tilde u$ is discontinuous across the interfaces. To overcome this difficulty, we first use the weak type-$(1,1)$ estimate for solutions to parabolic systems with coefficients depending only on one direction and a certain decomposition of $\tilde u$, to establish a decay estimate Proposition \ref{lemma iteraphi} of
\begin{equation}\label{funcional}
\inf_{\mathbf q^{k'},\mathbf Q\in\mathbb R^{n}}\left(\fint_{Q_r^-(\Lambda z_0)}\big(|D_{y^{k'}}\tilde{\mathfrak u}-\mathbf q^{k'}|^{\frac{1}{2}}+|\mathcal{A}^{d\beta}D_{y^\beta}\tilde{\mathfrak u}-\tilde {\mathfrak f}^d-\mathbf Q|^{\frac{1}{2}}\big)\,dt\,dy\right)^{2},
\end{equation}
where  $r\in(0,1/4)$, $y=\Lambda x$, $\Lambda=(\Lambda^{\alpha\beta})_{\alpha,\beta=1}^{d}$ is a $d\times d$ orthogonal matrix representing the linear transformation from the coordinate system associated with $0$ to the coordinate system associated with $x_0$, and $$
\tilde {\mathfrak u}(t,y)=\tilde u(t,x), \quad \mathcal{A}^{\alpha\beta}(t,y)=\Lambda^{\alpha k}A^{ks}(t,x)\Lambda^{s\beta}, \quad \tilde {\mathfrak f}^\alpha(t,y)=\Lambda^{\alpha k}\tilde f^k(t,x).
$$
Then with the help of  the decay estimate of the functional in \eqref{funcional} together with the estimates of $|D_{\ell^{k'}}\tilde u-D_{y^{k'}}\tilde{\mathfrak u}|$ and $|\tilde U-\mathcal{A}^{d\beta}D_{y^\beta}\tilde{\mathfrak u}+\tilde {\mathfrak f}^d|$ in \eqref{difference-coor} below, in Lemma \ref{lemma itera} we obtain the decay estimate  of
$$
\inf_{\mathbf q^{k'},\mathbf Q\in\mathbb R^{n}}\left(\fint_{Q_r^-(z_0)}\big(|D_{\ell^{k'}}\tilde u-\mathbf q^{k}|^{\frac{1}{2}}+|\tilde U-\mathbf Q|^{\frac{1}{2}}\big)\,dz\right)^{2},$$
where $\tilde U:=n^\alpha(A^{\alpha\beta}D_\beta \tilde u-\tilde f^\alpha)$ and $n^\alpha$ is defined in \eqref{defnorm} below. The desired result is then proved  by utilizing the decay estimate.

In Section \ref{secgeneral}, we deal with the case when $s\ge 2$ by using a similar scheme. To this end, we consider the following parabolic system in divergence form
$$
-\breve u_t+D_\alpha(A^{\alpha\beta}D_\beta \breve u)=D_\alpha \breve f^\alpha+\breve g,$$
where $\breve u:=u^{\ell}-{\mathsf u}$, ${\mathsf u}$ is defined in \eqref{defsfu} which is piecewise  smooth,
\begin{align*}
u^{\ell}=D_{\ell}D_{\ell}\cdots D_{\ell}u-u_0,
\end{align*}
$u_0$ is defined in \eqref{defu-0}, $\breve f^\alpha$ and $\breve g$ are defined in \eqref{def-brevefalpha} and \eqref{defg2}, respectively.  Then via a similar argument used in Section \ref{C2delta}, we establish a decay estimate of
$$
\inf_{\mathbf q^{k},\mathbf Q\in\mathbb R^{n}}\left(\fint_{Q_r^-(z_0)}\big(|D_{\ell^{k}}\breve u-\mathbf q^{k}|^{\frac{1}{2}}+|\breve U-\mathbf Q|^{\frac{1}{2}}\big)\,dz\right)^{2},$$
where  $\breve U:=n^\alpha(A^{\alpha\beta}D_\beta \breve u-\breve f^\alpha)$. With this estimate,   Theorem \ref{Mainthm} follows.

In the Appendix \ref{Append}, we present the local boundedness and piecewise H\"{o}lder-regularity of $Du$ for the solution of \eqref{systems},  and some auxiliary estimates proved in \cite{dx2021}, which are used frequently in the current paper.

\section{Preliminaries}\label{preliminary}
In this section, we first present the basic notation and  the function spaces. Then we provide the assumptions of the subdomains by assuming $\cD$ is a unit ball $B_1$ for simplicity and introduce vector fields near the origin $0$ together with some properties of the vector fields.

\subsection{Notation and definitions}
For $r>0$ and $z=(t,x)\in\mathbb R^{d+1}$, $d\geq2$, we denote
$$Q_{r}^{-}(t,x):=(t-r^2,t)\times B_{r}(x),$$
where $B_{r}(x):=\{y\in\mathbb R^{d}: |y-x|<r\}$.  For $x\in\mathbb R^{d}$, we denote $x'=(x^1,\ldots,x^{d-1})\in\mathbb R^{d-1}$ and $B'_{r}(x'):=\{y'\in\mathbb R^{d-1}: |y'-x'|<r\}$. We often write $B_r$ and $B'_{r}$ for $B_{r}(0)$ and $B'_{r}(0')$, respectively.
The  parabolic distance between two points $z_1=(t_1,x_1)$ and $z_2=(t_2,x_2)$ is defined by
\begin{equation*}
|z_1-z_2|_{p}:=\max\left\{|t_1-t_2|^{1/2},|x_1-x_2|\right\}.
\end{equation*}
We denote the parabolic boundary of a cylinder $\cQ=(a,b)\times\cD$ by
$$\partial_{p}\cQ=((a,b)\times\partial\cD)\cup(\{a\}\times\overline \cD).$$
For a function $f$ defined in $\mathbb R^{d+1}$, we set
$$(f)_{\cQ}=\frac{1}{|\cQ|}\int_{\cQ}f(t,x)\ dx\,dt=\fint_{\cQ}f(t,x)\ dx\,dt,$$
where $|\cQ|$ is the $d+1$-dimensional Lebesgue measure of $\cQ$.

Next, for $\gamma\in(0,1]$, we denote the $C^{\gamma/2,\gamma}$ semi-norm by
$$[u]_{\gamma/2,\gamma;\cQ}:=\sup_{\begin{subarray}{1}(t,x),(s,y)\in\cQ\\
	(t,x)\neq (s,y)
	\end{subarray}}\frac{|u(t,x)-u(s,y)|}{|t-s|^{\gamma/2}+|x-y|^{\gamma}},$$
and the $C^{\gamma/2,\gamma}$ norm by
$$|u|_{\gamma/2,\gamma;\cQ}:=[u]_{\gamma/2,\gamma;\cQ}+|u|_{0;\cQ},\quad \text{where}\,\,|u|_{0;\cQ}=\sup_{\cQ}|u|.$$
Define
\begin{align*}
[u]_{t,\gamma;\cQ}&:=\sup_{\begin{subarray}{1}(t,x),(s,x)\in\cQ\\
	\quad t\neq s
	\end{subarray}}\frac{|u(t,x)-u(s,x)|}{|t-s|^{\gamma}},\\ [u]_{(1+\gamma)/2,1+\gamma;\cQ}&:=[Du]_{\gamma/2,\gamma;\cQ}+[u]_{t,(1+\gamma)/2;\cQ},
\end{align*}
and
$$|u|_{(1+\gamma)/2,1+\gamma;\cQ}:=[u]_{(1+\gamma)/2,1+\gamma;\cQ}+|u|_{0;\cQ}.$$
We denote $C^{(1+\gamma)/2,1+\gamma}$  to be the set of all bounded measurable functions $u$ for which $Du$ are bounded and continuous in $\cQ$ and $[u]_{(1+\gamma)/2,1+\gamma;\cQ}<\infty$. The function spaces $C^{(l+\gamma)/2,l+\gamma},l=2,3,\ldots,$ are defined
accordingly.

For $p\in(1,\infty)$, we denote
$$W_{p}^{1,2}(\cQ):=\{u: u, u_t, Du, D^2u\in L_{p}(\cQ)\}.$$
Set
$$\mathbb{H}_{p}^{-1}(\cQ):=\bigg\{u: u=\sum_{|\alpha|\leq1}D^{\alpha}u_{\alpha}, u_{\alpha}\in L_{p}(\cQ)\bigg\}$$
and
$$\|u\|_{\mathbb{H}_{p}^{-1}(\cQ)}:=\inf\bigg\{\sum_{|\alpha|\leq1}\|u_{\alpha}\|_{L_{p}(\cQ)}: u=\sum_{|\alpha|\leq1}D^{\alpha}u_{\alpha}\bigg\}.$$
The  solution space $\mathcal{H}_{p}^{1}(\cQ)$ is defined by
$$\mathcal{H}_{p}^{1}(\cQ):=\{u: u_{t}\in \mathbb{H}_{p}^{-1}(\cQ), D^{\alpha}u\in L_{p}(\cQ), 0\leq|\alpha|\leq1\},$$
$$\|u\|_{\mathcal{H}_{p}^{1}(\cQ)}:=\|u_{t}\|_{\mathbb{H}_{p}^{-1}(\cQ)}+\sum_{|\alpha|\leq1}\|D^{\alpha}u\|_{L_{p}(\cQ)}.$$
Denote $C_{0}^{\infty}([-1,0]\times\cD)$ to be  the collection of infinitely differentiable functions $u=u(t,x)$ with compact supports in $[-1,0]\times\cD$.
Set $\mathcal{\mathring{H}}_{p}^{1}((-1,0)\times\cD)$ to be the closure of $C_{0}^{\infty}([-1,0]\times\cD)$ in $\mathcal{H}_{p}^{1}((-1,0)\times\cD)$.

\subsection{Assumptions and auxiliary results}\label{Assumption}
Our objective is to prove the interior regularity of solutions by  establishing  the local  estimates. We will slightly abuse the notation to localize the problem by taking $\cQ$ to be a unit cylinder $Q_1^-:=(-1,0)\times B_1$. By suitable rotation and scaling, we may suppose that a finite number of subdomains lie in $B_{1}$ and that they can be represented by
$$x^{d}=h_{j}(x'),\quad\forall~x'\in B'_{1},~j=1,\ldots,m<M,$$
where
\begin{equation}\label{eqhj}
-1<h_{1}(x')<\dots<h_{m}(x')<1,
\end{equation}
$h_{j}(x')\in C^{s+1+\mu}(B'_{1})$ with $s\in\mathbb N$. Set $h_{0}(x')=-1$ and $h_{m+1}(x')=1$. Then we have $m+1$ regions:
$$\cD_{j}:=\{x\in \cD: h_{j-1}(x')<x^{d}<h_{j}(x')\},\quad1\leq j\leq m+1.$$

For $j=1,\ldots,m$, the normal direction of the interfacial boundary $\Gamma_j=\{x^d=h_j(x')\}$ is given by
\begin{equation}\label{normal}
n_j:=(n_j^1,\ldots,n_j^d)=(1+|D_{x'}h_j(x')|^2)^{-1/2}(-D_{x'}h_j(x'),1)^\top\in \bR^{d}.
\end{equation}
For each $k=1,\ldots,d-1$, we define a vector field $\ell^{k,0}:\bR^{d}\to \bR^d$ near $0$ as follows:
$$
\ell^{k,0}_k=1,\quad \ell^{k,0}_i=0,\,\,i\neq k,d,
$$
\begin{equation}\label{elld}
\ell^{k,0}_d=\left\{
\begin{aligned}
D_kh_m(x'),\quad&x^d\ge h_m,\\
\frac {x^d-h_{j-1}}{h_{j}-h_{j-1}}D_kh_j(x')+\frac {h_{j}-x^d}{h_{j}-h_{j-1}}
D_kh_{j-1}(x'),\quad&h_{j-1}\le x^d< h_j,\\
D_kh_1(x'),\quad&x^d< h_1.
\end{aligned}
\right.
\end{equation}
It is easily seen that  on $\Gamma_j$, $\ell_d^{k,0}=D_kh_{j}(x')$
so that $\ell^{k,0}$ is in a tangential direction. Moreover, it follows from $h_j\in C^{s+1+\mu}$ that $\ell^{k,0}$ is $C^{s+\mu}$ on $\Gamma_j$.
Define the projection operator by
$$\mbox{proj}_{a}b=\frac{\langle a,b\rangle}{\langle a,a\rangle}a,$$
where $\langle a,b\rangle$ denotes the inner product of the vectors $a$ and $b$, and $\langle a,a\rangle=|a|^{2}$. Then we  make the vector field orthogonal to each other by using the Gram-Schmidt process:
\begin{equation}\label{defell}
\begin{split}
\tilde\ell^{1}&=\ell^{1,0},
\quad\ell^1=\frac{\tilde\ell^{1}}{|\tilde\ell^{1}|},\\
\tilde\ell^{2}&=\ell^{2,0}
-\mbox{proj}_{\ell^{1}}\ell^{2,0},
\quad\ell^2=\frac{\tilde\ell^{2}}{|\tilde\ell^{2}|},\\
&\vdots\\
\tilde\ell^{d-1}&=\ell^{d-1,0}-\sum_{j=1}^{d-2}
\mbox{proj}_{\ell^{j}}\ell^{d-1,0},\quad\ell^{d-1}=
\frac{\tilde\ell^{d-1}}{|\tilde\ell^{d-1}|}.
\end{split}
\end{equation}
From the definition of $\ell^{k,0}$, we define the corresponding unit normal direction which is orthogonal to $\ell^{k,0}$, $k=1,\ldots,d-1,$ (and thus also $\ell^k$):
\begin{equation}\label{defnorm}
n(x)=(n^1,\ldots,n^d)^{\top}=\big(1+\sum_{k=1}^{d-1}(\ell_d^{k,0})^2\big)^{-1/2}(-\ell_d^{1,0},\ldots,-\ell_d^{d-1,0},1)^{\top}.
\end{equation}
Obviously, $n(x)=n_j$ on $\Gamma_j$.

\begin{lemma}\label{lemell}
Let $\ell^k$ be defined in \eqref{defell}, $k=1,\ldots,d-1$. Then  the following assertions hold.

(i) We have $\ell^{k}\in C^{1/2}(\cD)$ and the $C^{1/2}$ norms are independent of the distance between the subdomains.

(ii) It holds that $D_{\ell^{k_1}}D_{\ell^{k_2}}\cdots D_{\ell^{k_{s-1}}}\ell^{k_s}\in C^0(\cD)$, where  $k_{\tau}=1,\ldots,d-1$, $\tau=1,\ldots,s$, and $s\in\mathbb N$ with $s\geq2$. Moreover, the $C^0$ norms are independent of the distance between the subdomains.

(iii) We have $|D\ell^k|\leq N|h_j-h_{j-1}|^{-1/2}$ and $|DD_{\ell^{k_1}}D_{\ell^{k_2}}\cdots D_{\ell^{k_{s-1}}}\ell^{k_s}|\leq N|h_j-h_{j-1}|^{-1}$, where $N$ depends on the $C^{s+1}$ norms of $h_j$.
\end{lemma}

\begin{proof}
(i) We start with proving that $\ell_d^{k,0}$ is $C^{1/2}$ in the vertical direction $x^d$. For any two points $(x',x_1^d), (x',x_2^d)$ satisfying $h_{j-1}(x')\le x_i^d< h_j(x')$, $i=1,2$, we have
\begin{align*}
\ell^{k,0}_d(x',x_1^d)-\ell^{k,0}_d(x',x_2^d)=\frac{x_1^d-x_2^d}{h_j-h_{j-1}}D_{k}(h_j-h_{j-1}).
\end{align*}
It follows from $h_{j}\in C^{s+1+\mu}$ and $h_j>h_{j-1}$ that
\begin{equation}\label{Diff_Dh}
|D_kh_{j}(x')-D_{k}h_{j-1}(x')|\leq N|h_{j}(x')-h_{j-1}(x')|^{1/2}.
\end{equation}
See, for instance, \cite[(50)]{lv00}. This together with $|x_1^d-x_2^d|\leq h_j-h_{j-1}$ gives
\begin{align*}
|\ell^{k,0}_d(x',x_1^d)-\ell^{k,0}_d(x',x_2^d)|\leq N|x_1^d-x_2^d|^{1/2}.
\end{align*}
We continue to prove that $\ell_d^{k,0}$ is $C^{1/2}$ in $x'$. For any two points $(x'_1,x^d)$ and $(x'_2,x^d)$ with $h_{j-1}(x'_i)\leq x^d<h_j(x'_i)$, $i=1,2$,  we have
\begin{align}\label{diffeelldk0}
&\ell_d^{k,0}(x'_1,x^d)-\ell_d^{k,0}(x'_2,x^d)\nonumber\\
&=\frac {x^d-h_{j-1}(x'_1)}{(h_{j}-h_{j-1})(x'_1)}D_k(h_j-h_{j-1})(x'_1)+D_kh_{j-1}(x'_1)-\frac {x^d-h_{j-1}(x'_2)}{(h_{j}-h_{j-1})(x'_2)}D_k(h_j-h_{j-1})(x'_2)\nonumber\\
&\quad-D_kh_{j-1}(x'_2)\nonumber\\
&=D_kh_{j-1}(x'_1)-D_kh_{j-1}(x'_2)+\frac {x^d-h_{j-1}(x'_1)}{(h_{j}-h_{j-1})(x'_1)}\big(D_k(h_j-h_{j-1})(x'_1)-D_k(h_j-h_{j-1})(x'_2)\big)\nonumber\\
&\quad+D_k(h_j-h_{j-1})(x'_2)\left(\frac {x^d-h_{j-1}(x'_1)}{(h_{j}-h_{j-1})(x'_1)}-\frac {x^d-h_{j-1}(x'_2)}{(h_{j}-h_{j-1})(x'_2)}\right).
\end{align}
Now we estimate the last term in \eqref{diffeelldk0}. Without loss of generality, we assume that
\begin{equation}\label{hjhj-1}
|(h_j-h_{j-1})(x'_1)|=\sup_{r=0,1}|(h_j-h_{j-1})(rx'_1+(1-r)x'_2)|.
\end{equation}
If $|x'_1-x'_2|>|(h_j-h_{j-1})(x'_1)|$, then by using $h_{j-1}(x'_i)\leq x^d<h_j(x'_i)$ and \eqref{Diff_Dh}, $i=1,2$, we have
\begin{align*}
&\left|D_k(h_j-h_{j-1})(x'_2)\left(\frac {x^d-h_{j-1}(x'_1)}{(h_{j}-h_{j-1})(x'_1)}-\frac {x^d-h_{j-1}(x'_2)}{(h_{j}-h_{j-1})(x'_2)}\right)\right|\\
&\leq 2|D_k(h_j-h_{j-1})(x'_2)|\leq N|(h_j-h_{j-1})(x'_2)|^{1/2}\leq N|x'_1-x'_2|^{1/2}.
\end{align*}
If $|x'_1-x'_2|\leq|(h_j-h_{j-1})(x'_1)|$, then by \eqref{Diff_Dh} and \eqref{hjhj-1}, we derive
\begin{align}\label{lastterm2}
|D_k(h_j-h_{j-1})(x'_2)|\leq N|(h_j-h_{j-1})(x'_2)|^{1/2}\leq N|(h_j-h_{j-1})(x'_1)|^{1/2}.
\end{align}
Using $h_j\in C^{s+1+\mu}$, we obtain
\begin{align}\label{lastterm}
&\left|\frac {x^d-h_{j-1}(x'_1)}{(h_{j}-h_{j-1})(x'_1)}-\frac {x^d-h_{j-1}(x'_2)}{(h_{j}-h_{j-1})(x'_2)}\right|\notag\\
&=\Bigg|\frac {h_{j-1}(x'_2)-h_{j-1}(x'_1)}{(h_{j}-h_{j-1})(x'_1)}
+\frac{(x^d-h_{j-1}(x'_2))\big((h_{j}-h_{j-1})(x'_2)-(h_{j}-h_{j-1})(x'_1)\big)}{(h_{j}-h_{j-1})(x'_1)
\cdot(h_{j}-h_{j-1})(x'_2)}\Bigg|\nonumber\\
&\leq \frac{N|x'_1-x'_2|}{(h_{j}-h_{j-1})(x'_1)}.
\end{align}
Combining  \eqref{lastterm2} and \eqref{lastterm}, we deduce
\begin{align*}
&\left|D_k(h_j-h_{j-1})(x'_2)\left(\frac {x^d-h_{j-1}(x'_1)}{(h_{j}-h_{j-1})(x'_1)}-\frac {x^d-h_{j-1}(x'_2)}{(h_{j}-h_{j-1})(x'_2)}\right)\right|\\
&\leq \frac{N|x'_1-x'_2|}{|(h_j-h_{j-1})(x'_1)|^{1/2}}\leq N|x'_1-x'_2|^{1/2}.
\end{align*}
Therefore, coming back to \eqref{diffeelldk0}, we have
$$|\ell_d^{k,0}(x'_1,x^d)-\ell_d^{k,0}(x'_2,x^d)|\leq N|x'_1-x'_2|^{1/2}.$$
We hence conclude that $\ell_d^{k,0}$ is $C^{1/2}$ in $\cD$. Combining with the definition of $\ell^k$ in \eqref{defell},  we derive the $C^{1/2}$-regularity of $\ell^k$.
	
(ii) In view of  \eqref{elld}, a  direct calculation gives
\begin{align}\label{Dk'ell}
D_{k_1}\ell_d^{k_2,0}&=\frac{D_{k_1}h_{j-1}D_{k_2}(h_{j-1}-h_j)}{h_j-h_{j-1}}-\frac{(x^d-h_{j-1})D_{k_2}(h_j-h_{j-1})D_{k_1}(h_j-h_{j-1})}{(h_j-h_{j-1})^2}\nonumber\\
&\quad +\frac {x^d-h_{j-1}}{h_{j}-h_{j-1}}D_{k_1}D_{k_2}h_j+\frac {h_{j}-x^d}{h_{j}-h_{j-1}}D_{k_1}D_{k_2}h_{j-1},
\end{align}
and
\begin{equation}\label{Ddell}
D_d\ell_d^{k_2,0}=\frac{D_{k_2}(h_j-h_{j-1})}{h_j-h_{j-1}},
\end{equation}
where $h_{j-1}\le x^d< h_j$. Since
$$\ell_d^{k_1,0}=\frac{(x^d-h_{j-1})D_{k_1}(h_j-h_{j-1})}{h_j-h_{j-1}}+D_{k_1}h_{j-1},\quad h_{j-1}\le x^d< h_j,$$
we have
\begin{align*}
D_{k_1}\ell_d^{k_2,0}+\ell_d^{k_1,0}D_d\ell_d^{k_2,0}=\frac {x^d-h_{j-1}}{h_{j}-h_{j-1}}D_{k_1}D_{k_2}h_j+\frac {h_{j}-x^d}{h_{j}-h_{j-1}}D_{k_1}D_{k_2}h_{j-1},
\end{align*}
when $h_{j-1}\le x^d< h_j$.
Therefore, we obtain
\begin{align*}
D_{\ell^{k_1,0}}\ell_d^{k_2,0}&=D_{k_1}\ell_d^{k_2,0}+\ell_d^{k_1,0}D_d\ell_d^{k_2,0}\nonumber\\
&=
\begin{cases}
D_{k_1}D_{k_2}h_m(x'),&\quad x^d\ge h_m,\\
\frac {x^d-h_{j-1}}{h_{j}-h_{j-1}}D_{k_1}D_{k_2}h_j(x')+\frac {h_{j}-x^d}{h_{j}-h_{j-1}}
D_{k_1}D_{k_2}h_{j-1}(x'),&\quad h_{j-1}\le x^d< h_j,\\
D_{k_1}D_{k_2}h_1(x'),&\quad x^d< h_1.
\end{cases}
\end{align*}
This together with \eqref{eqhj} implies that $D_{\ell^{k_1,0}}\ell_d^{k_2,0}\in C^0(\cD)$ and the $C^0$ norm is independent of the distance between the subdomains.
From this, and using  the definition of $\ell^{k}$ in \eqref{defell} and the fact that $\tilde \ell^{k}$ is a linear combination of $\ell^{j,0}$, $j=1,\ldots,k$, we also deduce that
\begin{equation*}
D_{\ell^{k_1}}\ell^{k_2}\in C^0(\cD).
\end{equation*}
By an induction argument, we conclude
\begin{align}\label{Dellelldkk}
&D_{\ell^{k_1,0}}D_{\ell^{k_2,0}}\cdots D_{\ell^{k_{s-1},0}}\ell_d^{k_s,0}\nonumber\\
&=
\begin{cases}
D_{k_1}D_{k_2}\cdots D_{k_s}h_m(x'),&\,\, x^d\ge h_m,\\
\frac {x^d-h_{j-1}}{h_{j}-h_{j-1}}D_{k_1}D_{k_2}\cdots D_{k_s}h_j(x')+\frac {h_{j}-x^d}{h_{j}-h_{j-1}}
D_{k_1}D_{k_2}\cdots D_{k_s}h_{j-1}(x'),&\,\, h_{j-1}\le x^d< h_j,\\
D_{k_1}D_{k_2}\cdots D_{k_s}h_1(x'),&\,\, x^d< h_1,
\end{cases}
\end{align}
and thus, similarly, $D_{\ell^{k_1}}D_{\ell^{k_2}}\cdots D_{\ell^{k_{s-1}}}\ell^{k_s}\in C^0(\cD)$.

(iii) By virtue of \eqref{Dk'ell}, \eqref{Ddell},  and \eqref{Diff_Dh}, we have  in $B_{r}(x_0)\cap\cD_{j}$,
\begin{equation*}
|D\ell_d^{k,0}|\leq N|h_j-h_{j-1}|^{-1/2}.
\end{equation*}
Then recalling the definition of $\ell^k$ in \eqref{defell}, this also holds for $\ell^k$. Similarly, from \eqref{Dellelldkk}, we have
\begin{equation*}
|DD_{\ell^{k_1,0}}D_{\ell^{k_2,0}}\cdots D_{\ell^{k_{s-1},0}}\ell_d^{k_s,0}|\leq N|h_j-h_{j-1}|^{-1},
\end{equation*}
and thus
\begin{equation*}
|DD_{\ell^{k_1}}D_{\ell^{k_2}}\cdots D_{\ell^{k_{s-1}}}\ell^{k_s}|\leq N|h_j-h_{j-1}|^{-1}.
\end{equation*}
The lemma is proved.
\end{proof}

\subsection{Coordinate systems.}
        \label{sec 2.3}
In the proofs, we will use different coordinate systems associated with different points defined as follows.

We fix a coordinate system such that $0\in\cD_{i_0}$ for some $i_0\in\{1,\ldots,m+1\}$ and the closest point on $\partial\cD_{i_0}$ is $x_{i_0}=(0',h_{i_0}(0'))$, and $\nabla_{x'}h_{i_0}(0')=0'$. Throughout the paper, we shall use $x=(x',x^d)$ and $D_{x}$ to denote the point and the derivatives, respectively, in this coordinate system. Then at the point $x_{i_0}$, $\ell^{k}=(0,\ldots,0,1,0,\ldots,0)^{\top}$ and $n=(0',1)^{\top}$. See Figure \ref{coordinates}.

For any point $x_0\in B_{3/4}\cap \cD_{j_{0}}$, $j_0=1,\ldots,m+1$, suppose  the closest point on $\partial \cD_{j_{0}}$ to $x_0$ is $y_0:=(y'_0,h_{j_{0}}(y'_0))$.  Let
\begin{equation}\label{defny0}
n_{y_0}=(n^1_{y_0},\ldots, n_{y_0}^d)^{\top}=\big(1+|\nabla_{x'}h_{j_0}(y'_0)|^{2})^{-1/2}
\big(-\nabla_{x'}h_{j_0}(y'_0),1\big)^{\top}
\end{equation}
be the unit normal vector at $(y'_0,h_{j_0}(y'_0))$ on the surface $\Gamma_{j_0}$. Define the tangential vectors by
\begin{equation}\label{deftauk}
\tau_k=\ell^{k}(y_0),\quad k=1,\ldots,d-1,
\end{equation}
where $\ell^{k}$ is defined in \eqref{defell}. See Figure \ref{coordinates}.

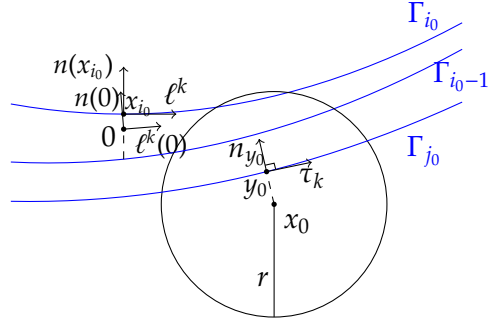
\begin{figure}
\begin{tikzpicture}
\fill (0,0.5) circle (1pt) node at (-0.2,0.4) {$0$};
\draw [->] (0,0.5) -- (-0.04,1) node at (-0.35,0.9) {{\small$n(0)$}};
\draw [->] (0,0.5) -- (0.5,0.54) node at (0.5,0.3) {{\small$\ell^k(0)$}};
\draw[blue,domain=-1.5:4.5] plot(\x,0.05*\x*\x+0.1*\x+0.1) node at (4.5,1.2){$\Gamma_{i_0-1}$};
\draw[blue,domain=-1.5:4.5] plot(\x,0.06*\x*\x+0.7) node at (4,2){$\Gamma_{i_0}$};
\fill (0,0.7) circle (1pt) node at (0.2,0.8) {{\small$x_{i_0}$}};
\draw [dashed]  (0,0.1) -- (0,0.5);
\draw [->] (0,0.7) -- (0,1.33) node[left] {{\small$n(x_{i_0})$}};
\draw [->] (0,0.7) -- (0.7,0.7) node[above] {$\ell^k$};
\fill (2,-0.5) circle (1pt) node[below right] {$x_0$};
\draw (2,-0.5) circle (1.5);
\draw [-] (2,-0.5) -- (2,-2) node at (1.85,-1.5){$r$};
\draw[blue,domain=-1.5:4.5] plot(\x,0.04*\x*\x+0.1*\x-0.4) node at (4,0.2){$\Gamma_{j_0}$};
\fill (1.9,-0.0656) circle (1pt) node at (1.75,-0.3) {{\small$y_0$}};
\draw [dashed] (2,-0.5) -- (1.9,-0.0656);
\draw [->] (1.9,-0.0656) -- (2.5,0.06) node at (2.5,-0.15) {$\tau_k$};
\draw [->] (1.9,-0.0656) -- (1.8,0.3688) node at (1.65,0.16) {$n_{y_0}$};
\draw [-] (1.88,0.02) -- (2,0.047) -- (2.02,-0.03);
\end{tikzpicture}
\vspace*{0.1mm}
\caption{Illustration of two coordinate systems}
\label{coordinates}
\end{figure}

In the coordinate system associated with $x_0$ with the axes parallelled to $n_{y_0}$ and $\tau_k,k=1,\ldots,d-1$, we will use $y=(y',y^d)$ and $D_{y}$ to denote the point and the derivatives, respectively. Moreover, we have $y=\Lambda x$, where $\Lambda=(\Lambda^1,\ldots,\Lambda^d)^\top=(\Lambda^{\alpha\beta})_{\alpha,\beta=1}^{d}$ is a $d\times d$ matrix representing the linear transformation from the coordinate system associated with $0$ to the coordinate system associated with $x_0$, and $\tau_k=(\Gamma^{1k},\dots,\Gamma^{dk})^\top,k=1,\ldots,d-1$, $n_{y_0}=(\Gamma^{1d},\dots,\Gamma^{dd})^\top$, where $\Gamma=\Lambda^{-1}$.
Now we introduce  $m+1$ ``strips" (in the $y$-coordinates)
$$\Omega_j:=\{y\in\cD: y_{j-1}^d<y^{d}<y_j^d\},\quad j=1,\ldots,m+1,$$
where $y_j:=(\Lambda'y_0,y_j^d)\in \Gamma_j$ and  $\Lambda'=(\Lambda^1,\ldots,\Lambda^{d-1})^\top$.
We also have for any $0<r\leq1/4$,
\begin{equation}\label{volume}
|(\cD_{j}\setminus\Omega_{j})\cap (B_{r}(\Lambda x_0))|\leq Nr^{d+1/2},\quad j=1,\ldots,m+1.
\end{equation}
See, for instance, \cite[Lemma 2.3]{dx2019}.

For a piecewise H\"{o}lder continuous function $f\in C^{\delta/2,\delta}((-1+\varepsilon,0)\times\overline{{\cD}_{j}})$, we define
\begin{align}\label{piecewisecont}
\bar{f}(t,y)=\fint_{Q_r^-(\Lambda z_0)\cap((-1+\varepsilon,0)\times\cD_j)}f(s,z)\, dz\,ds,
\,\, (t,y)\in Q_r^-(\Lambda z_0)\cap((-1+\varepsilon,0)\times\Omega_{j})
\end{align}
to be the corresponding piecewise constant function in $Q_{r}^-(\Lambda z_0):=(t_0-r^2/4,t_0)\times B_{r}(\Lambda x_0)$.
Note that $\bar f$ only depends on $y^d$ and we will denote $\bar f(y^d):=\bar f(t,y)$ for abbreviation. Finally, by using \eqref{volume}, we have
\begin{align}\label{estAbarf}
\|f-\bar f\|_{L_1(Q_{r}^-(\Lambda z_0))}&
\leq\sum_{j=1}^{m+1}\|f-\bar f\|_{L_1(Q_{r}^-(\Lambda z_0)\cap(-1+\varepsilon,0)\times(\cD_j\cap\Omega_j))}\nonumber\\
&\quad+\sum_{j=1}^{m+1}\|f-\bar f\|_{L_1(Q_{r}^-(\Lambda z_0)\cap(-1+\varepsilon,0)\times(\cD_j\setminus\Omega_j))}\leq Nr^{d+5/2}.
\end{align}

\section{Second derivative estimates}\label{C2delta}
We present a complete proof of Theorem \ref{Mainthm} with $s=1$ in this section. As in Section \ref{Assumption}, we will take $\cQ=(-1,0)\times B_1$. The equation \eqref{systems} is equivalent to a (homogeneous) transmission problem
\begin{align*}
\begin{cases}
-u_t+D_\alpha(A^{\alpha\beta}D_\beta u)=D_\alpha f^\alpha \qquad \text{in}\,\,\bigcup_{j=1}^{m+1}(-1,0)\times\cD_j, \\
u|_{(-1,0)\times\Gamma_j}^+=u|_{(-1,0)\times\Gamma_j}^-,\quad[n^\alpha_j(A^{\alpha\beta} D_\beta u -f^\alpha)]_{(-1,0)\times\Gamma_j}=0,\quad j=1,\ldots,m,
\end{cases}
\end{align*}
where
$$
[n_j^\alpha(A^{\alpha\beta} D_\beta u -f^\alpha)]_{(-1,0)\times\Gamma_j}:=n_j^\alpha(A^{\alpha\beta} D_\beta u -f^\alpha)|_{(-1,0)\times\Gamma_j}^+-n_j^\alpha(A^{\alpha\beta} D_\beta u -f^\alpha)|_{(-1,0)\times\Gamma_j}^-,
$$
$n_j$ is the unit normal vector on $\Gamma_j$ defined in \eqref{normal}, $u|_{(-1,0)\times\Gamma_j}^+$ and $u|_{(-1,0)\times\Gamma_j}^-$ ($(A^{\alpha\beta} D_\beta u -f^\alpha)|_{(-1,0)\times\Gamma_j}^+$ and $(A^{\alpha\beta} D_\beta u -f^\alpha)|_{(-1,0)\times\Gamma_j}^-$) are the left and right limits of $u$ (its conormal derivatives) on $(-1,0)\times\Gamma_j$, respectively, $j=1,\ldots,m$.

To show higher regularity, we take the directional derivative of $u$ in the direction $\ell:=\ell^k$,  $k=1,\ldots,d-1$, to get
\begin{align}\label{eq6.26}
\begin{cases}
-(D_\ell u)_t+D_\alpha(A^{\alpha\beta}D_\beta D_\ell u)=g+ D_\alpha f_1^\alpha \quad\text{in}\,\,\bigcup_{j=1}^{m+1}(-1,0)\times\cD_j, \\
D_\ell u|_{(-1,0)\times\Gamma_j}^+=D_\ell u|_{(-1,0)\times\Gamma_j}^-,\quad[n_j^\alpha (A^{\alpha\beta} D_\beta D_\ell u-f_1^\alpha)]_{(-1,0)\times\Gamma_j}=\tilde h_j,
\end{cases}
\end{align}
where
\begin{equation}\label{defg}
\begin{split}
g&=D_\alpha\ell(A^{\alpha\beta} D_\beta Du+DA^{\alpha\beta}D_\beta u-Df^\alpha),\\
f_1^\alpha&=D_\ell f^\alpha+A^{\alpha\beta}D_\beta \ell_i D_iu-D_{\ell} A^{\alpha\beta}D_\beta u,
\end{split}
\end{equation}
and
\begin{align}\label{deftildeh}
\tilde h_j=[D_\ell n_j^\alpha (-A^{\alpha\beta}D_\beta u+f^\alpha)]_{(-1,0)\times\Gamma_j}.
\end{align}
Note that $D_\ell n_j$ is a tangential direction on $\Gamma_j$ and we may write $\tilde h_j=\tilde h_j(t,x')$. Furthermore, by a direct calculation using \eqref{normal} and \eqref{defell},  we have $D_{\ell}n_j^\alpha\in C^{\mu}$ as a function of $x'$.

The equation \eqref{eq6.26} is also a transmission problem for $D_\ell u$, but is inhomogeneous. A difficulty is that $D\ell_d$ is singular at any point where two interfacial boundaries touch or are very close to each other. To cancel out the singularity, we consider
\begin{equation}\label{tildeu}
u_{_\ell}:=u_{_\ell}(z;z_0)=D_{\ell}u-u_0,
\end{equation}
where $z_0=(t_0,x_{0})\in (-9/16,0)\times (B_{3/4}\cap \overline{\cD_{j_{0}}})$,
\begin{align}\label{defu0}
u_0:=u_0(x;z_0)=\sum_{j=1}^{m+1}\tilde\ell_{i,j}D_iu(t_0,P_jx_0),
\end{align}
\begin{align}\label{Pjx}
P_jx_0=\begin{cases}
x_0&\quad\mbox{for}\quad x_0\in\overline{\cD_{j_0}},\\
(x'_0,h_j(x'_0))&\quad\mbox{for}\quad j<j_0,\\
(x'_0,h_{j-1}(x'_0))&\quad\mbox{for}\quad j>j_0,
\end{cases}
\end{align}
and the vector field
$\tilde\ell_{,j}:=(\tilde\ell_{1,j},\dots,\tilde\ell_{d,j})$ is a smooth extension of $\ell|_{\cD_j}$ to $\cup_{k=1,k\neq j}^{m+1}\cD_k$.
From \eqref{eq6.26}, we have
\begin{align}\label{eq6.59}
\begin{cases}
-\partial_tu_{_\ell}+D_\alpha(A^{\alpha\beta}D_\beta u_{_\ell})=g+ D_\alpha f_2^\alpha \quad\text{in}\,\,\bigcup_{j=1}^{m+1}(-1,0)\times\cD_j, \\
[n_j^\alpha (A^{\alpha\beta} D_\beta u_{_\ell}-f_2^\alpha)]_{(-1,0)\times\Gamma_j}=\tilde h_j,
\end{cases}
\end{align}
where
\begin{align*}
f_2^\alpha&:=f_2^\alpha(z;z_0)=f_1^\alpha-A^{\alpha\beta}\sum_{j=1}^{m+1}D_\beta \tilde\ell_{i,j}D_iu(t_0,P_jx_0)\nonumber\\
&=D_\ell f^\alpha-D_{\ell} A^{\alpha\beta}D_\beta u+A^{\alpha\beta}\left(D_\beta \ell_i D_i u-\sum_{j=1}^{m+1}D_\beta \tilde\ell_{i,j}D_iu(t_0,P_jx_0)\right).
\end{align*}
Now by solving a conormal boundary value problem (or simply adding a term
$$
\sum_{j=1}^{m}D_d(\mathbbm{1}_{x^d>h_j(x')} (n^d_j(x'))^{-1}\tilde h_j(t,x'))
$$
to the equation; see \cite{dx2021}), where $\mathbbm{1}_{\bullet}$ is  the indicator function, we can get rid of $\tilde h_j$ in the second equation of \eqref{eq6.59} and reduce the problem \eqref{eq6.59} to a homogeneous transmission problem:
\begin{align}\label{homosecond}
\begin{cases}
-\partial_tu_{_\ell}+D_\alpha(A^{\alpha\beta}D_\beta u_{_\ell})=g+ D_\alpha f_3^\alpha \quad\text{in}\,\,\bigcup_{j=1}^{m+1}(-1,0)\times\cD_j, \\
[n_j^\alpha (A^{\alpha\beta} D_\beta u_{_\ell}-f_3^\alpha)]_{(-1,0)\times\Gamma_j}=0,
\end{cases}
\end{align}
where
\begin{align}\label{tildef1}
f_3^\alpha:=f_3^\alpha(z;z_0)&=f_2^\alpha+\delta_{\alpha d}\sum_{j=1}^{m}\mathbbm{1}_{x^d>h_j(x')} (n^d_j(x'))^{-1}\tilde h_j(t,x')\nonumber\\
&=D_\ell f^\alpha-D_{\ell} A^{\alpha\beta}D_\beta u+A^{\alpha\beta}\big(D_\beta \ell_i D_i u-\sum_{j=1}^{m+1}D_\beta \tilde\ell_{i,j}D_iu(t_0,P_jx_0)\big)\nonumber\\
&\quad+\delta_{\alpha d}\sum_{j=1}^{m}\mathbbm{1}_{x^d>h_j(x')} (n^d_j(x'))^{-1}\tilde h_j(t,x'),
\end{align}
where $\delta_{\alpha d}=1$ if $\alpha=d$, and $\delta_{\alpha d}=0$ if $\alpha\neq d$.

To prove piecewise regularity of $u_{_\ell}$, we need to show the mean oscillation of $Du_{\ell}$ in cylinders vanishes in a certain  order as the radii of the cylinders go to zero, so that Campanato's approach in  \cite{c1963,g1983} can be applied. However, we note that the mean oscillation of
$$A^{\alpha\beta}\big(D_\beta \ell_i D_i u-\sum_{j=1}^{m+1}D_\beta \tilde\ell_{i,j}D_iu(t_0,P_jx_0)\big)$$
in \eqref{tildef1} is only bounded. To this end, we choose a cut-off function $\zeta\in C_{0}^\infty(B_1)$ satisfying
$$0\leq\zeta\leq1,\quad\zeta\equiv 1~\mbox{in}~B_{3/4},\quad|D\zeta|\leq8.$$
Denote
\begin{equation}\label{mathcalA}
\boldsymbol{A}^{\alpha\beta}:=\zeta A^{\alpha\beta}+\nu(1-\zeta)\delta_{\alpha\beta}\delta_{ij}.
\end{equation}
For $j=1,\ldots,m+1$, denote $\cD_j^c:=\cD\setminus\cD_j$. Let ${\mathfrak u}_j(\cdot;z_0)\in\mathcal{H}_p^1(Q_1^-)$ be the weak solution of the problem
\begin{align}\label{eq-rmu}
\begin{cases}
-\partial_t{\mathfrak u}_j(\cdot;z_0)+D_{\alpha}(\boldsymbol{A}^{\alpha\beta}D_\beta {\mathfrak u}_j(\cdot;z_0))=-D_{\alpha}(\mathbbm{1}_{_{(-1,0)\times\cD_j^c}}A^{\alpha\beta}D_\beta \tilde\ell_{i,j}D_iu(t_0,P_jx_0))&\,\, \mbox{in}~Q_1^-,\\
{\mathfrak u}_j(\cdot;z_0)=0&\,\, \mbox{on}~\partial_pQ_1^-,
\end{cases}
\end{align}
where $1<p<\infty$. The solvability of \eqref{eq-rmu} follows from Lemma \ref{solvability}.
Furthermore, we have
\begin{align}\label{rmuj}
\|{\mathfrak u}_j(\cdot;z_0)\|_{\mathcal{H}_p^1(Q_1^-)}&\leq N\|\mathbbm{1}_{_{(-1,0)\times\cD_j^c}}A^{\alpha\beta}D_\beta \tilde\ell_{i,j}D_iu(t_0,P_jx_0)\|_{L_p(Q_1^-)}\nonumber\\
&\leq N\big(\|Du\|_{L_{1}(\cQ)}+\sum_{j=1}^{M}|f|_{(1+\delta)/2,1+\delta;(-1,0)\times\overline{\cD_{j}}}\big),
\end{align}
where we used the local boundedness estimate of $Du$ in Lemma \ref{lemlocbdd}.
Since $\mathbbm{1}_{_{(-1,0)\times\cD_j^c}}D_\beta \tilde\ell_{,j}$ is piecewise $C^{\mu}$, from Lemma \ref{lemlocbdd}, we obtain
$$
{\mathfrak u}_j(\cdot;z_0)\in C^{(1+\mu')/2,1+\mu'}((-1+\varepsilon,0)\times(\overline{\cD_i}\cap B_{1-\varepsilon})),
$$
where $\mu':=\min\{\mu,\frac{1}{2}\}$ and $i=1,\ldots,m+1$. Thus, combining Lemma \ref{lemlocbdd} and \eqref{rmuj}, we get for $i=1,\ldots,m+1$,
\begin{align}\label{eq9.46}
&\|{\mathfrak u}\|_{L_\infty(Q_{1/4}^-)}+|{\mathfrak u}|_{(1+\mu')/2,1+\mu';(-1+\varepsilon,0)\times(\overline{\cD_i}\cap B_{1-\varepsilon})}\nonumber\\
&\leq N\big(\sum_{j=1}^{m+1}\|D{\mathfrak u}_j\|_{L_{1}(\cQ)}+\sum_{j=1}^{m+1}|\mathbbm{1}_{_{(-1,0)\times\cD_j^c}}A^{\alpha\beta}D_\beta \tilde\ell_{i,j}D_iu(t_0,P_jx_0)|_{\mu/2,\mu;(-1,0)\times\overline{\cD_{j}}}\big)\nonumber\\
&\le N\big(\|Du\|_{L_{1}(\cQ)}+\sum_{j=1}^{M}|f|_{(1+\delta)/2,1+\delta;(-1,0)\times\overline{\cD_{j}}}\big),
\end{align}
where
\begin{equation}\label{def-v}
{\mathfrak u}:={\mathfrak u}(\cdot;z_0)=\sum_{j=1}^{m+1}{\mathfrak u}_j(\cdot;z_0).
\end{equation}

Now we define
\begin{equation}\label{def-tildeu}
\tilde u:=\tilde u(z;z_0)=u_{_\ell}-{\mathfrak u}=D_{\ell}u-u_0-{\mathfrak u}.
\end{equation}
Then $\tilde u$ satisfies
\begin{equation}\label{eqtildeu}
-\tilde u_t+D_\alpha(A^{\alpha\beta}D_\beta \tilde u)=g+ D_\alpha \tilde f^\alpha\quad\mbox{in}~Q_{3/4}^-,
\end{equation}
where
\begin{equation}\label{def-tildefalpha}
\tilde f^\alpha:=\tilde f^\alpha(z;z_0)=f_3^\alpha+A^{\alpha\beta}\sum_{j=1}^{m+1}\mathbbm{1}_{_{(-1,0)\times\cD_j^c}}D_\beta \tilde\ell_{i,j}D_iu(t_0,P_jx_0),
\end{equation}
where $f_3^\alpha$ is defined in \eqref{tildef1}.  The mean oscillation of $\tilde f^\alpha$ vanishes at a certain  rate as the radii of the cylinders go to zero; the details can be found in the proof of \eqref{estF} below. Therefore, we deduce from \eqref{eq9.46} and \eqref{def-tildeu} that to prove  piecewise regularity of $u_{_\ell}$, we only need to prove that of $\tilde u$.

Denote
\begin{equation}\label{deftildeU}
\tilde U:=\tilde U(z;z_0)=n^\alpha(A^{\alpha\beta}D_\beta \tilde u-\tilde f^\alpha),
\end{equation}
where
$n^\alpha$ is defined in \eqref{defnorm}, $\alpha=1,\ldots,d$. The rest part of this section is devoted to deriving  piecewise $C^{\mu'}$-continuity of $D_{\ell^{k'}}\tilde u$ and $\tilde U$, where $k'=1,\ldots,d-1$, and $\mu'=\min\big\{\frac{1}{2},\mu\big\}$. For this, we will prove the following proposition.
\begin{proposition}\label{proptildeu}
Let $\varepsilon\in(0,1)$ and $p\in(1,\infty)$. Suppose that $A^{\alpha\beta}$ and $f^\alpha$ satisfy Assumption \ref{assump} with $s=1$. If $u\in \mathcal{H}_{p}^{1}((-1,0)\times B_{1})$ is a weak solution to
$$
-u_{t}+D_{\alpha}(A^{\alpha\beta}D_{\beta}u)=D_{\alpha}f^\alpha\quad\mbox{in}~ (-1,0)\times B_{1},
$$
then the following assertions hold.

\begin{enumerate}
\item For any $z_{0}, z_1\in  (-1+\varepsilon,0)\times B_{1-\varepsilon}$,  we have
\begin{align}\label{holdertildeuU}
&|(D_{\ell^{k'}}\tilde u(z_{0};z_0)-D_{\ell^{k'}}\tilde u(z_{1};z_1)|+|\tilde U(z_{0};z_0)-\tilde U(z_{1};z_1)|\nonumber\\
&\leq N|z_{0}-z_{1}|_p^{\mu'}\Big(\|Du\|_{L_{1}((-1,0)\times B_{1})}+\sum_{j=1}^{M}|f|_{(1+\delta)/2,1+\delta;(-1,0)\times \overline{\cD_{j}}}\Big),
\end{align}
where  $\tilde u$ and $\tilde U$ are defined in \eqref{def-tildeu} and \eqref{deftildeU}, respectively, $\mu'=\min\big\{\frac{1}{2},\mu\big\}$, $N$ depends on $n,d,m,p,\nu,\varepsilon$, $|A|_{(1+\delta)/2,1+\delta;(-1,0)\times\overline{\cD_{j}}}$, and the $C^{2+\mu}$ characteristic of $\cD_{j}$.
\item For $j=1,\ldots,m+1$, it  holds that $u\in C^{1+\mu'/2,2+\mu'}((-1+\varepsilon,0)\times (B_{1-\varepsilon}\cap\overline{{\cD}_{j}}))$.
\end{enumerate}
\end{proposition}

The proof of Proposition \ref{proptildeu} is based on the idea in the proof of \cite[Proposition 4.2]{dx2021}, which is an adaptation of Campanato's method in  \cite{c1963,g1983}.  We shall first establish an a priori estimate of the modulus of continuity of $(D_{\ell^{k'}}\tilde u,\tilde U)$ by assuming that $Du$ is $C^{1/2}$ in $t$, and piecewise $C^1$ in $x$. The general case follows from an approximation argument together with  the technique of locally flattening the boundaries \cite[p. 2466]{dx2019}.

We next derive some auxiliary results in Sections \ref{auxilemma} and \ref{sublocalbound}, and then prove Proposition \ref{proptildeu} in Section \ref{prfprop}.

\subsection{Decay estimates}\label{auxilemma}
As in Section \ref{Assumption}, we take $z_0=(t_0,x_{0})\in (-9/16,0)\times (B_{3/4}\cap \cD_{j_{0}})$ and $r\in(0,1/4)$.

\begin{lemma}\label{lemmeanDu}
Let $\tilde u$ and $\tilde f^\alpha,\alpha=1,\ldots,d,$ be defined in \eqref{def-tildeu} and \eqref{def-tildefalpha}, respectively. Then  we have
\begin{align}\label{estfintDtildeu}
\fint_{Q_{r}^{-}(z_{0})}|D\tilde u|\,dz
\leq N\mathcal{C}_0,
\end{align}
and
\begin{align}\label{estfinttildef}
\fint_{Q_{r}^{-}(z_{0})}|\tilde f^\alpha|\,dz
\leq N\mathcal{C}_0,
\end{align}
where $N$ depends on $n,d,m,p,\nu,|A|_{(1+\delta)/2,1+\delta;(-1,0)\times\overline{\cD_{j}}}$, and the $C^{2+\mu}$ norm of $h_j$, and
\begin{align}\label{defC0} \mathcal{C}_0&:=\sum_{j=1}^{m+1}[Du]_{1/2,1;Q_{r}^-(z_0)\cap((-1+\varepsilon,0)\times\overline{\cD_j})}+\sum_{j=1}^{M}|f|_{(1+\delta)/2,1+\delta;(-1,0)\times\overline{\cD_{j}}}+\|Du\|_{L_{1}(\cQ)}.
\end{align}
\end{lemma}

\begin{proof}
{\bf Step 1. Proof of \eqref{estfintDtildeu}.} It follows from \eqref{def-tildeu} and \eqref{defu0} that
\begin{align}\label{meanDtildeu}
\fint_{Q_{r}^{-}(z_{0})}|D\tilde u|\,dz&\leq \sum_{j=1}^{m+1}\|D^2u\|_{L_\infty(Q_{r}^-(z_0)\cap((-1+\varepsilon,0)\times\cD_j))}+\|D{\mathfrak u}\|_{L_\infty(Q_{r}^-(z_0))}\nonumber\\
&\quad+\fint_{Q_{r}^{-}(z_{0})}\big|D\ell^k Du-\sum_{j=1}^{m+1}D\tilde\ell_{,j}^kDu(t_0,P_jx_0)\big|\,dz.
\end{align}
For each $i=1,\ldots,m+1$,
\begin{align}\label{omegalk00}
&\big\|D\ell^k Du-\sum_{j=1}^{m+1}D\tilde\ell_{,j}^kDu(t_0,P_jx_0)\big\|_{L_1(Q_{r}^-(z_0)\cap((-1+\varepsilon,0)\times\cD_i))}\nonumber\\
&\leq N\big\|D\ell^k (Du-Du(t_0,P_ix_0))\big\|_{L_1(Q_{r}^-(z_0)\cap((-1+\varepsilon,0)\times\cD_i))}\nonumber\\
&\quad+N\big\|\sum_{j=1,j\neq i}^{m+1}D\tilde\ell_{,j}^kDu(t_0,P_jx_0)\big\|_{L_1(Q_{r}^-(z_0)\cap((-1+\varepsilon,0)\times\cD_i))}.
\end{align}
To estimate the right-hand side of \eqref{omegalk00}, we may assume that $Q_{r}^-(z_0)\cap((-1+\varepsilon,0)\times\cD_i)\neq\emptyset$. It follows from the definition of $P_ix_0$ in \eqref{Pjx} that if $i>{j_0}$, then $P_ix_0=(x'_0,h_{i-1}(x'_0))$ and
\begin{equation*}
|x_0-P_ix_0|=|x_0^d-h_{i-1}(x'_0)|\leq Nr.
\end{equation*}
The case of  $i\leq j_0$ is proved similarly.
Thus, we obtain
$$|x-P_ix_0|\leq|x-x_0|+|x_0-P_ix_0|\leq Nr$$
and
\begin{align*}
|Du(t,x)-Du(t_0,P_ix_0)|\leq Nr[Du]_{1/2,1;Q_{r}^-(z_0)\cap((-1+\varepsilon,0)\times\overline{\cD_i})}.
\end{align*}
With this, we have
\begin{align}\label{estDellk00}
&\big\|D\ell^k (Du-Du(t_0,P_ix_0))\big\|_{L_1(Q_{r}^-(z_0)\cap((-1+\varepsilon,0)\times\cD_i))}\nonumber\\
&\leq Nr^{3}[Du]_{1/2,1;Q_{r}^-(z_0)\cap((-1+\varepsilon,0)\times\overline{\cD_i})}\int_{B_r(x_0)\cap\cD_i}|D\ell^k|\ dx.
\end{align}
By using Lemma \ref{lemell}(iii) and $\min\{a,b\}\leq\sqrt{ab}$ for $a,b\geq0$, we have
\begin{align}\label{estDellk0}
\int_{B_r(x_0)\cap\cD_i}|D\ell^k|\ dx\leq N\int_{B_r'(x'_0)}\frac{\min\{2r,h_i-h_{i-1}\}}{|h_{i}-h_{i-1}|^{1/2}}\ dx'\leq Nr^{d-1/2}.
\end{align}
Thus, substituting \eqref{estDellk0} into \eqref{estDellk00}, we obtain
\begin{align}\label{estDellk}
&\big\|D\ell^k (Du-Du(t_0,P_ix_0))\big\|_{L_1(Q_{r}^-(z_0)\cap((-1+\varepsilon,0)\times\cD_i))}\notag\\
&\leq Nr^{d+5/2}[Du]_{1/2,1;Q_{r}^-(z_0)\cap((-1+\varepsilon,0)\times\overline{\cD_i})}.
\end{align}
Making use of the fact that $\tilde\ell_{,j}$ is the smooth extension of $\ell|_{\cD_j}$ to $\cup_{k=1,k\neq j}^{m+1}\cD_k$ and the local boundedness of $Du$ in Lemma \ref{lemlocbdd},  we have for each $i=1,\ldots,m+1$,
\begin{align}\label{estDelld}
&\Big\|\sum_{j=1,j\neq i}^{m+1}D\tilde\ell_{,j}^kDu(t_0,P_jx_0)\Big\|_{L_1(Q_{r}^-(z_0)\cap((-1+\varepsilon,0)\times\cD_i))}\nonumber\\
&\leq Nr^{d+2}\big(\|Du\|_{L_{1}(\cQ)}+\sum_{j=1}^{M}|f|_{(1+\delta)/2,1+\delta;(-1,0)\times\overline{\cD_{j}}}\big).
\end{align}
Substituting \eqref{estDellk} and \eqref{estDelld} into \eqref{omegalk00}, we obtain
\begin{align}\label{omegalk}
&\big\|D\ell^k Du-\sum_{j=1}^{m+1}D\tilde\ell_{,j}^kDu(t_0,P_jx_0)\big\|_{L_1(Q_{r}^-(z_0)\cap((-1+\varepsilon,0)\times\cD_i))}\nonumber\\
&\leq Nr^{d+2}\big([Du]_{1/2,1;Q_{r}^-(z_0)\cap((-1+\varepsilon,0)\times\overline{\cD_i})}+\|Du\|_{L_{1}(\cQ)}+\sum_{j=1}^{M}|f|_{(1+\delta)/2,1+\delta;(-1,0)\times\overline{\cD_{j}}}\big).
\end{align}
Substituting \eqref{eq9.46} and \eqref{omegalk} into \eqref{meanDtildeu}, we obtain \eqref{estfintDtildeu}.

{\bf Step 2. Proof of \eqref{estfinttildef}.} By \eqref{tildef1} and \eqref{def-tildefalpha}, we have
\begin{align}\label{meantildef}
\fint_{Q_{r}^{-}(z_{0})}|\tilde f^\alpha|\,dz
&\leq N\big(\sum_{j=1}^{M}|f|_{(1+\delta)/2,1+\delta;(-1,0)\times\overline{\cD_{j}}}+\|Du\|_{L_{1}(\cQ)}\big)\nonumber\\
&\quad+N\fint_{Q_{r}^-(z_0)}\big|D_\beta \ell_i D_iu-\sum_{j=1}^{m+1}\mathbbm{1}_{_{(-1,0)\times\cD_j}}D_\beta \ell_{i}D_iu(t_0,P_jx_0)\big|\,dz\nonumber\\
&\quad+\fint_{Q_{r}^{-}(z_{0})}\big|\delta_{\alpha d}\sum_{j=1}^{m}\mathbbm{1}_{x^d>h_j(x')} (n^d_j(x'))^{-1}\tilde h_j(t,x')\big|\,dz.
\end{align}
Using Lemma \ref{lemell}(iii) and similar to the proof of \eqref{estDellk}, we deduce that
\begin{align}\label{estsumDelldu}
&\fint_{Q_{r}^-(z_0)}\big|D_\beta \ell_i D_iu-\sum_{j=1}^{m+1}\mathbbm{1}_{_{(-1,0)\times\cD_j}}D_\beta \ell_{i}D_iu(t_0,P_jx_0)\big|\,dz\nonumber\\
&\leq N\sum_{j=1}^{m+1}\fint_{Q_{r}^-(z_0)\cap((-1+\varepsilon,0)\times\cD_{j})}|D_\beta\ell_i(D_iu-D_iu(t_0,P_jx_0))|\,dz\nonumber\\
&\leq Nr^{\frac{1}{2}}\sum_{j=1}^{m+1}[Du]_{1/2,1;Q_{r}^-(z_0)\cap((-1+\varepsilon,0)\times\overline\cD_j)}.
\end{align}
Furthermore, by \eqref{normal} and \eqref{deftildeh}, we have
\begin{align}\label{esttildeh}
&\fint_{Q_{r}^-(z_0)}\big|\sum_{j=1}^{m}\mathbbm{1}_{x^d>h_j(x')} (n^d_j(x'))^{-1}\tilde h_j(t,x')\big|\,dz\notag\\
&\leq N\big(\|Du\|_{L_{1}(\cQ)}+\sum_{j=1}^{M}|f|_{(1+\delta)/2,1+\delta;(-1,0)\times\overline{\cD_{j}}}\big).
\end{align}
Coming back to \eqref{meantildef} and using \eqref{estsumDelldu} and \eqref{esttildeh}, the estimate \eqref{estfinttildef} is proved. The proof of the lemma is complete.
\end{proof}

Denote
\begin{equation*}
\Phi(z_{0},r):=\inf_{\mathbf q^{k'},\mathbf Q\in\mathbb R^{n}}\left(\fint_{Q_r^-(z_0)}\big(|D_{\ell^{k'}}\tilde u(z;z_0)-\mathbf q^{k'}|^{\frac{1}{2}}+|\tilde U(z;z_0)-\mathbf Q|^{\frac{1}{2}}\big)\,dz \right)^{2},
\end{equation*}
where $\tilde u$ and $\tilde U$ are defined in \eqref{def-tildeu} and \eqref{deftildeU}, respectively. We shall establish a decay estimate of $\Phi(z_{0},r)$. Before this, we first set
\begin{equation}\label{transformation}
\begin{split}
&\tilde {\mathfrak u}(t,y;\Lambda z_0)=\tilde u(t,x;z_0), \quad \mathcal{A}^{\alpha\beta}(t,y)=\Lambda^{\alpha k}A^{ks}(t,x)\Lambda^{s\beta}, \\
&\tilde {\mathfrak f}^\alpha(t,y;\Lambda z_0)=\Lambda^{\alpha k}\tilde f^k(t,x;z_0),\quad {\mathfrak g}(t,y)=g(t,x),
\end{split}
\end{equation}
where $y=\Lambda x$ and $\Lambda=(\Lambda^{\alpha\beta})_{\alpha,\beta=1}^{d}$ is a $d\times d$ orthogonal matrix representing the linear transformation from the coordinate system associated with $0$ to the coordinate system associated with $x_0$ defined in Section \ref{sec 2.3}. Then we have from \eqref{eqtildeu} that $\tilde{\mathfrak u}$ satisfies
\begin{equation*}
-\tilde{\mathfrak u}_t+D_\alpha(\mathcal{A}^{\alpha\beta}D_\beta \tilde{\mathfrak u})={\mathfrak g}+ D_\alpha \tilde {\mathfrak f}^\alpha\quad\mbox{in}~\Lambda(Q_{3/4}^-),
\end{equation*}
where $\Lambda(Q_{3/4}^-):=(-9/16,0)\times\Lambda(B_{3/4})$.
Denote
\begin{align*}
&\phi(\Lambda z_{0},r)\\
&:=\inf_{\mathbf q^{k'},\mathbf Q\in\mathbb R^{n}}\left(\fint_{Q_r^-(\Lambda z_0)}\big(|D_{y^{k'}}\tilde{\mathfrak u}(z;\Lambda z_0)-\mathbf q^{k'}|^{\frac{1}{2}}+|\mathcal{A}^{d\beta}D_{y^\beta}\tilde{\mathfrak u}(z;\Lambda z_0)-\tilde {\mathfrak f}^d-\mathbf Q|^{\frac{1}{2}}\big)\,dz\right)^{2}.
\end{align*}
Then we have the following decay estimate of $\phi(\Lambda z_{0},r)$.

\begin{proposition}\label{lemma iteraphi}
Under the same assumptions as in Proposition \ref{proptildeu}. For any $0<\rho\leq r\leq 1/4$, we have
\begin{align*}
\phi(\Lambda z_0,\rho)&\leq N\Big(\frac{\rho}{r}\Big)^{\mu'}\phi(\Lambda z_0,r/2)+N\rho^{\mu'}\mathcal{C}_0,
\end{align*}
where $\mathcal{C}_0$ is defined in \eqref{defC0}, $\mu'=\min\big\{\frac{1}{2},\mu\big\}$, $N$ depends on $n,d,m,p,\nu$, the $C^{2+\mu}$ norm of $h_j$, and $|A|_{(1+\delta)/2,1+\delta;(-1,0)\times\overline{\cD_{j}}}$.
\end{proposition}

The proof of  Proposition \ref{lemma iteraphi} will be given later. We first use it to prove a decay estimate of $\Phi(z_{0},r)$.
\begin{lemma}\label{lemma itera}
Under the same assumptions as in Proposition \ref{proptildeu}.
For any $0<\rho\leq r\leq 1/4$, we have
\begin{align}\label{est phi'}
\Phi(z_{0},\rho)&\leq N\Big(\frac{\rho}{r}\Big)^{\mu'}\Phi(z_{0},r/2)+N\rho^{\mu'}\mathcal{C}_0,
\end{align}
where $\mathcal{C}_0$ is defined in \eqref{defC0}, $\mu'=\min\big\{\frac{1}{2},\mu\big\}$, $N$ depends on $n,d,m,p,\nu$, the $C^{2+\mu}$ norm of $h_j$, and $|A|_{(1+\delta)/2,1+\delta;(-1,0)\times\overline{\cD_{j}}}$.
\end{lemma}

\begin{proof}
Let $y_0$ be as in Section \ref{sec 2.3}. Note that
\begin{equation}\label{Dell-nalpha}
\begin{split}
&D_{\ell^k}\tilde u(t,x;z_0)-D_{y^k}\tilde{\mathfrak u}(t,y;\Lambda z_0)=(\ell^k(x)-\tau_k)\cdot D\tilde u(t,x;z_0),\\
&\tilde U(t,x;z_0)-\mathcal{A}^{d\beta}(t,y)D_{y^\beta}\tilde{\mathfrak u}(t,y;\Lambda z_0)+\tilde{\mathfrak f}^d(t,y;\Lambda z_0)\\
&=(n^\alpha-n_{y_0}^\alpha)(A^{\alpha\beta}(t,x)D_\beta \tilde u(t,x;z_0)-\tilde f^\alpha(t,x;z_0)),
\end{split}
\end{equation}
where  $\tau_k$ and $n_{y_0}^\alpha$ are defined in \eqref{deftauk} and \eqref{defny0}, respectively. For any $x\in B_r(x_0)\cap\cD_{j}$, where $r\in(|x_0-y_0|,1)$ and $j=1,\ldots,m+1$, we shall first estimate $|\ell_d^{k,0}(x)-D_{k}h_{j_0}(y'_0)|$ according to the following three cases:

{\bf Case 1.}  If $B_r(x_0)\cap\Gamma_j\neq\emptyset$ and $B_r(x_0)\cap\Gamma_{j-1}\neq\emptyset$, then by using \eqref{Diff_Dh}, we obtain
\begin{align}\label{diffe_Dkh}
|D_kh_{j}(x')-D_{k}h_{j-1}(x')|\leq N|h_{j}(x')-h_{j-1}(x')|^{1/2}\leq N\sqrt r.
\end{align}
Thus, by the triangle inequality,
\begin{align}\label{diffell}
&|\ell_d^{k,0}(x)-D_{k}h_{j_0}(y'_0)|\notag\\
&\leq |\ell_d^{k,0}(x)-D_{k}h_{j}(x')|+|D_{k}h_{j}(x')-D_{k}h_{j_0}(y'_0)|\nonumber\\
&\leq \Big|\frac {h_{j}-x^d}{h_{j}-h_{j-1}}
D_k(h_{j-1}-h_j)(x')\Big|+|D_kh_{j}(x')-D_{k}h_{j_0}(x')|+Nr\leq N\sqrt r.
\end{align}

{\bf Case 2.} If $B_r(x_0)\cap\Gamma_j\neq\emptyset$ and  $B_r(x_0)\cap\Gamma_{j-1}=\emptyset$ (see Figure \ref{case2}), then $|h_{j}(x')-x^d|\leq Nr$, and by \eqref{diffe_Dkh}, we have
\begin{align}\label{diffell2}
|\ell_d^{k,0}(x)-D_{k}h_{j_0}(y'_0)|
&\leq |\ell_d^{k,0}(x)-D_{k}h_{j}(x')|+|D_{k}h_{j}(x')-D_{k}h_{j_0}(y'_0)|\nonumber\\
&\leq \Big|\frac {h_{j}-x^d}{h_{j}-h_{j-1}}
D_k(h_{j-1}-h_j)(x')\Big|+N\sqrt r\nonumber\\
&\leq \frac {N|h_{j}-x^d|}{|h_{j}-h_{j-1}|^{1/2}}+N\sqrt r\leq N|h_j-x^d|^{1/2}+N\sqrt r\leq N\sqrt r.
\end{align}

\begin{figure}
\begin{tikzpicture}
\draw (0,0) circle (1.5);
\fill (0,0) circle (1pt) node[right] {$x_0$};
\draw [-] (0,0) -- (-1.45,-0.4) node at (-0.75,-0.35){$r$};
\draw[blue,domain=-2.5:2.5] plot(\x,0.1*\x*\x+0.2*\x+0.5) node at (2.4,1.9){$\Gamma_{j_0}$};
\fill (-0.1,0.481) circle (1pt) node[below left] {$y_0$};
\draw [->] (-0.1,0.481) -- (0.6,0.621) node[above] {$\tau_k$};
\draw [->] (-0.1,0.481) -- (-0.2,0.973) node[left] {$n_{y_0}$};
\draw [dashed] (0,0) -- (-0.1,0.481);
\draw [-] (-0.12,0.5772) -- (0.05,0.6112) -- (0.07,0.515);
\draw [-] (-0.12,0.5772) -- (0.05,0.6112);
\draw[blue,domain=-2.5:2.5] plot(\x,0.1*\x*\x+0.05*\x-1) node at (2.4,0){$\Gamma_{j}$};
\draw[blue,domain=-2.5:2.5] plot(\x,-0.1*\x*\x-0.05*\x-1.8) node at (3,-2.5){$\Gamma_{j-1}$};
\fill (0.7,-1.1) circle (1pt) node[below left] {$x$};
\fill (1.5,-1.5)  node[right] {$\cD_j$};
\fill (0.7,-0.916) circle (1pt) node at (0.6,-0.6) {$(x',h_j(x'))$};
\fill (0.7,-1.884) circle (1pt) node at (0.4,-2.2) {$(x',h_{j-1}(x'))$};
\draw [dashed] (0.7,-0.916) -- (0.7,-1.884);
\end{tikzpicture}
\vspace*{0.01mm}
\caption{Illustration of Case 2}
\label{case2}
\end{figure}
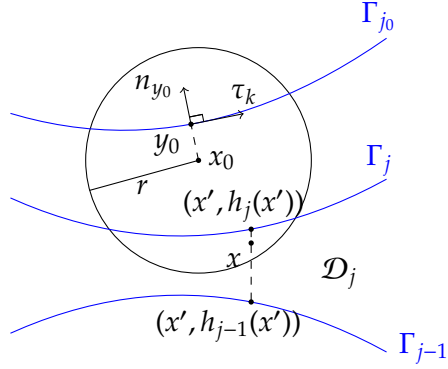

{\bf Case 3.} If $B_r(x_0)\cap\Gamma_j=\emptyset$ and $B_r(x_0)\cap\Gamma_{j-1}\neq\emptyset$, then similar to Case 2, by \eqref{diffe_Dkh},
\begin{align}\label{diffell3}
|\ell_d^{k,0}(x)-D_{k}h_{j_0}(y'_0)|
\leq N\sqrt r.
\end{align}

Combining \eqref{diffell}--\eqref{diffell3}, we have
\begin{align*}
|\ell_d^{k,0}(x)-D_{k}h_{j_0}(y'_0)|\leq N\sqrt r.
\end{align*}
Thus, recalling \eqref{defell}, \eqref{deftauk}, \eqref{defnorm}, and \eqref{defny0}, we obtain
\begin{align*}
|\ell^k(x)-\tau_{k}|\leq N\sqrt r,\quad |n(x)-n_{y_0}|\leq N\sqrt r,
\end{align*}
where $k=1,\ldots,d-1$.
This together with \eqref{Dell-nalpha} gives
\begin{equation}\label{difference-coor}
\begin{aligned}
&|D_{\ell^k}\tilde u(t,x;z_0)-D_{y^k}\tilde{\mathfrak u}(t,y;\Lambda z_0)|\leq N\sqrt r|D\tilde u(t,x;z_0)|,\\
&|\tilde U(t,x;z_0)-\mathcal{A}^{d\beta}(t,y)D_{y^\beta}\tilde{\mathfrak u}(t,y;\Lambda z_0)+\tilde {\mathfrak f}^d(t,y;\Lambda z_0)|\\
&\leq N\sqrt r(|D\tilde u(t,x;z_0)|+|\tilde f(t,x;z_0)|).
\end{aligned}
\end{equation}

Now by using the triangle inequality, \eqref{difference-coor}, \eqref{estfintDtildeu}, and \eqref{estfinttildef}, we have
\begin{align*}
&\left(\fint_{Q_\rho^-(z_0)}\big(|D_{\ell^{k'}}\tilde u(z;z_0)-\mathbf q^{k'}|^{\frac{1}{2}}+|\tilde U(z;z_0)-\mathbf Q|^{\frac{1}{2}}\big)\,dz \right)^{2}\\
&\leq \Bigg(\fint_{Q_\rho^-(\Lambda z_0)}\big(|D_{y^{k'}}\tilde{\mathfrak u}(t,y;\Lambda z_0)-\mathbf q^{k'}|^{\frac{1}{2}}\nonumber\\
&\qquad+|\mathcal{A}^{d\beta}(t,y)D_{y^\beta}\tilde{\mathfrak u}(t,y;\Lambda z_0)-\tilde {\mathfrak f}^d(t,y;\Lambda z_0)-\mathbf Q|^{\frac{1}{2}}\Big)\,dt\,dy\Bigg)^{2}+N\sqrt{\rho}\mathcal{C}_0,
\end{align*}
where $0<\rho\leq r\leq 1/4$ and $\mathcal{C}_0$ is defined in \eqref{defC0}.
Since $\mathbf{q}^{k'}, \mathbf{Q}\in\mathbb R^{n}$ are arbitrary, we obtain
\begin{align*}
\Phi(z_{0},\rho)\leq \phi(\Lambda z_{0},\rho)+N\sqrt{\rho}\mathcal{C}_0.
\end{align*}
This, in combination with  Proposition \ref{lemma iteraphi}, leads to that
\begin{align}\label{est phiPhi}
\Phi(z_{0},\rho)\leq N\Big(\frac{\rho}{r}\Big)^{\mu'}\phi(\Lambda z_{0},r/2)+N\rho^{\mu'}\mathcal{C}_0.
\end{align}
Similarly, we have
\begin{align*}
\phi(\Lambda z_{0},r/2)\leq\Phi(z_{0},r/2)+N\sqrt{r}\mathcal{C}_0.
\end{align*}
Substituting it into \eqref{est phiPhi} and using $\mu'\leq1/2$, we obtain
\begin{align*}
\Phi(z_{0},\rho)\leq N\Big(\frac{\rho}{r}\Big)^{\mu'}\Phi(z_{0},r/2)+N\rho^{\mu'}\mathcal{C}_0.
\end{align*}	
The lemma is proved.
\end{proof}

The rest of this section is to prove Proposition \ref{lemma iteraphi}. For this, we define
$$\widetilde{{\mathcal A}^{\alpha\beta}}=\eta \overline{{\mathcal A}^{\alpha\beta}}(y^{d})+\nu(1-\eta)\delta_{\alpha\beta}\delta_{ij},$$ where $\overline{{\mathcal A}^{\alpha\beta}}(y^{d})$ are piecewise constant functions corresponding to $\mathcal A^{\alpha\beta}$ (cf. p.\pageref{piecewisecont}), $y=\Lambda x$, $\Lambda=(\Lambda^{\alpha\beta})_{\alpha,\beta=1}^{d}$ is a $d\times d$ orthogonal matrix defined in Section \ref{sec 2.3}, and $\eta\in C_{0}^{\infty}(B_{r}(x_{0}))$ satisfies
$$0\leq\eta\leq1,\quad\eta\equiv1~\mbox{in}~B_{2r/3}(x_{0}),
\quad|D\eta|\leq {6}/{r}.$$
We have the following weak type-$(1,1)$ estimate (cf.  \cite[Lemma 4.3]{dx2021}).

\begin{lemma}\label{weak est barv}
Let $p\in(1,\infty)$. Let $v\in \mathcal{H}_{p}^{1}(Q_{r}^-(\Lambda z_0))$ be a weak solution to the problem
\begin{align*}
\begin{cases}
-v_{t}+D_{\alpha}(\widetilde{{\mathcal A}^{\alpha\beta}}D_{\beta}v)=G\mathbbm{1}_{Q_{r/2}^-(\Lambda z_0)}+\Div(F\mathbbm{1}_{Q_{r/2}^-(\Lambda z_0)})&\ \mbox{in}~Q_{r}^-(\Lambda z_0),\\
v=0&\ \mbox{on}~\partial_p Q_{r}^-(\Lambda z_0),
\end{cases}
\end{align*}
where $F, G\in L_{p}(Q_{r/2}^-(\Lambda z_0))$. Then for any $s>0$, we have
\begin{align*}
|\{z\in Q_{r/2}^-(\Lambda z_0): |Dv(z)|>s\}|\leq\frac{N}{s}\left(\|F\|_{L_{1}(Q_{r/2}^-(\Lambda z_0))}+r\|G\|_{L_{1}(Q_{r/2}^-(\Lambda z_0))}\right),
\end{align*}
where $N=N(n,d,p,\nu)$.
\end{lemma}

We shall choose suitable $F$ and $G$ in order to apply Lemma \ref{weak est barv}.  We first denote
\begin{align}\label{deff1}
\tilde f_1^\alpha(t,x;z_0)&:=A^{\alpha\beta}\big(D_\beta \ell_i D_i u-\sum_{j=1}^{m+1}\mathbbm{1}_{_{(-1,0)\times\cD_j}}D_\beta \ell_{i}D_iu(t_0,P_jx_0)\big),
\end{align}
and
\begin{align*}
\tilde f_2^\alpha(t,x):=D_\ell f^\alpha-D_{\ell} A^{\alpha\beta}D_\beta u+\delta_{\alpha d}\sum_{j=1}^{m}\mathbbm{1}_{x^d>h_j(x')} (n^d_j(x'))^{-1}\tilde h_j(t,x').
\end{align*}
Then $\tilde f_1^\alpha+\tilde f_2^\alpha=\tilde f^\alpha$ which is defined in \eqref{def-tildefalpha}. Moreover, it follows from $f^\alpha, A^{\alpha\beta}\in C^{\frac{1+\delta}{2},1+\delta}((-1+\varepsilon,0)\times\overline{\cD_{j}})$, $D_{\ell}n_j^\alpha\in C^{\mu}$ together with the definition of $\tilde h_j(t,x')$ in \eqref{deftildeh}, and the fact that the vector field $\ell$ is $C^{1/2}$, which is proved in Lemma \ref{lemell} (i), that
$$\tilde f_2^\alpha\in C^{\mu'/2,\mu'}((-1+\varepsilon,0)\times\overline{\cD_{j}}),\quad  \mu'=\min\big\{\frac{1}{2},\mu\big\}.$$
Define
\begin{equation}\label{deff1f2}
\tilde {\mathfrak f}_1^\alpha(t,y;\Lambda z_0)=\Lambda^{\alpha k}\tilde f_1^k(t,x;z_0),\quad \tilde {\mathfrak f}_2^\alpha(t,y)=\Lambda^{\alpha k}\tilde f_2^k(t,x).
\end{equation}
Then $\tilde {\mathfrak f}_1^\alpha+\tilde {\mathfrak f}_2^\alpha=\tilde {\mathfrak f}^\alpha$ which is defined in \eqref{transformation} and $\tilde {\mathfrak f}_2^\alpha$ is also piecewise $C^{\mu'/2,\mu'}$. Now we choose
\begin{align}\label{def-Falpha}
F^\alpha:=F^\alpha(t,y;\Lambda z_0)&=(\overline{{\mathcal A}^{\alpha\beta}}(y^{d})-\mathcal{A}^{\alpha\beta}(t,y))D_{y^\beta}\tilde{\mathfrak u}(t,y;\Lambda z_0)+\tilde {\mathfrak f}_1^\alpha(t,y;\Lambda z_0)\nonumber\\
&\quad+\tilde {\mathfrak f}_2^\alpha(t,y)-\bar{\mathfrak f}_2^\alpha(y^d),\quad F=(F^1,\ldots,F^d),
\end{align}
and
\begin{equation}\label{def-G}
G={\mathfrak g}(t,y)=g(t,x),
\end{equation}
where $\bar{\mathfrak f}_2^\alpha(y^d)$ is a piecewise constant function corresponding to $\tilde {\mathfrak f}_2^\alpha(t,y)$ and $g(t,x)$ is defined in \eqref{defg}.
The following result holds.

\begin{lemma}\label{lemmaFG}
Let $F$ and $G$ be defined as in \eqref{def-Falpha} and \eqref{def-G}, respectively. Then we have
\begin{align}\label{estF}
\|F\|_{L_1(Q_{r}^-(\Lambda z_0))}\leq Nr^{d+2+\mu'}\mathcal{C}_0
\end{align}
and
\begin{align*}
&\|G\|_{L_{1}(Q_{r}^-(\Lambda z_0))}\\
&\leq N r^{d+\frac{3}{2}}\Big(\sum_{j=1}^{m+1}\|D^2u\|_{L_\infty(Q_{r}^-(z_0)\cap((-1+\varepsilon,0)
\times\cD_j))}\\
&\qquad\qquad +\sum_{j=1}^{M}|f|_{(1+\delta)/2,1+\delta;(-1,0)\times\overline{\cD_{j}}}
+\|Du\|_{L_{1}(\cQ)}\Big),
\end{align*}
where $\mathcal{C}_0$ is defined in \eqref{defC0}, $\mu'=\min\big\{\frac{1}{2},\mu\big\}$, $N$ depends on  $|A|_{(1+\delta)/2,1+\delta;(-1,0)\times\overline{\cD_{j}}}$,  $n,d,p,m,\nu$, and the $C^{2+\mu}$ norm of $h_j$.
\end{lemma}

\begin{proof}
Combining  \eqref{deff1} and $A^{ks}(t,x)=\Gamma^{k\alpha }\mathcal{A}^{\alpha\beta}(t,y)\Gamma^{\beta s}$, where $\Gamma=\Lambda^{-1}$, we have
\begin{align}\label{f1alpha}
&\tilde {\mathfrak f}_1^\alpha(t,y;\Lambda z_0)=\Lambda^{\alpha k}\tilde f_1^k(t,x;z_0)\notag\\
&=\Lambda^{\alpha k}A^{ks}\big(D_s \ell_i D_i u-\sum_{j=1}^{m+1}\mathbbm{1}_{_{(-1,0)\times\cD_j}}D_s \ell_{i}D_iu(t_0,P_jx_0)\big)\nonumber\\
&=\mathcal{A}^{\alpha\beta}(t,y)\Gamma^{\beta s}\big(D_s \ell_i D_i u-\sum_{j=1}^{m+1}
\mathbbm{1}_{_{(-1,0)\times\cD_j}}D_s \ell_{i}D_iu(t_0,P_jx_0)\big).
\end{align}
From $\tilde {\mathfrak u}(t,y;\Lambda z_0)=\tilde u(t,x;z_0)$ in \eqref{transformation}, one has
\begin{align*}
(\overline{{\mathcal A}^{\alpha\beta}}(y^{d})-\mathcal{A}^{\alpha\beta}(t,y))D_{y^\beta}\tilde{\mathfrak u}(t,y;\Lambda z_0)=(\overline{{\mathcal A}^{\alpha\beta}}(y^{d})-\mathcal{A}^{\alpha\beta}(t,y))\Gamma^{\beta s}D_s\tilde u(t,x;z_0).
\end{align*}
By using \eqref{defu0} and \eqref{def-tildeu}, we have
\begin{align*}
D_s\tilde u(t,x;z_0)=\ell_iD_{s}D_iu-D_{s}\mathfrak u+D_s \ell_i D_iu-\sum_{j=1}^{m+1}D_s \tilde\ell_{i,j}D_iu(t_0,P_jx_0).
\end{align*}
These together with \eqref{f1alpha} give
\begin{align}\label{Af1}
&(\overline{{\mathcal A}^{\alpha\beta}}(y^{d})-\mathcal{A}^{\alpha\beta}(t,y))D_{y^\beta}\tilde{\mathfrak u}(t,y;\Lambda z_0)+\tilde {\mathfrak f}_1^\alpha(t,y;\Lambda z_0)\nonumber\\
&=(\overline{{\mathcal A}^{\alpha\beta}}(y^{d})-\mathcal{A}^{\alpha\beta}(t,y))\Gamma^{\beta s}D_s\tilde u(t,x;z_0)+\Lambda^{\alpha k}\tilde f_1^k(t,x;z_0)\nonumber\\
&=(\overline{{\mathcal A}^{\alpha\beta}}(y^{d})-\mathcal{A}^{\alpha\beta}(t,y))\Gamma^{\beta s}(\ell_iD_{s}D_iu-D_{s}\mathfrak u-\sum_{j=1}^{m+1}\mathbbm{1}_{_{(-1,0)\times\cD_j^c}}D_s \tilde\ell_{i,j}D_iu(t_0,P_jx_0))\nonumber\\
&\quad+\overline{{\mathcal A}^{\alpha\beta}}(y^{d})\Gamma^{\beta s}
\big(D_s \ell_i D_iu-\sum_{j=1}^{m+1}
\mathbbm{1}_{_{(-1,0)\times\cD_j}}
D_s\ell_{i}D_iu(t_0,P_jx_0)\big)=:I_{1}+I_{2}.
\end{align}
Similar to \eqref{estAbarf}, in view of $\mathcal{A}\in C^{(1+\delta)/2,1+\delta}((-1+\varepsilon,0)\times(\cD_{\varepsilon}\cap\overline{\cD}_j))$ and \eqref{volume}, we obtain
\begin{equation*}
\|\overline{\mathcal{A}^{\alpha\beta}}(y^{d})-{\mathcal{A}^{\alpha\beta}}(\cdot)\|_{L_1(Q_{r}^-(\Lambda z_0))}
\leq Nr^{d+\frac{5}{2}}.
\end{equation*}
From this together with \eqref{eq9.46}, Lemma \ref{lemlocbdd}, and the fact that $\mathbbm{1}_{_{(-1,0)\times\cD_j^c}}D_s \tilde\ell_{i,j}$ is piecewise $C^\mu$, it follows  that
\begin{align}\label{est-I11}
&\|I_{1}\|_{L_1(Q_{r}^-(\Lambda z_0))}\notag\\
&\leq\|\overline{\mathcal{A}^{\alpha\beta}}(y^{d})-{\mathcal{A}^{\alpha\beta}}(\cdot)\|_{L_1(Q_{r}^-(\Lambda z_0))}\nonumber\\
&\quad\cdot\|\ell_iD_{s}D_iu-D_{s}\mathfrak u-\sum_{j=1}^{m+1}\mathbbm{1}_{_{(-1,0)\times\cD_j^c}}D_s \tilde\ell_{i,j}D_iu(t_0,P_jx_0)\|_{L_\infty(Q_{r}^-(z_0))}\nonumber\\ &\leq Nr^{d+\frac{5}{2}}\Big(\sum_{j=1}^{m+1}\|D^2u\|_{L_\infty(Q_{r}^-(z_0)\cap((-1+\varepsilon,0)\times\cD_j))}+\sum_{j=1}^{M}|f|_{(1+\delta)/2,1+\delta;(-1,0)\times\overline{\cD_{j}}}\nonumber\\
&\qquad\quad+\|Du\|_{L_{1}(\cQ)}\Big).
\end{align}
By using a similar argument that led to \eqref{estDellk}, we have
\begin{align}\label{est-I12}
\|I_{2}\|_{L_1(Q_{r}^-(\Lambda z_0))}
&\leq N\sum_{j=1}^{m+1}\big\|D\ell_i(D_iu-D_iu(t_0,P_jx_0))\big\|_{L_1(Q_{r}^-(z_0))\cap((-1+\varepsilon,0)\times\cD_j)}\nonumber\\
&\leq Nr^{d+\frac{5}{2}}\sum_{j=1}^{m+1}[Du]_{1/2,1;Q_{r}^-(z_0)\cap((-1+\varepsilon,0)\times\overline\cD_j)}.
\end{align}
Now coming back to \eqref{Af1} and using \eqref{est-I11} and \eqref{est-I12}, we obtain
\begin{align}\label{estADtildeu}
\|(\overline{{\mathcal A}^{\alpha\beta}}(y^{d})-\mathcal{A}^{\alpha\beta}(\cdot))D_{y^\beta}\tilde{\mathfrak u}(\cdot;\Lambda z_0)+\tilde {\mathfrak f}_1^\alpha(\cdot;\Lambda z_0)\|_{L_1(Q_{r}^-(\Lambda z_0))}
\leq Nr^{d+\frac{5}{2}}\mathcal{C}_0.
\end{align}
Since $\tilde {\mathfrak f}_2^\alpha\in C^{\mu'/2,\mu'}((-1+\varepsilon,0)\times\overline{\cD_{j}})$, the estimate \eqref{estAbarf} also holds for $\tilde {\mathfrak f}_2^\alpha$ and thus
\begin{align}\label{estDf}
\int_{Q_{r}^-(\Lambda z_0)}\big|\tilde {\mathfrak f}_2^\alpha(t,y)-\bar{\mathfrak f}_2^\alpha(y^d)\big|
\leq Nr^{d+2+\mu'}\bigg( \sum_{j=1}^{M}|f|_{(1+\delta)/2,1+\delta;(-1,0)\times\overline{\cD_{j}}}+\|Du\|_{L_{1}(\cQ)}\bigg),
\end{align}
where $\mu'=\min\big\{\frac{1}{2},\mu\big\}$.
Combining  \eqref{estADtildeu} and \eqref{estDf}, we have \eqref{estF}.

Next we show the estimate of $\|G\|_{L_{1}(Q_{r}^-(\Lambda z_0))}$. We obtain from  \eqref{estDellk0} that
\begin{align*}
\|D\ell\|_{L_{1}(Q_{r}^{-}(z_{0})\cap((-1+\varepsilon,0)\times\cD_j))}
\leq Nr^{d+\frac{3}{2}}.
\end{align*}
Then we have
\begin{align*}
&\|G\|_{L_{1}(Q_{r}^-(\Lambda z_0))}\\
&\leq N r^{d+\frac{3}{2}}\Big(\sum_{j=1}^{m+1}\|D^2u\|_{L_\infty(Q_{r}^-(z_0)\cap((-1+\varepsilon,0)\times\cD_j))}\\
&\qquad\quad +\sum_{j=1}^{M}|f|_{(1+\delta)/2,1+\delta;(-1,0)\times\overline{\cD_{j}}}
+\|Du\|_{L_{1}(\cQ)}\Big).
\end{align*}
The lemma is proved.
\end{proof}

Now we are ready to prove Proposition \ref{lemma iteraphi}.
\begin{proof}[Proof of Proposition \ref{lemma iteraphi}.]
Recall that $v$ in Lemma \ref{weak est barv} satisfies
\begin{align*}
\begin{cases}
-v_{t}+D_{\alpha}(\widetilde{{\mathcal A}^{\alpha\beta}}D_{\beta}v)=G\mathbbm{1}_{Q_{r/2}^-(\Lambda z_0)}+\Div(F\mathbbm{1}_{Q_{r/2}^-(\Lambda z_0)})&\ \mbox{in}~Q_{r}^-(\Lambda z_0),\\
v=0&\ \mbox{on}~\partial_p Q_{r}^-(\Lambda z_0),
\end{cases}
\end{align*}
where $F$ and $G$ are given in \eqref{def-Falpha} and \eqref{def-G}, respectively. Similar to \cite[(4.8)]{dx2021}, by using Lemmas \ref{weak est barv} and \ref{lemmaFG}, we have
\begin{align}\label{holder v}
\left(\fint_{Q_{r/2}^-(\Lambda z_0)}|Dv|^{\frac{1}{2}}\,dz\right)^2\leq Nr^{\mu'}\mathcal{C}_0,
\end{align}
where  $\mathcal{C}_0$ is defined in \eqref{defC0}.
Let
\begin{equation}\label{defw u1}
u_{1}(y^{d})=\int_{\Lambda x_{0}^{d}}^{y^{d}}(\overline{\mathcal A^{dd}}(s))^{-1}\bar{\mathfrak f}_2^{d}(s)\,ds\quad\text{and}\quad w:=w(z;\Lambda z_0)=\tilde{\mathfrak u}(z;\Lambda z_0)-u_{1}-v,
\end{equation}
where $\bar{\mathfrak f}_2^{d}$ is the piecewise constant function corresponding to $\tilde {\mathfrak f}_2^d$ defined in \eqref{deff1f2}. Then $w$ satisfies
$$-w_{t}+D_{\alpha}(\overline {\mathcal A^{\alpha\beta}}(y^d)D_{\beta}w)=0\quad\mbox{in}~Q_{r/2}^-(\Lambda z_0).$$
Since the coefficient $\overline {\mathcal A^{\alpha\beta}}(y^d)$ only depends on $y^d$, for any $\kappa\in(0,1/2)$ to be fixed later, by Lemma \ref{lemma xn}, we have
\begin{align}\label{DW kappa}
&\|D_{y^{k'}}w(\cdot;\Lambda z_0)-(D_{y^{k'}}w)_{Q_{\kappa r}^{-}(\Lambda z_{0})}\|_{L_{1/2}(Q_{\kappa r}^{-}(\Lambda z_{0}))}^{1/2}\notag\\
&\quad +\|W(\cdot;\Lambda z_0)-(W)_{Q_{\kappa r}^{-}(\Lambda z_{0})}\|_{L_{1/2}(Q_{\kappa r}^{-}(\Lambda z_{0}))}^{1/2}\nonumber\\
&\leq N(\kappa r)^{d+5/2}\left([D_{y^{k'}}w]_{1/2,1;Q_{r/4}^{-}(\Lambda z_{0})}^{1/2}
+[W]_{1/2,1;Q_{r/4}^{-}(\Lambda z_{0})}^{1/2}\right)\nonumber\\
&\leq N\kappa^{d+5/2}\int_{Q_{r/2}^{-}(\Lambda z_{0})}|(D_{y^{k'}}w(z;\Lambda z_0),W(z;\Lambda z_0))|^{1/2}\,dz,
\end{align}
where $W=\overline{\mathcal A^{d\beta}}(y^d)D_{y^\beta}w(z;\Lambda z_0)$. Define
\begin{equation*}
h(y^{d}):=\int_{0}^{y^{d}}\Big(\overline{\mathcal A^{dd}}(s)\Big)^{-1}
\Big(\mathbf{Q}-\sum_{\beta=1}^{d-1}\overline{\mathcal A^{d\beta}}(s)\mathbf{q}^{\beta}\Big)\,ds
\end{equation*}
and
\begin{align*}
\tilde{w}:=w-\sum_{\beta=1}^{d-1}y^{\beta}\mathbf{q}^{\beta}-h(y^{d}),
\end{align*}
where $\mathbf q^{\beta},\mathbf Q\in\mathbb R^{n}$, $\beta=1,\ldots,d-1$.
Then $$-\tilde w_{t}+D_{\alpha}(\overline {\mathcal A^{\alpha\beta}}(y^d)D_{\beta}\tilde w)=0\quad\mbox{in}~ Q_{r/2}^{-}(\Lambda z_{0}),$$
and
$$D_{y^k}\tilde{w}=D_{y^{k'}}w-\mathbf{q}^{k'},\quad \widetilde{W}:=\overline{\mathcal A^{d\beta}}(y^{d})D_{\beta}\tilde{w}=W-\mathbf{Q}.$$
Replacing $w$ and $W$ with $\tilde{w}$ and $\widetilde{W}$ in \eqref{DW kappa}, respectively, we have
\begin{align*}
&\|D_{y^{k'}}w(\cdot;\Lambda z_0)-(D_{y^{k'}}w)_{Q_{\kappa r}^{-}(\Lambda z_{0})}\|_{L_{1/2}(Q_{\kappa r}^{-}(\Lambda z_{0}))}^{1/2}\\
&\quad +\|W(\cdot;\Lambda z_0)-(W)_{Q_{\kappa r}^{-}(\Lambda z_{0})}\|_{L_{1/2}(Q_{\kappa r}^{-}(\Lambda z_{0}))}^{1/2}\nonumber\\
&\leq N\kappa^{d+5/2}\int_{Q_{r/2}^{-}(\Lambda z_{0})}|(D_{y^{k'}}w(z;\Lambda z_0)-\mathbf{q}^{k'},W(z;\Lambda z_0)-\mathbf{Q})|^{1/2}\,dz,
\end{align*}
which implies
\begin{align}\label{diff-wW}
&\Bigg(\fint_{Q_{\kappa r}^-(\Lambda z_0)}\big(|D_{y^{k'}}w(z;\Lambda z_0)-(D_{y^{k'}}w)_{Q_{\kappa r}^{-}(\Lambda z_{0})}|^{\frac{1}{2}}+|W(z;\Lambda z_0)-(W)_{Q_{\kappa r}^{-}(\Lambda z_{0})}|^{\frac{1}{2}}\big)\,dz\Bigg)^{2}\nonumber\\
&\leq N\kappa\left(\fint_{Q_{r/2}^-(\Lambda z_0)}\big(|D_{y^{k'}}w(z;\Lambda z_0)-\mathbf{q}^{k'}|^{\frac{1}{2}}+|W(z;\Lambda z_0)-\mathbf Q|^{\frac{1}{2}}\big)\,dz\right)^{2}.
\end{align}
By the identity $\tilde{\mathfrak u}=w+u_1+v$ and \eqref{defw u1}, we have
\begin{align*}
D_{y^{k'}}\tilde{\mathfrak u}=D_{y^{k'}}w+D_{y^{k'}} v
\end{align*}
and
\begin{align*}
&\mathcal A^{d\beta}D_{y^\beta}\tilde{\mathfrak u}-\tilde {\mathfrak f}^d\\
&=(\mathcal A^{d\beta}-\overline{\mathcal A^{d\beta}}(y^d))D_{y^\beta}\tilde{\mathfrak u}+\overline{\mathcal A^{d\beta}}(y^d)D_{y^\beta}w+\overline{\mathcal A^{d\beta}}(y^d)D_{y^\beta}v+\bar{\mathfrak f}_2^d(y^d)-\tilde {\mathfrak f}^d\\
&=\overline{\mathcal A^{d\beta}}(y^d)D_{y^\beta}w+\overline{\mathcal A^{d\beta}}(y^d)D_{y^\beta}v-F^d,
\end{align*}
where $F^d=(\overline{\mathcal A^{d\beta}}(y^d)-\mathcal A^{d\beta})D_{y^\beta}\tilde{\mathfrak u}-\bar{\mathfrak f}_2^d(y^d)+\tilde {\mathfrak f}^d$.
Then we have from the triangle inequality, \eqref{holder v}, and \eqref{diff-wW}  that
\begin{align*}
&\Bigg(\fint_{Q_{\kappa r}^-(\Lambda z_0)}\big(|D_{y^{k'}}\tilde{\mathfrak u}(z;\Lambda z_0)-(D_{y^{k'}}w)_{Q_{\kappa r}^{-}(\Lambda z_{0})}|^{\frac{1}{2}}\nonumber\\
&\qquad\qquad\quad+|\mathcal A^{d\beta}D_{y^\beta}\tilde{\mathfrak u}(z;\Lambda z_0)-\tilde {\mathfrak f}^d(z;\Lambda z_0)-(W)_{Q_{\kappa r}^{-}(\Lambda z_{0})}|^{\frac{1}{2}}\big)\,dz\Bigg)^{2}\nonumber\\
&\leq N\kappa\left(\fint_{Q_{r/2}^-(\Lambda z_0)}\big(|D_{y^{k'}}\tilde{\mathfrak u}(z;\Lambda z_0)-\mathbf{q}^{k'}|^{\frac{1}{2}}+|\mathcal A^{d\beta}D_{y^\beta}\tilde{\mathfrak u}(z;\Lambda z_0)-\tilde {\mathfrak f}^d(z;\Lambda z_0)-\mathbf Q|^{\frac{1}{2}}\big)\,dz\right)^{2}\nonumber\\
&\quad+N\kappa^{-2(d+2)}\left(\fint_{Q_{r/2}^-(\Lambda z_0)}|F^d(z;\Lambda z_0)|^{\frac{1}{2}}\,dz\right)^{2}+N\kappa^{-2(d+2)}r^{\mu'}\mathcal{C}_0.
\end{align*}
Using \eqref{estF}, we deduce
\begin{align*}
&\Bigg(\fint_{Q_{\kappa r}^-(\Lambda z_0)}\big(|D_{y^{k'}}\tilde{\mathfrak u}(z;\Lambda z_0)-(D_{y^{k'}}w)_{Q_{\kappa r}^{-}(\Lambda z_{0})}|^{\frac{1}{2}}\nonumber\\
&\qquad\qquad\quad+|\mathcal A^{d\beta}D_{y^\beta}\tilde{\mathfrak u}(z;\Lambda z_0)-\tilde {\mathfrak f}^d(z;\Lambda z_0)-(W)_{Q_{\kappa r}^{-}(\Lambda z_{0})}|^{\frac{1}{2}}\big)\,dz\Bigg)^{2}\nonumber\\
&\leq N\kappa\Big(\fint_{Q_{r/2}^-(\Lambda z_0)}\big(|D_{y^{k'}}\tilde{\mathfrak u}(z;\Lambda z_0)-\mathbf{q}^{k'}|^{\frac{1}{2}}\nonumber\\
&\qquad\quad+|\mathcal A^{d\beta}D_{y^\beta}\tilde{\mathfrak u}(z;\Lambda z_0)-\tilde {\mathfrak f}^d(z;\Lambda z_0)-\mathbf Q|^{\frac{1}{2}}\big)\,dz\Big)^{2}+N\kappa^{-2(d+2)}r^{\mu'}\mathcal{C}_0.
\end{align*}
Since $\mathbf{q}^{k'}, \mathbf{Q}\in\mathbb R^{n}$ are arbitrary, we deduce that
\begin{align*}
\phi(\Lambda z_0,\kappa r)\leq N_{0}\kappa\phi(\Lambda z_0,r/2)+N\kappa^{-2(d+2)}r^{\mu'}\mathcal{C}_0.
\end{align*}
Choosing $\kappa\in(0,1/2)$  small enough so that $N_{0}\kappa\leq\kappa^{\gamma}$ for any fixed $\gamma\in(\mu',1)$ and iterating, we have
\begin{align*}
\phi(\Lambda z_0,\kappa^{j}r)\leq\kappa^{j\mu'}\phi(\Lambda z_0,r/2)+N(\kappa^{j}r)^{\mu'}\mathcal{C}_0.
\end{align*}
Hence, we have for any $\rho$ with $0<\rho\leq r\leq1/4$ and $\kappa^j r\leq\rho<\kappa^{j-1}$,
\begin{align*}
\phi(\Lambda z_0,\rho)\leq N\Big(\frac{\rho}{r}\Big)^{\mu'}\phi(\Lambda z_0,r/2)+N\rho^{\mu'}\mathcal{C}_0.
\end{align*}
The lemma is proved.
\end{proof}

\subsection{Piecewise regularity of \texorpdfstring{$u_t$}{} and the boundedness of \texorpdfstring{$[Du]_{t;(1+\delta)/2}$}{} and \texorpdfstring{$\|D^2u\|_{L_\infty}$}{}}\label{sublocalbound}
This section is devoted to the proof of the piecewise regularity of $u_t$ and the boundedness of $[Du]_{t;(1+\delta)/2}$ and $\|D^2u\|_{L_\infty}$.

\begin{lemma}\label{lemut}
Under the same assumptions as in Proposition \ref{proptildeu}, we have
\begin{align*}
&\sum_{j=1}^{m+1}[Du]_{t,(1+\delta)/2;(-1+3\varepsilon,0)\times (B_{1-2\varepsilon}\cap\overline{\cD_{j}})}+\sum_{j=1}^{m+1}|u_t|_{1/4,1/2;(-1+3\varepsilon,0)\times (B_{1-2\varepsilon}\cap\overline{\cD_{j}})}\\
&\leq N\|Du\|_{L_{1}(\cQ)}+N\sum_{j=1}^{M}|f|_{(1+\delta)/2,1+\delta;(-1,0)\times\overline{\cD_{j}}},
\end{align*}
where $N>0$ is a constant depending only on $n,d,m,p,\nu,\varepsilon$, $|A|_{(1+\delta)/2,1+\delta;(-1,0)\times\overline{\cD_{j}}}$, and the $C^{2+\mu}$ norm of $h_j$.
\end{lemma}

\begin{proof}
Define
\begin{equation}\label{def-deltaf}
\delta_h^{\gamma}f(t,x):=\frac{f(t,x)-f(t-h,x)}{h^{^\gamma}},
\end{equation}
where $\gamma\in\big(0,\frac{1+\delta}{2}\big)$ and $h\in(0,\varepsilon)$.
Then from \eqref{systems}, we have
\begin{equation}\label{systemut}
-(\delta_h^\gamma u)_t+D_\alpha(A^{\alpha\beta}D_\beta \delta_h^\gamma u)=D_\alpha \delta_h^\gamma f^\alpha-D_\alpha(\delta_h^\gamma A^{\alpha\beta}D_\beta u(t-h,x)).
\end{equation}
For any $(t_1,x_1),(t_2,x_2)\in (-1+2\varepsilon,0)\times (B_{1-\varepsilon}\cap\overline{{\cD}_{j}})$ with $(t_1,x_1)\neq(t_2,x_2)$, we have
\begin{align*}
\frac{|\delta_h^{\gamma}f^\alpha(t_1,x_1)-\delta_h^{\gamma}f^{\alpha}(t_2,x_2)|}{|t_1-s_1|^{\frac{1+\delta}{2}-\gamma}+|x_1-x_2|^{1+\delta-2\gamma}}
\leq N [f^\alpha]_{(1+\delta)/2,1+\delta;(-1+\varepsilon,0)\times (B_{1-\varepsilon}\cap\overline{{\cD}_{j}})}.
\end{align*}
This means $\delta_h^{\gamma}f^\alpha\in C^{\frac{1+\delta}{2}-\gamma,1+\delta-2\gamma}((-1+2\varepsilon,0)\times(B_{1-\varepsilon}\cap\overline{{\cD}_{j}}))$. Similarly, we get $\delta_h^{\gamma}A^{\alpha\beta}\in C^{\frac{1+\delta}{2}-\gamma,1+\delta-2\gamma}((-1+2\varepsilon,0)\times(B_{1-\varepsilon}\cap\overline{{\cD}_{j}}))$. Then by applying Lemma \ref{lemlocbdd} to \eqref{systemut}, we get $\delta_h^{\gamma}u\in C^{(1+\delta_1)/2,1+\delta_1}((-1+3\varepsilon,0)\times (B_{1-2\varepsilon}\cap\overline{{\cD}_{j}}))$ with $\delta_1:=\min\{1+\delta-2\gamma,\frac{1}{2}\}>0$. Moreover,
\begin{equation}
                \label{eq11.38}
|\delta_h^{\gamma}u|_{(1+\delta_1)/2,1+\delta_1;(-1+3\varepsilon,0)\times (B_{1-2\varepsilon}\cap\overline{{\cD}_{j}})}\leq N\big(\sum_{j=1}^{M}|f|_{(1+\delta)/2,1+\delta;(-1,0)\times\overline{\cD_{j}}}+\|Du\|_{L_{1}(\cQ)}\big).
\end{equation}
Therefore, we obtain for any fixed $x\in B_{1-2\varepsilon}\cap\overline{{\cD}_{j}}$ and $\gamma\in \big(\frac{1}{4},\frac{1+\delta}{2}\big)$,
\begin{equation}\label{est-ut}
u_t(\cdot,x)\in C^{\delta_2/2}((-1+\varepsilon,0)),\quad \delta_2:=\delta_1+2\gamma-1=\min\big\{\delta,2\gamma-\frac{1}{2}\big\},
\end{equation}
and satisfies
\begin{equation*}
|u_t|_{t,\delta_2/2;(-1+\varepsilon,0)}\leq N\big(\sum_{j=1}^{M}|f|_{(1+\delta)/2,1+\delta;(-1,0)\times\overline{\cD_{j}}}+\|Du\|_{L_{1}(\cQ)}\big).
\end{equation*}
We also obtain from $\delta_h^{\gamma}u\in C^{(1+\delta_1)/2,1+\delta_1}((-1+3\varepsilon,0)\times (B_{1-2\varepsilon}\cap\overline{\cD_{j}}))$ that
\begin{equation*}
D\delta_h^{\gamma}u\in C^{\delta_1/2,\delta_1}((-1+3\varepsilon,0)\times (B_{1-2\varepsilon}\cap\overline{\cD_{j}})).
\end{equation*}
If we take $\gamma$ to be sufficiently close to $\frac{1+\delta}{2}$, then $\delta_1=1+\delta-2\gamma$, $\gamma+\frac{\delta_1}{2}=\frac{1+\delta}{2}$, and
\begin{equation*}
[Du]_{t,(1+\delta)/2;(-1+3\varepsilon,0)\times (B_{1-2\varepsilon}\cap\overline{\cD_{j}})}\leq N\big(\sum_{j=1}^{M}|f|_{(1+\delta)/2,1+\delta;(-1,0)\times\overline{\cD_{j}}}+\|Du\|_{L_{1}(\cQ)}\big).
\end{equation*}

On the other hand, for any $(t,x_1),(t,x_2)\in (-1+3\varepsilon,0)\times (B_{1-2\varepsilon}\cap\overline{{\cD}_{j}})$ with $x_1\neq x_2$, by using the triangle inequality, Taylor's formula, \eqref{eq11.38}, and \eqref{est-ut}, we have
\begin{align*}
&|u_t(t,x_1)-u_t(t,x_2)|\\
&\leq |u_t(t,x_1)-\delta_hu(t,x_1)|+|u_t(t,x_2)-\delta_hu(t,x_2)|+|\delta_hu(t,x_1)-\delta_hu(t,x_2)|\nonumber\\
&\leq 2h^{\frac{\delta_2}{2}}[u_t]_{t,\delta_2/2;(-1+3\varepsilon,0)\times (B_{1-2\varepsilon}\cap\overline{{\cD}_{j}})}
+h^{\gamma-1}|\delta_h^{\gamma}u(t,x_1)-\delta_h^{\gamma}u(t,x_2)|\\
&\leq 2h^{\frac{\delta_2}{2}}[u_t]_{t,\delta_2/2;(-1+3\varepsilon,0)\times (B_{1-2\varepsilon}\cap\overline{{\cD}_{j}})}\\
&\quad +Nh^{\gamma-1}|x_1-x_2||\delta_h^\gamma u|_{(1+\delta_1)/2,1+\delta_1;(-1+3\varepsilon,0)\times (B_{1-2\varepsilon}\cap\overline{{\cD}_{j}})}.
\end{align*}
where $\delta_hu:=\delta_h^1u$; the definition of $\delta_h^1u$ can be found in \eqref{def-deltaf}. We obtain by optimizing in $h$ that
\begin{align*}
&|u_t(t,x_1)-u_t(t,x_2)|\\
&\leq N|x_1-x_2|^{\frac{\delta_2}{2}\cdot\frac{1}{1+\delta_2/2-\gamma}}
[u_t]_{t,\delta_2/2;(-1+3\varepsilon,0)\times (B_{1-2\varepsilon}\cap\overline{{\cD}_{j}})}^{\frac{1-\gamma}{1+\delta_2/2-\gamma}}\\
&\quad \cdot|\delta_h^\gamma u|_{(1+\delta_1)/2,1+\delta_1;(-1+3\varepsilon,0)\times (B_{1-2\varepsilon}\cap\overline{{\cD}_{j}})}^{\frac{\delta_2/2}{1+\delta_2/2-\gamma}}.
\end{align*}
Now letting $\gamma=\frac{1+\delta}{2}-\varepsilon_0$ for any small constant $0<\varepsilon_0\leq \delta-\frac{1}{2}$, then $\delta_2=\delta$ and $$\frac{\delta_2}{2}\cdot\frac{1}{1+\frac{\delta_2}{2}-\gamma}=\frac{\delta}{1+2\varepsilon_0}\geq\frac{1}{2},$$
where we used the assumption that $\delta>\frac{1}{2}$. We thus obtain
\begin{equation*}
|u_t|_{1/4,1/2;(-1+3\varepsilon,0)\times (B_{1-2\varepsilon}\cap\overline{{\cD}_{j}})}\leq N\big(\sum_{j=1}^{M}|f|_{(1+\delta)/2,1+\delta;(-1,0)\times\overline{\cD_{j}}}+\|Du\|_{L_{1}(\cQ)}\big).
\end{equation*}
The lemma is proved.
\end{proof}

\begin{lemma}\label{lemma Dtildeu}
Under the same assumptions as in Proposition \ref{proptildeu},
 we have
\begin{align*}
\sum_{j=1}^{m+1}\|D^2u\|_{L_\infty(Q_{1/4}^-\cap((-1+\varepsilon,0)\times\cD_j))}\leq N\|Du\|_{L_{1}(Q_{3/4}^{-})}+N\sum_{j=1}^{M}|f|_{(1+\delta)/2,1+\delta;(-1,0)\times\overline{\cD_{j}}},
\end{align*}
where $N>0$ is a constant depending only on
$n,d,m,p,\nu,\varepsilon$, $|A|_{(1+\delta)/2,1+\delta;(-1,0)\times\overline{\cD_{j}}}$, and the $C^{2+\mu}$ norm of $h_j$.
\end{lemma}

\begin{proof}
We prove this lemma in two steps.

{\bf Step 1.} We first prove the claim: For any $k'=1,\ldots,d-1$, $z_0\in (-1+\varepsilon,0)\times (\cD_{\varepsilon}\cap{{\cD}_{j}})$, and $r\in(0,1/4)$,
\begin{align}\label{estDtildeuU}
&|D_{\ell^{k'}}\tilde u(z_0;z_0)|+|\tilde U(z_0;z_0)|\nonumber\\
&\leq Nr^{\mu'}\sum_{j=1}^{m+1}\|D^2u\|_{L_\infty(Q_{r}^-(z_0)\cap((-1+\varepsilon,0)\times\cD_j))}\notag\\
&\quad +Nr^{-1}\big(\sum_{j=1}^{M}|f|_{(1+\delta)/2,1+\delta;(-1,0)\times\overline{\cD_{j}}}+\|Du\|_{L_{1}(\cQ)}\big),
\end{align}
where $\tilde u$ and $\tilde U$ are defined in \eqref{def-tildeu} and \eqref{deftildeU}, respectively, and  $\mu'=\min\big\{\frac{1}{2},\mu\big\}$. Indeed, for any $s\in(0,1)$, let $\mathbf q_{z_{0},s}^{k'},\mathbf Q_{z_{0},s}\in\mathbb R^n$  be chosen such that
$$\Phi(z_{0},s)=\left(\fint_{Q_{s}^{-}(z_{0})}\big(|D_{\ell^{k'}}\tilde u(z;z_0)-\mathbf q_{z_{0},s}^{k'}|^{\frac{1}{2}}+|\tilde U(z;z_0)-\mathbf Q_{z_{0},s}|^{\frac{1}{2}}\big)\,dz\right)^{2}.$$
Using the triangle inequality, we have
\begin{align*}
|\mathbf q_{z_{0},s/2}^{k'}-\mathbf q_{z_{0},s}^{k'}|^{\frac{1}{2}}\leq|D_{\ell^{k'}}\tilde u(z;z_0)-\mathbf q_{z_{0}, s/2}^{k'}|^{\frac{1}{2}}+|D_{\ell^{k'}}\tilde u(z;z_0)-\mathbf q_{z_{0},s}^{k'}|^{\frac{1}{2}}
\end{align*}
and
\begin{align*}
|\mathbf Q_{z_{0},s/2}-\mathbf Q_{z_{0},s}|^{\frac{1}{2}}\leq|\tilde U(z;z_0)-\mathbf Q_{z_{0},s/2}|^{\frac{1}{2}}+|\tilde U(z;z_0)-\mathbf Q_{z_{0},s}|^{\frac{1}{2}}.
\end{align*}
Now taking the average over $z\in Q_{s/2}^{-}(z_{0})$ and then taking the square, we obtain
\begin{align*}
|\mathbf q_{z_{0},s/2}^{k'}-\mathbf q_{z_{0},s}^{k'}|+|\mathbf Q_{z_{0},s/2}-\mathbf Q_{z_{0},s}|\leq N(\Phi(z_{0},s/2)+\Phi(z_{0},s)).
\end{align*}
By iterating and using the triangle inequality, we deduce
\begin{equation}\label{itera}
|\mathbf q_{z_{0},2^{-L}s}^{k'}-\mathbf q_{z_{0},s}^{k'}|+|\mathbf Q_{z_{0},2^{-L} s}-\mathbf Q_{z_{0},s}|\leq N\sum_{j=0}^{L}\Phi(z_{0},2^{-j} s).
\end{equation}
Recalling the definitions of $\tilde u$ and $u_{_\ell}$ in \eqref{def-tildeu} and \eqref{tildeu}, respectively,  a direct calculation gives
\begin{align}\label{Dellk'tildeu}
D_{\ell^{k'}}\tilde u(z;z_0)=\ell_i^k\ell_j^{k'}D_iD_ju+D_{\ell^{k'}}\ell_i D_iu-\sum_{j=1}^{m+1}D_{\ell^{k'}}\tilde\ell_{i,j} D_iu(t_0,P_jx_0)-D_{\ell^{k'}}{\mathfrak u}.
\end{align}
Using \eqref{def-tildefalpha}, we have
\begin{align}\label{tildeU}
&\tilde U(z;z_0)=n^\alpha(A^{\alpha\beta}D_\beta\tilde u-\tilde f^\alpha)\nonumber\\
&=n^\alpha\Big(A^{\alpha\beta}D_\beta D_i u\ell_i^k-D_{\ell}f^\alpha+D_\ell A^{\alpha\beta}D_\beta u-A^{\alpha\beta}D_\beta{\mathfrak u}\nonumber\\
&\quad-\delta_{\alpha d}\sum_{j=1}^{m}\mathbbm{1}_{x^d>h_j(x')} (n^d_j(x'))^{-1}\tilde h_j(t,x')-A^{\alpha\beta}\sum_{j=1}^{m+1}\mathbbm{1}_{_{(-1,0)\times\cD_j^c}}D_\beta \tilde\ell_{i,j}D_iu(t_0,P_jx_0)\Big).
\end{align}
From the assumption $Du$ is piecewise $C^1$ in $x$, $A^{\alpha\beta},f^\alpha\in C^{(1+\delta)/2,1+\delta}((-1+\varepsilon,0)\times (\cD_{\varepsilon}\cap\overline{{\cD}_{j}}))$, \eqref{eq9.46}, and Lemma \ref{lemell} (ii), it follows that $D_{\ell^{k'}}\tilde u(z;z_0),\tilde U(z;z_0)\in C^0((-1+\varepsilon,0)\times (\cD_{\varepsilon}\cap\overline{{\cD}_{j}}))$.
Taking $\rho=2^{-L} s$ in \eqref{est phi'}, we have
$$\lim_{L\rightarrow\infty}\Phi(z_{0},2^{-L} s)=0.$$
Thus, for any $z_0\in (-1+\varepsilon,0)\times (\cD_{\varepsilon}\cap{\cD}_{j})$,  we obtain
\begin{equation*}
\lim_{L\rightarrow\infty}\mathbf q_{z_{0},2^{-L}s}^{k'}=D_{\ell^{k'}}\tilde u(z_{0};z_0),\quad \lim_{L\rightarrow\infty}\mathbf Q_{z_{0},2^{-L} s}=\tilde U(z_{0};z_0).
\end{equation*}
Now taking $L\rightarrow\infty $ in \eqref{itera}, choosing $s=r/2$, and using Lemma \ref{lemma itera}, we have for $r\in(0,1/4)$, $k'=1,\ldots,d-1$, and $z_0\in (-1+\varepsilon,0)\times (\cD_{\varepsilon}\cap{{\cD}_{j}})$,
\begin{align}\label{diffDtildeu}
&|D_{\ell^{k'}}\tilde u(z_{0};z_0)-\mathbf q_{z_{0},r/2}^{k'}|+|\tilde U(z_{0};z_0)-\mathbf Q_{z_{0},r/2}|\leq N\sum_{j=0}^{\infty}\Phi(z_{0},2^{-j-1}r)\nonumber\\
&\leq N\Phi(z_{0},r/2)+Nr^{\mu'}\Big(\sum_{j=1}^{m+1}
[Du]_{1/2,1;Q_{r}^-(z_0)\cap((-1+\varepsilon,0)\times\overline{\cD_j})}\nonumber\\
&\quad
+\sum_{j=1}^{M}|f|_{(1+\delta)/2,1+\delta;(-1,0)\times\overline{\cD_{j}}}+\|Du\|_{L_{1}(\cQ)}\Big).
\end{align}
By averaging the inequality
\begin{align*}
&|\mathbf q_{z_{0},r/2}^{k'}|+|\mathbf Q_{z_{0},r/2}|\\
&\leq |D_{\ell^{k'}}\tilde u(z;z_0)-\mathbf q_{z_{0},r/2}^{k'}|+|\tilde U(z;z_0)-\mathbf Q_{z_{0},r/2}|+|D_{\ell^{k'}}\tilde u(z;z_0)|+|\tilde U(z;z_0)|
\end{align*}
over $z\in Q_{r/2}^{-}(z_{0})$ and then taking the square, we have
\begin{align*}
&|\mathbf q_{z_{0},r/2}^{k'}|+|\mathbf Q_{z_{0},r/2}|\\
&\leq N\Phi(z_0,r/2)+N\left(\fint_{Q_{r/2}^{-}(z_{0})}\big(|D_{\ell^{k'}}\tilde u(z;z_0)|^{\frac{1}{2}}+|\tilde U(z;z_0)|^{\frac{1}{2}}\big)\,dz\right)^{2}.
\end{align*}
This, in combination with \eqref{diffDtildeu}, Lemmas \ref{lemut}, \ref{lemma Dtildeu}, the triangle inequality, and
\begin{align*}
\Phi(z_{0},r/2)&\leq \left(\fint_{Q_{r/2}^{-}(z_{0})}\big(|D_{\ell^{k'}}\tilde u(z;z_0)|^{\frac{1}{2}}+|\tilde U(z;z_0)|^{\frac{1}{2}}\big)\,dz\right)^{2}\\
&\leq Nr^{-d-2}\Big(\|D_{\ell^{k'}}\tilde u(\cdot;z_0)\|_{L_{1}(Q_{r/2}^{-}(z_{0}))}+\|\tilde U(\cdot;z_0)\|_{L_{1}(Q_{r/2}^{-}(z_{0}))}\Big),
\end{align*}
leads to that
\begin{align}\label{Dx'Uz0}
&|D_{\ell^{k'}}\tilde u(z_0;z_0)|+|\tilde U(z_0;z_0)|\nonumber\\
&\leq Nr^{-d-2}\Big(\|D_{\ell^{k'}}\tilde u(\cdot;z_0)\|_{L_{1}(Q_{r/2}^{-}(z_{0}))}+\|\tilde U(\cdot;z_0)\|_{L_{1}(Q_{r/2}^{-}(z_{0}))}\Big)\nonumber\\
&\quad+Nr^{\mu'}\Big(\sum_{j=1}^{m+1}\|D^2u\|_{L_\infty(Q_{r}^-(z_0)\cap((-1+\varepsilon,0)\times\cD_j))}\notag\\
&\qquad\qquad+\sum_{j=1}^{M}|f|_{(1+\delta)/2,1+\delta;(-1,0)\times\overline{\cD_{j}}}+\|Du\|_{L_{1}(\cQ)}\Big).
\end{align}

Next we estimate $\|D_{\ell^{k'}}\tilde u(\cdot;z_0)\|_{L_{1}(Q_{r/2}^{-}(z_{0}))}$ and $\|\tilde U(\cdot;z_0)\|_{L_{1}(Q_{r/2}^{-}(z_{0}))}$ on the right-hand side above. Using the definition of weak solutions, the transmission problem \eqref{homosecond} is equivalent to
$$-\partial_tu_{_\ell}+D_\alpha(A^{\alpha\beta}D_\beta u_{_\ell})=g+ D_\alpha f_3^\alpha\quad\mbox{in}~Q_1^-.$$
Then by Lemma \ref{lem loc lq} with a suitable scaling, we obtain
\begin{align*}
&\|Du_{_\ell}(\cdot;z_0)\|_{L_2(Q_{r/2}^-(z_0))}\\
&\leq N\big(r^{-1-\frac{d+2}{2}}\|u_{_\ell}(\cdot;z_0)\|_{L_1(Q_{r}^-(z_0))}+r\|g\|_{L_2(Q_{r}^-(z_0))}+\|f_3(\cdot;z_0)\|_{L_2(Q_{r}^-(z_0))}\big).
\end{align*}
By using H\"{o}lder's inequality, we have
\begin{align}\label{Duell_L1}
\|Du_{_\ell}(\cdot;z_0)\|_{L_1(Q_{r/2}^-(z_0))}&\leq N\big(r^{-1}\|u_{_\ell}(\cdot;z_0)\|_{L_1(Q_{r}^-(z_0))}+r^{1+\frac{d+2}{2}}\|g\|_{L_2(Q_{r}^-(z_0))}\nonumber\\
&\quad+r^{\frac{d+2}{2}}\|f_3(\cdot;z_0)\|_{L_2(Q_{r}^-(z_0))}\big).
\end{align}
From the definition of $u_{_\ell}$ in \eqref{tildeu} and the local boundedness of $Du$ in Lemma \ref{lemlocbdd}, we get
\begin{equation}\label{est-uell}
\|u_{_\ell}(\cdot;z_0)\|_{L_1(Q_{r}^-(z_0))}\leq Nr^{d+2}\Big(\|Du\|_{L_{1}(\cQ)}+\sum_{j=1}^{M}|f|_{(1+\delta)/2,1+\delta;(-1,0)\times\overline{\cD_{j}}}\Big).
\end{equation}
Using Lemma \ref{lemell}(iii) , we have
\begin{align*}
\int_{B_r(x_0)\cap\cD_i}|D\ell^k|^2\ dx\leq N\int_{B_r'(x'_0)}\frac{\min\{2r,h_i-h_{i-1}\}}{|h_{i}-h_{i-1}|}\ dx'\leq Nr^{d-1}.
\end{align*}
Together with this, we have from \eqref{defg} and \eqref{tildef1} that
\begin{align}\label{est-g1}
\|g\|_{L_2(Q_{r}^-(z_0))}&\leq Nr^{\frac{d+1}{2}}\Big(\sum_{j=1}^{m+1}\|D^2u\|_{L_\infty(Q_{r}^-(z_0)\cap((-1+\varepsilon,0)\times\cD_j))}+\|Du\|_{L_{1}(\cQ)}\nonumber\\
&\qquad\quad+\sum_{j=1}^{M}|f|_{(1+\delta)/2,1+\delta;(-1,0)\times\overline{\cD_{j}}}\Big)
\end{align}
and
\begin{align}\label{est-f1breve}
\|f_3(\cdot;z_0)\|_{L_2(Q_{r}^-(z_0))}&\leq Nr^{\frac{d+3}{2}}\sum_{j=1}^{m+1}[Du]_{1/2,1;Q_{r}^-(z_0)\cap((-1+\varepsilon,0)\times\overline\cD_j)}\nonumber\\ &\quad+Nr^{\frac{d+2}{2}}\Big(\|Du\|_{L_{1}(\cQ)}+\sum_{j=1}^{M}|f|_{(1+\delta)/2,1+\delta;(-1,0)\times\overline{\cD_{j}}}\Big).
\end{align}
Substituting \eqref{est-uell}--\eqref{est-f1breve}  into \eqref{Duell_L1}, and using Lemma \ref{lemut}, we have
\begin{align*}
\|Du_{_\ell}(\cdot;z_0)\|_{L_1(Q_{r/2}^-(z_0))}&\leq Nr^{d+\frac{5}{2}}\sum_{j=1}^{m+1}\|D^2u\|_{L_\infty(Q_{r}^-(z_0)\cap((-1+\varepsilon,0)\times\cD_j))}\\
&\quad+Nr^{d+1}\Big(\|Du\|_{L_{1}(\cQ)}+\sum_{j=1}^{M}|f|_{(1+\delta)/2,1+\delta;(-1,0)\times\overline{\cD_{j}}}\Big).
\end{align*}
Now we obtain from the definition of $\tilde u$ in \eqref{def-tildeu}  and  \eqref{eq9.46} that
\begin{align*}
\|D\tilde u(\cdot;z_0)\|_{L_1(Q_{r/2}^-(z_0))}&\leq\|Du_{_\ell}(\cdot;z_0)\|_{L_1(Q_{r/2}^-(z_0))}+\|D{\mathfrak u}\|_{L_1(Q_{r/2}^-(z_0))}\nonumber\\
&\leq Nr^{d+\frac{5}{2}}\sum_{j=1}^{m+1}\|D^2u\|_{L_\infty(Q_{r}^-(z_0)\cap((-1+\varepsilon,0)\times\cD_j))}\nonumber\\
&\quad+Nr^{d+1}\Big(\|Du\|_{L_{1}(\cQ)}+\sum_{j=1}^{M}|f|_{(1+\delta)/2,1+\delta;(-1,0)\times\overline{\cD_{j}}}\Big).
\end{align*}
Similarly, we have
\begin{align*}
\|\tilde U(\cdot;z_0)\|_{L_1(Q_{r/2}^-(z_0))}
&\leq Nr^{d+\frac{5}{2}}\sum_{j=1}^{m+1}\|D^2u\|_{L_\infty(Q_{r}^-(z_0)\cap((-1+\varepsilon,0)\times\cD_j))}\\
&\quad+Nr^{d+1}\Big(\|Du\|_{L_{1}(\cQ)}+\sum_{j=1}^{M}|f|_{(1+\delta)/2,1+\delta;(-1,0)\times\overline{\cD_{j}}}\Big).
\end{align*}
Coming back to \eqref{Dx'Uz0}, we obtain \eqref{estDtildeuU}.

{\bf Step 2.} We finish the proof of the desired result.
Since $k,k'=1,\ldots,d-1$, there are $\frac{nd(d-1)}{2}$ equations in \eqref{Dellk'tildeu} and $n(d-1)$ equations in \eqref{tildeU}. Thus we have $\frac{n(d-1)(d+2)}{2}$ equations, while $D^2u$ has $\frac{nd(d+1)}{2}$ components. A simple calculation gives
$$
\frac{nd(d+1)}{2}-\frac{n(d-1)(d+2)}{2}=n,
$$
which implies that we have to consider $n$ more equations. To this end,  we  rewrite the equation \eqref{systems} as
\begin{equation}\label{Dalpbetau}
A^{\alpha\beta}D_{\alpha\beta}u=D_\alpha f^\alpha-D_\alpha A^{\alpha\beta}D_{\beta}u+u_t
\end{equation}
in $(-1+\varepsilon,0)\times(B_{1-\varepsilon}\cap\cD_j)$, $j=1,\ldots,M$. We shall use Cramer's rule and  \eqref{Dellk'tildeu}, \eqref{tildeU}, and \eqref{Dalpbetau} to solve for $D^{2}u$. However, it is not easy to verify that whether the determinant of the coefficient matrix of
$$(D_{1}^2u,\ldots,D_{1}D_{d} u,D_{2}^2 u,\ldots,D_{2}D_{d} u,\ldots,D_{d}^2 u)^\top$$
is not equal to $0$. For this,  we introduce the linear transformation $\Lambda$ from the coordinate system associated with $0$ to the coordinate system associated with the fixed point $x\in B_r(x_0)$, given by $\Lambda^k=\ell^k(x)$ and $\Lambda^d=n(x)$, $k=1,\ldots,d-1$, where $\ell^k(x)$ and $n(x)$ are defined in \eqref{defell} and \eqref{defnorm}, respectively. Now we define
$$y=\Lambda x,\quad\hat u(t,y)=u(t,x),\quad \hat A(t,y)=\Lambda A(t,x)\Lambda^\top.$$
Then
\begin{equation*}
\ell_i^k(x)\ell_j^{k'}(x)D_iD_ju(t,x)=D_{k}D_{k'}\hat u(t,y),\quad
n^\alpha A^{\alpha\beta}D_\beta D_i u\ell_i^k=\hat{A}^{d\beta}(t,y)D_{\beta}D_{k}\hat u(t,y),
\end{equation*}
and
\begin{equation*}
A^{\alpha\beta}D_{\alpha\beta}u=\hat A^{\alpha\beta}(t,y)D_{\alpha\beta}\hat u(t,y).
\end{equation*}
A direct calculation yields the  determinant of the coefficient matrix of
$$(D_{1}^2\hat u,\ldots,D_{1}D_{d}\hat u,D_{2}^2\hat u,\ldots,D_{2}D_{d}\hat u,\ldots,D_{d}^2\hat u)^\top$$
is $(\hat A^{dd})^{d}>0$. Therefore, by using Cramer's rule, we can solve for $D^2\hat u$ and thus $D^{2}u$. In particular, by using Lemma \ref{lemlocbdd} and \eqref{deftildeh}, we have
\begin{align*}
|D^2u(z_0)|&\leq N\Big(|D_{\ell^{k'}}\tilde u(z_0;z_0)|+|\tilde U(z_0;z_0)|+|u_t(z_0)|+|D{\mathfrak u}(z_0)|\Big)\\
&\quad+N\big(\sum_{j=1}^{M}|f|_{(1+\delta)/2,1+\delta;(-1,0)\times\overline{\cD_{j}}}+\|Du\|_{L_{1}(\cQ)}\big).
\end{align*}
Now combining \eqref{estDtildeuU}, Lemma \ref{lemut},  and \eqref{eq9.46}, we deduce
\begin{align*}
|D^2u(z_0)|
&\leq Nr^{\mu'}\sum_{j=1}^{m+1}\|D^2u\|_{L_\infty(Q_{r}^-(z_0)\cap((-1+\varepsilon,0)\times\cD_j))}\\
&\quad+Nr^{-1}\big(\sum_{j=1}^{M}|f|_{(1+\delta)/2,1+\delta;(-1,0)\times\overline{\cD_{j}}}+\|Du\|_{L_{1}(\cQ)}\big).
\end{align*}
For any $z_1\in Q_{1/4}^-$ and $r\in(0,1/4)$, by taking supremum with respect to $z_0\in Q_r^-(z_1)\cap ((-1+\varepsilon,0)\times\cD_j)$, we have
\begin{align*}
&\sum_{j=1}^{m+1}\|D^2u\|_{L_\infty(Q_{r}^-(z_1)\cap((-1+\varepsilon,0)\times\cD_j))}\nonumber\\
&\leq Nr^{\mu'}\sum_{j=1}^{m+1}\|D^2u\|_{L_\infty(Q_{2r}^-(z_1)\cap((-1+\varepsilon,0)\times\cD_j))}\\
&\quad +Nr^{-1}\Big(\sum_{j=1}^{M}|f|_{(1+\delta)/2,1+\delta;(-1,0)\times\overline{\cD_{j}}}+\|Du\|_{L_{1}(\cQ)}\Big),
\end{align*}
where $\mu'=\min\big\{\frac{1}{2},\mu\big\}$. By an iteration argument which is essentially the same as that in \cite[Lemma 3.4]{dx2019}, we get
\begin{align*}
\sum_{j=1}^{m+1}\|D^2u\|_{L_\infty(Q_{1/4}^-\cap((-1+\varepsilon,0)\times\cD_j))}\leq N\|Du\|_{L_{1}(Q_{3/4}^{-})}+N\sum_{j=1}^{M}|f|_{(1+\delta)/2,1+\delta;(-1,0)\times\overline{\cD_{j}}},
\end{align*}
and the lemma is proved.
\end{proof}

\subsection{Proof of Proposition \ref{proptildeu}}\label{prfprop}
We are in a position to finish the proof of Proposition \ref{proptildeu}.
\begin{proof}[{\bf Proof of Proposition \ref{proptildeu}}]
(a)  By using Lemmas \ref{lemma itera}, \ref{lemut}, \ref{lemma Dtildeu}, and \ref{lemmeanDu} with $Q_{1/4}^-$ in place of $Q_r^-(z_0)$, we have for $r\in (0,1/8)$,
\begin{align}\label{estsupphi}
\sup_{z_{0}\in Q_{1/8}^{-}}\Phi(z_{0},r)&\leq Nr^{\mu'}\Big(\|D\tilde u\|_{L_1(Q_{1/4}^-)}+\|\tilde f\|_{L_1(Q_{1/4}^-)}+\sum_{j=1}^{m+1}[Du]_{1/2,1;Q_{1/4}^-\cap((-1+\varepsilon,0)\times\overline\cD_j)}\nonumber\\
&\quad+\sum_{j=1}^{M}|f|_{(1+\delta)/2,1+\delta;(-1,0)\times \overline{\cD_{j}}}+\|Du\|_{L_{1}(\cQ)}\Big)\nonumber\\
&\leq Nr^{\mu'}\Big(\sum_{j=1}^{M}|f|_{(1+\delta)/2,1+\delta;(-1,0)\times \overline{\cD_{j}}}+\|Du\|_{L_{1}(\cQ)}\Big).
\end{align}
Applying \eqref{diffDtildeu} with $r$ in place of $r/2$  and using \eqref{estsupphi}, we derive
\begin{align}\label{estDk'DdU}
&\sup_{z_{0}\in Q_{1/8}^{-}}\big(|D_{\ell^{k'}}\tilde u(z_{0};z_0)-\mathbf q_{z_{0},r}^{k'}|+|\tilde U(z_{0};z_0)-\mathbf Q_{z_{0},r}|\big)\nonumber\\
&\leq Nr^{\mu'}\Big(\sum_{j=1}^{M}|f|_{(1+\delta)/2,1+\delta;(-1,0)\times \overline{\cD_{j}}}+\|Du\|_{L_{1}(\cQ)}\Big),
\end{align}
where $\mathbf q_{z_{0},r}^{k'},\mathbf Q_{z_{0},r}\in\mathbb R^n$  satisfy	
$$
\Phi(z_{0},r)=\left(\fint_{Q_{r}^{-}(z_{0})}\big(|D_{\ell^{k'}}\tilde u(z;z_0)-\mathbf q_{z_{0},r}^{k'}|^{\frac{1}{2}}+|\tilde U(z;z_0)-\mathbf Q_{z_{0},r}|^{\frac{1}{2}}\big)\,dz\right)^{2}.
$$

Suppose that $z_{1}=(t_1,x_1)\in Q_{1/8}^{-}\cap((-1+\varepsilon,0)\times\cD_{j_{1}})$ for some $j_{1}\in\{1,\ldots,m+1\}$. If $|z_{0}-z_{1}|_{p}\geq1/16$, then we  obtain from \eqref{Dellk'tildeu}  that
\begin{align}\label{Dellk'tildeu00}
&D_{\ell^{k'}}\tilde u(z_0;z_0)\notag\\
&=\ell_i^k(x_0)\ell_j^{k'}(x_0)D_iD_ju(z_0)-\sum_{j=1,j\neq j_0}^{m+1}D_{\ell^{k'}}\tilde\ell_{i,j}(x_0) D_iu(t_0,P_jx_0)-D_{\ell^{k'}}{\mathfrak u}(z_0).
\end{align}
Thus, by using Lemma \ref{lemlocbdd}, Lemma \ref{lemma Dtildeu}, and \eqref{eq9.46}, we have
\begin{align}
&|D_{\ell^{k'}}\tilde u(z_{0};z_0)-D_{\ell^{k'}}\tilde u(z_{1};z_1)|\notag\\
&\leq N\sum_{j=1}^{m+1}\|D^2u\|_{L_\infty(Q_{1/4}^-\cap((-1+\varepsilon,0)\times\cD_j))}+N\|D{\mathfrak u}\|_{L_\infty(Q_{1/4}^-)}\notag\\
&\quad+N\big(\|Du\|_{L_{1}(\cQ)}+\sum_{j=1}^{M}
|f|_{(1+\delta)/2,1+\delta;(-1,0)\times\overline{\cD_{j}}}\big)\notag\\
\label{esttildeu}
&\leq |z_{0}-z_{1}|_{p}^{\mu'}\Big(\sum_{j=1}^{M}|f|_{(1+\delta)/2,1+\delta;(-1,0)\times \overline{\cD_{j}}}+\|Du\|_{L_{1}(\cQ)}\Big).
\end{align}
From \eqref{tildeU}, it follows that
\begin{align}\label{tildeU00}
\tilde U(z_0;z_0)&=n^\alpha(x_0)\Big(A^{\alpha\beta}(z_0)D_\beta D_i u(z_0)\ell_i^k(x_0)-D_{\ell}f^\alpha(z_0)+D_\ell A^{\alpha\beta}(z_0)D_\beta u(z_0)\nonumber\\
&\quad-A^{\alpha\beta}(z_0)D_\beta{\mathfrak u}(z_0)-\delta_{\alpha d}\sum_{j=1}^{m}\mathbbm{1}_{x^d>h_j(x')} (n^d_j(x'_0))^{-1}\tilde h_j(t_0,x'_0)\nonumber\\
&\quad-A^{\alpha\beta}(z_0)\sum_{j=1,j\neq j_0}^{m+1}\mathbbm{1}_{_{(-1,0)\times\cD_j^c}}D_\beta \tilde\ell_{i,j}(x_0)D_iu(t_0,P_jx_0)\Big).
\end{align}
Using  a similar argument in deriving \eqref{esttildeu}, we obtain
\begin{align*}
|\tilde U(z_{0};z_0)-\tilde U(z_{1};z_1)|
\leq |z_{0}-z_{1}|_{p}^{\mu'}\Big(\sum_{j=1}^{M}|f|_{(1+\delta)/2,1+\delta;(-1,0)\times \overline{\cD_{j}}}+\|Du\|_{L_{1}(\cQ)}\Big).
\end{align*}

If $|z_{0}-z_{1}|_{p}<1/16$, then we set $r=|z_{0}-z_{1}|_{p}$. Without loss of generality, we assume that $z_1$ is above $z_0$.
By the triangle inequality, we have for any $z\in Q_{r}^{-}(z_{0})$,
\begin{equation}\label{differenceDu}
\begin{split}
&|D_{\ell^{k'}}\tilde u(z_{0};z_0)-D_{\ell^{k'}}\tilde u(z_{1};z_1)|^{\frac{1}{2}}+|\tilde U(z_{0};z_0)-\tilde U(z_{1};z_1)|^{\frac{1}{2}}\\
&\leq|D_{\ell^{k'}}\tilde u(z_{0};z_0)-\mathbf q_{z_{0},r}^{k'}|^{\frac{1}{2}}+|D_{\ell^{k'}}\tilde u(z;z_0)-\mathbf q_{z_{0},r}^{k'}|^{\frac{1}{2}}+|D_{\ell^{k'}}\tilde u(z;z_1)-\mathbf q_{z_{1},2r}^{k'}|^{\frac{1}{2}}\\
&\quad+|D_{\ell^{k'}}\tilde u(z;z_0)-D_{\ell^{k'}}\tilde u(z;z_1)|^{\frac{1}{2}}+|D_{\ell^{k'}}\tilde u(z_{1};z_1)-\mathbf q_{z_{1},2r}^{k'}|^{\frac{1}{2}}\\
&\quad+|\tilde U(z_{0};z_0)-\mathbf Q_{z_{0},r}|^{\frac{1}{2}}+|\tilde U(z;z_0)-\mathbf Q_{z_{0},r}|^{\frac{1}{2}}+|\tilde U(z;z_1)-\mathbf Q_{z_1,2r}|^{\frac{1}{2}}\\
&\quad+|\tilde U(z;z_0)-\tilde U(z;z_1)|^{\frac{1}{2}}+|\tilde U(z_1;z_1)-\mathbf Q_{z_{1},2r}|^{\frac{1}{2}},
\end{split}
\end{equation}
where $\mathbf q_{z_{1},2r}^{k'},\mathbf Q_{z_{1},2r}\in\mathbb R^n$, $k'=1,\ldots,d-1$, satisfy	
$$
\Phi(z_{1},2r)=\left(\fint_{Q_{2r}^{-}(z_{1})}\big(|D_{\ell^{k'}}\tilde u(z;z_1)-\mathbf q_{z_{1},2r}^{k'}|^{\frac{1}{2}}+|\tilde U(z;z_1)-\mathbf Q_{z_{1},2r}|^{\frac{1}{2}}\big)\,dz\right)^{2}.
$$
Now we take the average over $z\in Q_{r}^{-}(z_{0})$ and take the  square in \eqref{differenceDu} to obtain
\begin{equation}\label{diffDtildeuU}
\begin{split}
&|D_{\ell^{k'}}\tilde u(z_{0};z_0)-D_{\ell^{k'}}\tilde u(z_{1};z_1)|+|\tilde U(z_{0};z_0)-\tilde U(z_{1};z_1)|\\
&\leq|D_{\ell^{k'}}\tilde u(z_0;z_0)-\mathbf q_{z_{0},r}^{k'}|+|\tilde U(z_{0};z_0)-\mathbf Q_{z_{0},r}|+\Phi(z_0,r)+\Phi(z_1,2r)\\
&\quad+|D_{\ell^{k'}}\tilde u(z_{1};z_1)-\mathbf q_{z_{1},2r}^{k'}|+|\tilde U(z_1;z_1)-\mathbf Q_{z_{1},2r}|\\
&\quad+\left(\fint_{Q_{r}^{-}(z_{0})}\big(|D_{\ell^{k'}}\tilde u(z;z_0)-D_{\ell^{k'}}\tilde u(z;z_1)|^{\frac{1}{2}}+|\tilde U(z;z_0)-\tilde U(z;z_1)|^{\frac{1}{2}}\big)\,dz\right)^{2}.
\end{split}
\end{equation}

Next we estimate the last  term in \eqref{diffDtildeuU}. For any $x_0\in B_{1/8}\cap\overline\cD_{j_0}$ and $x_1\in B_{1/8}\cap\overline\cD_{j_1}$, since $r=|z_{0}-z_{1}|_{p}$,  by using \eqref{Pjx} and $h_j\in C^{2+\mu}$, we have
\begin{equation*}
|P_jx_0-P_jx_1|\leq N|x_0-x_1|.
\end{equation*}
This together with Lemmas \ref{lemut} and \ref{lemma Dtildeu} yields
\begin{align}\label{DuPjx}
&\left|Du(t_0,P_jx_0)-Du(t_1,P_jx_1)\right|\notag\\
&\leq Nr[Du]_{1/2,1;Q_{1/4}^-\cap((-1+\varepsilon,0)\times\overline\cD_j)}\nonumber\\
&\leq Nr\Big(\sum_{j=1}^{M}|f|_{(1+\delta)/2,1+\delta;(-1,0)\times \overline{\cD_{j}}}+\|Du\|_{L_{1}(\cQ)}\Big).
\end{align}
It follows from \eqref{eq-rmu} that $\tilde{\mathfrak u}_j:={\mathfrak u}_j(z;z_0)-{\mathfrak u}_j(z;z_1)\in\mathcal{H}_p^1(Q_1^-)$ satisfies
\begin{align*}
\begin{cases}
&-\partial_t\tilde{\mathfrak u}_j+D_{\alpha}(\boldsymbol{A}^{\alpha\beta}D_\beta \tilde{\mathfrak u}_j)\\
&=-D_{\alpha}(\mathbbm{1}_{_{(-1,0)\times\cD_j^c}}A^{\alpha\beta}D_\beta \tilde\ell_{i,j}(D_iu(t_0,P_jx_0)-D_iu(t_1,P_jx_1)))\quad \mbox{in}~Q_1^-,\\
&\tilde{\mathfrak u}_j=0\quad \mbox{on}~\partial_pQ_1^-.
\end{cases}
\end{align*}
Then by using Lemmas \ref{lemlocbdd}, \ref{solvability}, \eqref{DuPjx}, and the fact that $\mathbbm{1}_{_{(-1,0)\times\cD_j^c}}D_\beta \tilde\ell_{i,j}$ is piecewise $C^\mu$, we obtain
\begin{align*}
&|{\mathfrak u}(\cdot;z_0)-{\mathfrak u}(\cdot;z_1)|_{(1+\mu')/2,1+\mu';(-1+\varepsilon,0)\times(\overline{\cD_i}\cap B_{1-\varepsilon})}\nonumber\\
&\leq N\sum_{j=1}^{m+1}\|D\tilde{\mathfrak u}_j\|_{L_1(Q_{1}^-)}+N\sum_{j=1}^{m+1}|\mathbbm{1}_{_{(-1,0)\times\cD_j^c}}A^{\alpha\beta}D_\beta \tilde\ell_{i,j}(D_iu(t_0,P_jx_0)-D_iu(t_1,P_jx_1))|_{\mu/2,\mu;Q_1^-}\nonumber\\
&\leq N\sum_{j=1}^{m+1}\|\mathbbm{1}_{_{(-1,0)\times\cD_j^c}}A^{\alpha\beta}D_\beta \tilde\ell_{i,j}(D_iu(t_0,P_jx_0)-D_iu(t_1,P_jx_1))\|_{L_p(Q_{1}^-)}\nonumber\\
&\quad+Nr\big(\|Du\|_{L_{1}(\cQ)}+\sum_{j=1}^{M}|f|_{(1+\delta)/2,1+\delta;(-1,0)\times\overline{\cD_{j}}}\big)\nonumber\\
&\leq Nr\big(\|Du\|_{L_{1}(\cQ)}+\sum_{j=1}^{M}|f|_{(1+\delta)/2,1+\delta;(-1,0)\times\overline{\cD_{j}}}\big),
\end{align*}
where
$$
{\mathfrak u}(\cdot;z_0)-{\mathfrak u}(\cdot;z_1):=\sum_{j=1}^{m+1}\tilde{\mathfrak u}_j,\quad i=1,\ldots,m+1,
$$ and $\mu'=\min\{\mu,\frac{1}{2}\}$.
This together with  \eqref{def-tildeu}, \eqref{defu0}, and \eqref{DuPjx} yields
\begin{align}\label{diffDuz0z1}
&|D_{\ell^{k'}}\tilde u(z;z_0)-D_{\ell^{k'}}\tilde u(z;z_1)|\nonumber\\
&=\Big|\sum_{j=1}^{m+1}D_{\ell^{k'}}\tilde{\ell}_{i,j}\big(D_iu(t_0,P_jx_0)-D_iu(t_1,P_jx_1)\big)
+D_{\ell^{k'}}{\mathfrak u}(z;z_0)-D_{\ell^{k'}}{\mathfrak u}(z;z_1)\Big|\nonumber\\
&\leq Nr\Big(\sum_{j=1}^{M}|f|_{(1+\delta)/2,1+\delta;(-1,0)\times \overline{\cD_{j}}}+\|Du\|_{L_{1}(\cQ)}\Big).
\end{align}
Similarly,
\begin{align}\label{diffUz0z1}
|\tilde U(z;z_0)-\tilde U(z;z_1)|\leq Nr\Big(\sum_{j=1}^{M}|f|_{(1+\delta)/2,1+\delta;(-1,0)\times \overline{\cD_{j}}}+\|Du\|_{L_{1}(\cQ)}\Big).
\end{align}
Coming back to \eqref{diffDtildeuU}, and using \eqref{estsupphi}, \eqref{estDk'DdU},  \eqref{diffDuz0z1}, and \eqref{diffUz0z1}, we obtain
\begin{align*}
&|D_{\ell^{k'}}\tilde u(z_{0};z_0)-D_{\ell^{k'}}\tilde u(z_{1};z_1)|+|\tilde U(z_{0};z_0)-\tilde U(z_{1};z_1)|\\
&\leq Nr^{\mu'}\Big(\sum_{j=1}^{M}|f|_{(1+\delta)/2,1+\delta;(-1,0)\times \overline{\cD_{j}}}+\|Du\|_{L_{1}(\cQ)}\Big).
\end{align*}
	
(b) As showed in Step 2 of the proof of Lemma \ref{lemma Dtildeu},  from \eqref{Dalpbetau} with $z=z_0$, \eqref{Dellk'tildeu00}, and \eqref{tildeU00}, we find that $D^{2}u(z_0)$ is a combination of
\begin{equation}\label{uteq}
D_\alpha f^\alpha(z_0)-D_\alpha A^{\alpha\beta}(z_0)D_{\beta}u(z_0)+u_t(z_0),
\end{equation}
\begin{equation}\label{Dtildeueq}
D_{\ell^{k'}}\tilde u(z_0;z_0)+\sum_{j=1,j\neq j_0}^{m+1}D_{\ell^{k'}}\tilde\ell_{i,j}(x_0) D_iu(t_0,P_jx_0)+D_{\ell^{k'}}{\mathfrak u}(z_0),
\end{equation}
and
\begin{align}\label{tildeUeq}
&\tilde U(z_0;z_0)+n^\alpha(x_0)\Big(D_{\ell}f^\alpha(z_0)-D_\ell A^{\alpha\beta}(z_0)D_\beta u(z_0)+A^{\alpha\beta}(z_0)D_\beta{\mathfrak u}(z_0)\nonumber\\
&\,+\delta_{\alpha d}\sum_{j=1}^{m}\mathbbm{1}_{x^d>h_j(x')} (n^d_j(x'_0))^{-1}\tilde h_j(t_0,x'_0)\notag\\
&\,+A^{\alpha\beta}(z_0)\sum_{j=1,j\neq j_0}^{m+1}\mathbbm{1}_{_{(-1,0)\times\cD_j^c}}D_\beta \tilde\ell_{i,j}(x_0)D_iu(t_0,P_jx_0)\Big).
\end{align}
Similarly, for any $\tilde z_{0}\in (-1+\varepsilon,0)\times (B_{1-\varepsilon}\cap\overline{\cD}_{j_0})$, $D^{2}u(\tilde z_0)$ is a combination of \eqref{uteq}--\eqref{tildeUeq} with $z_0$ replaced with $\tilde z_0$. Combining  \eqref{holdertildeuU}, Lemma \ref{lemut}, and \eqref{eq9.46}, we have
$$[D^2u]_{\mu'/2,\mu';(-1+\varepsilon,0)\times (B_{1-\varepsilon}\cap\overline{\cD}_{j_0})}
\le N\Big(\sum_{j=1}^{M}|f|_{(1+\delta)/2,1+\delta;(-1,0)\times \overline{\cD_{j}}}+\|Du\|_{L_{1}(\cQ)}\Big)$$
for any $j_0=1,\ldots,m+1$. Proposition \ref{proptildeu} is proved.
\end{proof}

\section{General  \texorpdfstring{$C^{(s+1+\mu')/2,s+1+\mu'}$}{} estimates}\label{secgeneral}
In this section, we prove Theorem \ref{Mainthm} in the general case of $s\geq2$  by showing the key points and the main ingredients. We prove by induction on $s$ that if $A^{\alpha\beta}$ and $f^\alpha$ are piecewise $C^{(s-1+\delta)/2,s-1+\delta}$  and the interfacial boundaries are $C^{s+\mu}$, then any $\mathcal{H}_{p}^{1}(\cQ)$ weak solution  $u$ to \eqref{systems} is piecewise $C^{(s+\mu')/2,s+\mu'}$, with the estimate
\begin{align}\label{pieceu}
&|u|_{(s+\mu')/2,s+\mu';(-1+\varepsilon,0)\times (\cD_{\varepsilon}\cap\overline{{\cD}_{j_0}})}\notag\\
&\leq N\Big(\|Du\|_{L_{1}((-1,0)\times \cD)}+\sum_{j=1}^{M}|f|_{(s-1+\delta)/2,s-1+\delta;(-1,0)\times \overline{\cD_{j}}}\Big),
\end{align}
where $j_0=1,\ldots,m+1$, $\mu'=\min\big\{\frac{1}{2},\mu\big\}$, $N$ depends on $n,d,m,p,\nu,\varepsilon$, the $C^{s+\mu}$ characteristic of $\cD_{j}$, and $|A|_{(s-1+\delta)/2,s-1+\delta;(-1,0)\times\overline{\cD_{j}}}$.
Now suppose  that $A^{\alpha\beta}$ and  $f^{\alpha}$ are of the class $C^{(s+\delta)/2,s+\delta}((-T+\varepsilon,0)\times(\cD_{\varepsilon}\cap\overline{\cD}_j))$, and the interfacial boundaries are $C^{s+1+\mu}$. We shall prove that $u$ is piecewise $C^{(s+1+\mu')/2,s+1+\mu'}$.

As in \eqref{eq6.26}, we will first derive a new equation satisfied by $D_{\ell}\cdots D_{\ell}u$ which is the $s$-th directional derivative of $u$ along the vector fields $\ell^k,k=1,\ldots,d-1$. For this,  since $D_\ell(fg)=gD_\ell f+fD_\ell g$, by an induction argument for $s\geq 2$, we conclude
\begin{align}\label{Dlu}
D_{\ell}\cdots D_{\ell}u
=\ell_{i_1}\ell_{i_2}\cdots\ell_{i_{_{s}}}D_{i_1}D_{i_2}\cdots D_{i_{_{s}}}u+R(u),
\end{align}
where we used the Einstein summation convention over repeated indices, $\ell_{i_{\tau}}:=\ell_{i_{\tau}}^{k_{\tau}}$, $\tau=1,\ldots,s$, $k_{\tau}=1,\ldots,d-1$, $i_{\tau}=1,\ldots,d$, and
\begin{align}\label{defLu}
R(u)&=D_{\ell_{i_1}}(\ell_{i_{2}}\cdots\ell_{i_{_{s}}})D_{i_{2}}\cdots D_{i_{_{s}}}u\nonumber\\
&\quad+D_{\ell_{i_1}}\Bigg(D_{\ell_{i_2}}(\ell_{i_{3}}\cdots\ell_{i_{_{s}}})D_{i_{3}}\cdots D_{i_{_{s}}}u+D_{\ell_{i_2}}\Big(D_{\ell_{i_3}}(\ell_{i_{4}}\cdots\ell_{i_{_{s}}})D_{i_{4}}\cdots D_{i_{_{s}}}u\nonumber\\
&\qquad\qquad+D_{\ell_{i_3}}\big(D_{\ell_{i_4}}(\ell_{i_{5}}\cdots\ell_{i_{_{s}}})D_{i_{5}}\cdots D_{i_{_{s}}}u+\cdots+D_{\ell_{i_{s-2}}}(D_{\ell_{i_{s-1}}}\ell_{i_s}D_{i_s}u)\big)\Big)\Bigg),
\end{align}
which is the summation of the products of directional derivatives of $\ell$ and derivatives of $u$.
Furthermore, an induction argument gives
\begin{align*}
D_{i_1}D_{i_2}\cdots D_{i_{_{s}}}(A^{\alpha\beta}D_\beta u)
&=A^{\alpha\beta}D_\beta D_{i_1}D_{i_2}\cdots D_{i_{_{s}}}u+D_{i_1}A^{\alpha\beta}D_\beta D_{i_2}\cdots D_{i_{_{s}}}u\\
&\quad+\sum_{\tau=1}^{s-1}D_{i_1}\cdots D_{i_{\tau}}(D_{i_{\tau+1}}A^{\alpha\beta}D_\beta D_{i_{\tau+2}}\cdots D_{i_{_{s}}}u).
\end{align*}
Then we have
\begin{align}\label{lDu}
&\ell_{i_1}\ell_{i_2}\cdots\ell_{i_{_{s}}}D_{i_1}D_{i_2}\cdots D_{i_{_{s}}} D_\alpha(A^{\alpha\beta}D_\beta u)\nonumber\\
&=D_\alpha\big(\ell_{i_1}\ell_{i_2}\cdots\ell_{i_{_{s}}}D_{i_1}D_{i_2}\cdots D_{i_{_{s}}} (A^{\alpha\beta}D_\beta u)\big)-D_\alpha(\ell_{i_1}\ell_{i_2}\cdots\ell_{i_{_{s}}})D_{i_1}D_{i_2}\cdots D_{i_{_{s}}}(A^{\alpha\beta}D_\beta u)\nonumber\\
&=D_\alpha\big(A^{\alpha\beta}D_\beta(\ell_{i_1}\ell_{i_2}\cdots\ell_{i_{_{s}}}D_{i_1}D_{i_2}\cdots D_{i_{_{s}}}u)\big)-D_\alpha\big(A^{\alpha\beta}D_\beta(\ell_{i_1}\ell_{i_2}\cdots\ell_{i_{_{s}}})D_{i_1}D_{i_2}\cdots D_{i_{_{s}}}u\big)\nonumber\\
&\quad+D_\alpha\Big(\ell_{i_1}\ell_{i_2}\cdots\ell_{i_{_{s}}}\big(D_{i_1}A^{\alpha\beta}D_\beta D_{i_2}\cdots D_{i_{_{s}}}u+\sum_{\tau=1}^{s-1}D_{i_1}\cdots D_{i_{\tau}}(D_{i_{\tau+1}}A^{\alpha\beta}D_\beta D_{i_{\tau+2}}\cdots D_{i_{_{s}}}u)\big)\Big)\nonumber\\
&\quad-D_\alpha(\ell_{i_1}\ell_{i_2}\cdots\ell_{i_{_{s}}})D_{i_1}D_{i_2}\cdots D_{i_{_{s}}}(A^{\alpha\beta}D_\beta u)
\end{align}
and
\begin{align}\label{lDf}
&\ell_{i_1}\ell_{i_2}\cdots\ell_{i_{_{s}}}D_{i_1}D_{i_2}\cdots D_{i_{_{s}}} D_\alpha f^\alpha\nonumber\\
&=D_\alpha(\ell_{i_1}\ell_{i_2}\cdots\ell_{i_{_{s}}}D_{i_1}D_{i_2}\cdots D_{i_{_{s}}} f^\alpha)-D_\alpha(\ell_{i_1}\ell_{i_2}\cdots\ell_{i_{_{s}}})D_{i_1}D_{i_2}\cdots D_{i_{_{s}}} f^\alpha.
\end{align}
Taking $D_\ell\cdots D_{\ell}$ to the equation $-u_t+D_\alpha(A^{\alpha\beta}D_\beta u)=D_\alpha f^\alpha$, and using \eqref{Dlu}, \eqref{lDu}, and \eqref{lDf}, we obtain the equation
\begin{align*}
-(D_{\ell}\cdots D_{\ell}u)_t+D_\alpha\big(A^{\alpha\beta}D_\beta (D_{\ell}\cdots D_{\ell}u)\big)=D_\alpha \hat f_1^\alpha+\breve g
\end{align*}
in each subdomain $(-1,0)\times \cD_j$, where
\begin{align}\label{deff2}
\hat f_1^\alpha&:=\ell_{i_1}\ell_{i_2}\cdots\ell_{i_{_{s}}}D_{i_1}D_{i_2}\cdots D_{i_{_{s}}}f^\alpha\notag\\
&\quad +A^{\alpha\beta}D_\beta(\ell_{i_1}\ell_{i_2}\cdots\ell_{i_{_{s}}})D_{i_1}D_{i_2}\cdots D_{i_{_{s}}}u+A^{\alpha\beta}D_\beta(R(u))\nonumber\\
&\quad-\ell_{i_1}\ell_{i_2}\cdots\ell_{i_{_{s}}}\big(D_{i_1}A^{\alpha\beta}D_\beta D_{i_2}\cdots D_{i_{_{s}}}u+\sum_{\tau=1}^{s-1}D_{i_1}\cdots D_{i_{\tau}}(D_{i_{\tau+1}}A^{\alpha\beta}D_\beta D_{i_{\tau+2}}\cdots D_{i_{_{s}}}u)\big)
\end{align}
and
\begin{align}\label{defg2}
\breve g:=D_\alpha(\ell_{i_1}\ell_{i_2}\cdots\ell_{i_{_{s}}})\big(D_{i_1}D_{i_2}\cdots D_{i_{_{s}}}(A^{\alpha\beta}D_\beta u-f^\alpha)\big)+R(D_\alpha(f^\alpha-A^{\alpha\beta}D_\beta u)).
\end{align}
Similarly, by taking $D_{\ell}\cdots D_\ell$ to $[n^\alpha_j(A^{\alpha\beta} D_\beta u -f^\alpha)]_{(-1,0)\times\Gamma_j}=0$, we derive the boundary condition
$$[n_j^\alpha (A^{\alpha\beta} D_\beta (D_{\ell}\cdots D_{\ell}u)- \hat f_1^\alpha)]_{(-1,0)\times\Gamma_j}=\breve h_j,$$
where
\begin{align}\label{brevehj}
\breve h_j&:=\Big[-\ell_{i_1}\ell_{i_2}\cdots\ell_{i_{_{s}}}\big(\sum_{\tau=1}^{s}D_{i_{\tau}}n_j^\alpha D_{i_1}\cdots D_{\tau_{s-1}}D_{i_{\tau+1}}\cdots D_{i_s}(A^{\alpha\beta}D_\beta u-f^\alpha)\nonumber\\
&\quad+\sum_{1\leq \tau_1<\tau_2\leq s}D_{i_{\tau_1}}D_{i_{\tau_2}}n_j^\alpha D_{i_1}\cdots D_{i_{\tau_1}-1}D_{i_{\tau_1}+1}\cdots D_{i_{\tau_2}-1}D_{i_{\tau_2}+1}\cdots D_{i_s}(A^{\alpha\beta}D_\beta u-f^\alpha)\nonumber\\
&\quad+\cdots+D_{i_1}D_{i_2}\cdots D_{i_s}n_j^\alpha (A^{\alpha\beta}D_\beta u-f^\alpha)\big)\Big]_{(-1,0)\times\Gamma_j}\notag\\
&\quad -[R(n_j^\alpha (A^{\alpha\beta} D_\beta u- f^\alpha))]_{(-1,0)\times\Gamma_j}.
\end{align}
Consequently, the $s$-th directional derivative $D_{\ell}\cdots D_{\ell}u$ satisfies
\begin{align}\label{eqDllu}
\begin{cases}
-(D_{\ell}\cdots D_{\ell}u)_t+D_\alpha\big(A^{\alpha\beta}D_\beta (D_{\ell}\cdots D_{\ell}u)\big)=D_\alpha \hat f_1^\alpha+\breve g\quad\text{in}\,\bigcup_{j=1}^{m+1}(-1,0)\times\cD_j, \\
[n_j^\alpha (A^{\alpha\beta} D_\beta (D_{\ell}\cdots D_{\ell}u)- \hat f_1^\alpha)]_{(-1,0)\times\Gamma_j}=\breve h_j.
\end{cases}
\end{align}
Note that the terms  $D_\beta(\ell_{i_1}\ell_{i_2}\cdots\ell_{i_{_{s}}})$ and $D_\beta(R(u))$ in \eqref{deff2} are singular at any point where two interfaces touch or are close to each other. To cancel out the singularity, we choose a function $u_0$ as follows:
\begin{align}\label{defu-0}
&u_0:=u_0(x;z_0)\notag\\
&=\sum_{j=1}^{m+1}\tilde\ell_{i_{1},j}\tilde\ell_{i_{2},j}\cdots\tilde\ell_{i_{s},j}D_{i_1}D_{i_2}\cdots D_{i_{_{s}}}u(t_0,P_jx_0)\nonumber\\
&\quad+\sum_{j=1}^{m+1}\sum_{\tau=1}^{s-1}D_{\tilde\ell_{i_1,j}}D_{\tilde\ell_{i_2,j}}\cdots D_{\tilde\ell_{i_\tau},j}(\tilde\ell_{i_{\tau+1},j}\cdots\tilde\ell_{i_{_{s}},j})\big(D_{i_{\tau+1}}\cdots D_{i_{_{s}}}u(t_0,P_jx_0)\nonumber\\
&\quad+(x-x_0)\cdot DD_{i_{\tau+1}}\cdots D_{i_{_{s}}}u(t_0,P_jx_0)\big)+\cdots\nonumber\\
&\quad+\sum_{j=1}^{m+1}(D_{\tilde\ell_{i_{s-1},j}}\tilde \ell_{i_s,j})\tilde\ell_{i_{1},j}\tilde\ell_{i_{2},j}\cdots\tilde\ell_{i_{s-2},j}\big(D_{i_1}D_{i_2}\cdots D_{i_{_{s-2}}}D_{i_{_{s}}}u(t_0,P_jx_0)\nonumber\\
&\quad+(x-x_0)\cdot DD_{i_{1}}D_{i_{2}}\cdots D_{i_{_{s-2}}}D_{i_{_{s}}}u(t_0,P_jx_0)\big),
\end{align}
where $P_jx_0$ is defined in \eqref{Pjx},
$(t_0,x_{0})\in (-9/16,0)\times (B_{3/4}\cap \cD_{j_{0}})$, $r\in(0,1/4)$, $\tilde\ell_{,j}$ is the smooth extension of $\ell|_{\cD_j}$ to $\cup_{k=1,k\neq j}^{m+1}\cD_k$. Denote
\begin{align*}
u^{\ell}:=u^{\ell}(z;z_0)=D_{\ell}\cdots D_{\ell}u-u_0.
\end{align*}
Then \eqref{eqDllu} is equivalent to
\begin{align}\label{eq Dbreveu00}
\begin{cases}
-\partial_tu^{\ell}+D_\alpha(A^{\alpha\beta}D_\beta u^{\ell})=D_\alpha (\hat f_1^\alpha-A^{\alpha\beta}D_\beta u_0)+\breve g\quad\text{in}\,\bigcup_{j=1}^{m+1}(-1,0)\times\cD_j, \\
[n_j^\alpha (A^{\alpha\beta} D_\beta u^{\ell}-\hat f_1^\alpha+A^{\alpha\beta}D_\beta u_0)]_{(-1,0)\times\Gamma_j}=\breve h_j.
\end{cases}
\end{align}
As before, by adding a term
$$
\sum_{j=1}^{m} D_d(\mathbbm{1}_{x^d>h_j(x')} (n^d_j(x'))^{-1}\breve h_j(t,x'))
$$
to the equation, the problem \eqref{eq Dbreveu00} becomes a (homogeneous) transmission problem
\begin{align*}
\begin{cases}
-\partial_t u^{\ell}+D_\alpha(A^{\alpha\beta}D_\beta u^{\ell})=D_\alpha \hat f_2^\alpha+\breve g\quad\text{in}\,\bigcup_{j=1}^{m+1}(-1,0)\times\cD_j, \\
[n_j^\alpha (A^{\alpha\beta} D_\beta u^{\ell}-\hat f_2^\alpha)]_{(-1,0)\times\Gamma_j}=0,
\end{cases}
\end{align*}
where
\begin{align*}
\hat f_2^\alpha:=\hat f_2^\alpha(z;z_0)=\hat f_1^\alpha-A^{\alpha\beta}D_\beta u_0+\delta_{\alpha d}\sum_{j=1}^m\mathbbm{1}_{x^d>h_j(x')} (n^d_j(x'))^{-1}\breve h_j(t,x').
\end{align*}
With the function $u_0$, we can bound the mean oscillation of $\hat f_2^\alpha$ in cylinders. To make it vanish in a certain order,  we shall further consider the problem
\begin{align}\label{eqmathsfu}
\begin{cases}
-\partial_t{\mathsf u}_j+D_\alpha(\boldsymbol{A}^{\alpha\beta}D_\beta {\mathsf u}_j)=-D_{\alpha}\big(A^{\alpha\beta}F_\beta\big)&\mbox{in}~Q_1^-,\\
{\mathsf u}_j=0& \mbox{on}~\partial_pQ_1^-,
\end{cases}
\end{align}
where ${\mathsf u}_j:={\mathsf u}_j(\cdot;z_0)\in \mathcal{H}_p^1(Q_1^-)$, the coefficient $\boldsymbol{A}^{\alpha\beta}$ is defined in \eqref{mathcalA}, and
\begin{align}\label{defmathF}
F_\beta&:=\sum_{j=1}^{m+1}\mathbbm{1}_{_{(-1,0)\times\cD_j^c}}D_\beta (\tilde\ell_{i_{1},j}\tilde\ell_{i_{2},j}\cdots\tilde\ell_{i_{s},j})D_{i_1}D_{i_2}\cdots D_{i_{_{s}}}u(t_0,P_jx_0)+\cdots\nonumber\\
&\quad+\sum_{j=1}^{m+1}\mathbbm{1}_{_{(-1,0)\times\cD_j^c}}D_\beta\big((D_{\tilde\ell_{i_{s-1},j}}\tilde \ell_{i_s,j})\tilde\ell_{i_{1},j}\tilde\ell_{i_{2},j}\cdots\tilde\ell_{i_{s-2},j}\big)\big(D_{i_1}D_{i_2}\cdots D_{i_{_{s-2}}}D_{i_{_{s}}}u(t_0,P_jx_0)\nonumber\\
&\quad\quad+(x-x_0)\cdot DD_{i_{1}}D_{i_{2}}\cdots D_{i_{_{s-2}}}D_{i_{_{s}}}u(t_0,P_jx_0)\big)\notag\\
&\quad+\sum_{j=1}^{m+1}\mathbbm{1}_{_{(-1,0)\times\cD_j^c}}
(D_{\tilde\ell_{i_{s-1},j}}\tilde \ell_{i_s,j})\tilde\ell_{i_{1},j}\tilde\ell_{i_{2},j}\cdots
\tilde\ell_{i_{s-2},j} D_\beta D_{i_{1}}D_{i_{2}}\cdots D_{i_{_{s-2}}}D_{i_{_{s}}}u(t_0,P_jx_0),
\end{align}
which is the summation of the products of $\mathbbm{1}_{_{(-1,0)\times\cD_j^c}}$ and derivatives of $u_0$ defined in \eqref{defu-0}.
Combining Lemma \ref{solvability} and  \eqref{pieceu}, we have
\begin{align}\label{estmathu}
\|{\mathsf u}_j\|_{\mathcal{H}_p^1(Q_1^-)}\leq N\|A^{\alpha\beta}F_\beta\|_{L_p(Q_1^-)}\leq N\big(\|Du\|_{L_{1}(\cQ)}+\sum_{j=1}^{M}|f|_{(s-1+\delta)/2,s-1+\delta;(-1,0)\times\overline{\cD_{j}}}\big).
\end{align}
Using the fact that $\tilde\ell_{,j}$  is the smooth extension  of $\ell|_{\cD_j}$ to $\cup_{k=1,k\neq j}^{m+1}\cD_k$, one can verify that the right-hand side of the equation in \eqref{eqmathsfu} is piecewise $C^{\mu/2,\mu}$. Then from Lemma \ref{lemlocbdd}, it follows that  ${\mathsf u}_j(\cdot;z_0)\in C^{(1+\mu')/2,1+\mu'}((-1+\varepsilon,0)\times(\overline{\cD_i}\cap B_{1-\varepsilon}))$, and satisfies
\begin{align*}
&|{\mathsf u}_j|_{(1+\mu')/2,1+\mu';(-1+\varepsilon,0)\times(\overline{\cD_i}\cap B_{1-\varepsilon})}\\
&\leq N\|D{\mathsf u}_j\|_{L_1(Q_{1}^-)}+N|A^{\alpha\beta}F_\beta|_{\mu/2,\mu;Q_1^-}\nonumber\\
&\leq N\big(\|Du\|_{L_{1}(\cQ)}+\sum_{j=1}^{M}|f|_{(s-1+\delta)/2,s-1+\delta;(-1,0)\times\overline{\cD_{j}}}\big),
\end{align*}
where we used \eqref{estmathu}, \eqref{defu-0}, and \eqref{pieceu} in the second inequality, $\mu'=\min\{\mu,\frac{1}{2}\}$ and $i=1,\ldots,m+1$.
Thus, we obtain
\begin{equation}\label{defsfu}
{\mathsf u}:={\mathsf u}(\cdot;z_0)=\sum_{j=1}^{m+1}{\mathsf u}_j(\cdot;z_0)\in C^{(1+\mu')/2,1+\mu'}((-1+\varepsilon,0)\times(\overline{\cD_i}\cap B_{1-\varepsilon}))
\end{equation}
and
\begin{align}\label{est-Dmathu}
&{|\mathsf u_j|_{(1+\mu')/2,1+\mu';(-1+\varepsilon,0)\times(\overline{\cD_i}\cap B_{1-\varepsilon})}}\notag\\
&\leq N\big(\|Du\|_{L_{1}(\cQ)}+\sum_{j=1}^{M}|f|_{(s-1+\delta)/2,s-1+\delta;(-1,0)\times\overline{\cD_{j}}}\big).
\end{align}
Define
\begin{equation}\label{defbreveu}
\breve u:=\breve u(z;z_0)=u^{\ell}-{\mathsf u}.
\end{equation}
Then $\breve u$ satisfies
\begin{equation}\label{eqbreveu}
-\breve u_t+D_\alpha(A^{\alpha\beta}D_\beta \breve u)=\breve g+ D_\alpha \breve f^\alpha\quad\mbox{in}~Q_{3/4}^-,
\end{equation}
where
\begin{align}\label{def-brevefalpha}
\breve f^\alpha&:=\breve f^\alpha(z;z_0)=\hat f_2^\alpha(z;z_0)+A^{\alpha\beta}F_\beta\nonumber\\
&=\ell_{i_1}\ell_{i_2}\cdots\ell_{i_{_{s}}}D_{i_1}D_{i_2}\cdots D_{i_{_{s}}}f^\alpha+A^{\alpha\beta}D_\beta(\ell_{i_1}\ell_{i_2}\cdots\ell_{i_{_{s}}})D_{i_1}D_{i_2}\cdots D_{i_{_{s}}}u+A^{\alpha\beta}D_\beta(R(u))\nonumber\\
&\quad-\ell_{i_1}\ell_{i_2}\cdots\ell_{i_{_{s}}}\big(D_{i_1}A^{\alpha\beta}D_\beta D_{i_2}\cdots D_{i_{_{s}}}u+\sum_{\tau=1}^{s-1}D_{i_1}\cdots D_{i_{\tau}}(D_{i_{\tau+1}}A^{\alpha\beta}D_\beta D_{i_{\tau+2}}\cdots D_{i_{_{s}}}u)\big)\nonumber\\
&\quad+\delta_{\alpha d}\sum_{j=1}^m\mathbbm{1}_{x^d>h_j(x')} (n^d_j(x'))^{-1}\breve h_j(t,x')-A^{\alpha\beta}\tilde F_\beta,
\end{align}
and
\begin{align}\label{deftildeF}
\tilde F_\beta&:=\sum_{j=1}^{m+1}\mathbbm{1}_{_{(-1,0)\times\cD_j}}D_\beta (\ell_{i_{1}}\ell_{i_{2}}\cdots\ell_{i_{s}})D_{i_1}D_{i_2}\cdots D_{i_{_{s}}}u(t_0,P_jx_0)+\cdots\nonumber\\
&\quad+\sum_{j=1}^{m+1}\mathbbm{1}_{_{(-1,0)\times\cD_j}}D_\beta\big((D_{\ell_{i_{s-1}}} \ell_{i_s})\ell_{i_{1}}\ell_{i_{2}}\cdots\ell_{i_{s-2}}\big)\big(D_{i_1}D_{i_2}\cdots D_{i_{_{s-2}}}D_{i_{_{s}}}u(t_0,P_jx_0)\nonumber\\
&\quad\quad+(x-x_0)\cdot DD_{i_{1}}D_{i_{2}}\cdots D_{i_{_{s-2}}}D_{i_{_{s}}}u(t_0,P_jx_0)\big)\notag\\
&\quad+\sum_{j=1}^{m+1}\mathbbm{1}_{_{(-1,0)\times\cD_j}}
(D_{\ell_{i_{s-1}}}\ell_{i_s})\ell_{i_{1}}\ell_{i_{2}}\cdots
\ell_{i_{s-2}} D_\beta D_{i_{1}}D_{i_{2}}\cdots D_{i_{_{s-2}}}D_{i_{_{s}}}u(t_0,P_jx_0).
\end{align}
The mean oscillation of $\breve f^\alpha$ vanishes at a certain rate (see the proof of \eqref{estFhigh} below). Hence, it suffices to prove the piecewise regularity of $\breve u$.

Denote
$$\breve U:=\breve U(z;z_0)=n^\alpha(A^{\alpha\beta}D_\beta \breve u-\breve f^\alpha).$$
We are going to prove that $D_{\ell^{k}}\breve u$ and $\breve U$ are piecewise $C^{\mu'/2,\mu'}$, $k=1,\ldots,d-1$. Similar to the argument in Section \ref{C2delta}, we will prove the following proposition.
\begin{proposition}\label{proptbreveu}
Let $\varepsilon\in(0,1)$ and $p\in(1,\infty)$. Suppose that $A^{\alpha\beta}$ and $f^\alpha$ satisfy Assumption \ref{assump}  with $s\geq2$. If $u\in \mathcal{H}_{p}^{1}((-1,0)\times B_{1})$ is a weak solution to
$$-u_{t}+D_{\alpha}(A^{\alpha\beta}D_{\beta}u)=D_{\alpha}f^\alpha\quad\mbox{in}~ (-1,0)\times B_{1},$$
then the following assertions hold.

\begin{enumerate}
\item For any $z_{0},z_1\in  (-1+\varepsilon,0)\times B_{1-\varepsilon}$, we have for $k'=1,\ldots,d-1$,
\begin{align}\label{holderbreveuU}
&|(D_{\ell^{k}}\breve u(z_{0};z_0),\breve U(z_{0};z_0))-(D_{\ell^{k}}\breve u(z_{1};z_1),\breve U(z_{1};z_1))|\nonumber\\
&\leq N|z_{0}-z_{1}|_p^{\mu'}\Big(\|Du\|_{L_{1}((-1,0)\times B_{1})}+\sum_{j=1}^{M}|f|_{(s+\delta)/2,s+\delta;(-1,0)\times \overline{\cD_{j}}}\Big),
\end{align}
where $\mu'=\min\big\{\frac{1}{2},\mu\big\}$, $N$ depends on $n,d,m,p,\nu,\varepsilon,|A|_{(s+\delta)/2,s+\delta;(-1,0)\times\overline{\cD_{j}}}$, and the $C^{s+1+\mu}$ characteristic of $\cD_{j}$.
\item It holds that $u\in C^{(s+1+\mu')/2,s+1+\mu'}((-1+\varepsilon,0)\times (B_{1-\varepsilon}\cap\overline{{\cD}_{j_0}}))$,  $j_0=1,\ldots,m+1$, and
\begin{align*}
&|u|_{(s+1+\mu')/2,s+1+\mu';(-1+\varepsilon,0)\times (\cD_{\varepsilon}\cap\overline{{\cD}_{j_0}})}\\
&\leq N\Big(\|Du\|_{L_{1}((-1,0)\times \cD)}+\sum_{j=1}^{M}|f|_{(s+\delta)/2,s+\delta;(-T,0)\times \overline{\cD_{j}}}\Big),
\end{align*}
where $N$ depends on $n,d,m,p,\nu,\varepsilon,|A|_{(s+\delta)/2,s+\delta;(-1,0)\times\overline{\cD_{j}}}$, and the $C^{s+1+\mu}$ characteristic of $\cD_{j}$.
\end{enumerate}
\end{proposition}

To prove Proposition \ref{proptbreveu}, we denote
$$
\Psi(z_{0},r):=\inf_{\mathbf q^{k},\mathbf Q\in\mathbb R^{n}}\left(\fint_{Q_r^-(z_0)}\big(|D_{\ell^{k}}\breve u(z;z_0)-\mathbf q^{k}|^{\frac{1}{2}}+|\breve U(z;z_0)-\mathbf Q|^{\frac{1}{2}}\big)\,dz\right)^{2}$$
and prove a decay estimate of it. We set
\begin{equation*}
\begin{split}
&\breve {\mathfrak u}(t,y;\Lambda z_0)=\breve u(t,x;z_0), \quad \mathcal{A}^{\alpha\beta}(t,y)=\Lambda^{\alpha k}A^{ks}(t,x)\Lambda^{s\beta}, \\
& \breve {\mathfrak f}^\alpha(t,y;\Lambda z_0)=\Lambda^{\alpha k}\breve f^k(t,x;z_0),\quad \breve{\mathfrak g}(t,y)=\breve g(t,x),
\end{split}
\end{equation*}
where $y=\Lambda x$ and $\Lambda=(\Lambda^{\alpha\beta})_{\alpha,\beta=1}^{d}$ is a $d\times d$ orthogonal matrix defined in Section \ref{sec 2.3}. Then we have from \eqref{eqbreveu} that $\breve{\mathfrak u}$ satisfies
\begin{equation*}
-\breve{\mathfrak u}_t+D_\alpha(\mathcal{A}^{\alpha\beta}D_\beta \breve{\mathfrak u})=\breve{\mathfrak g}+ D_\alpha \breve {\mathfrak f}^\alpha\quad\mbox{in}~\Lambda(Q_{3/4}^-),
\end{equation*}
where $\Lambda(Q_{3/4}^-):=(-9/16,0)\times\Lambda(B_{3/4})$.
Denote
\begin{align*}
&\varphi(\Lambda z_{0},r)\\
&:=\inf_{\mathbf q^{k},\mathbf Q\in\mathbb R^{n}}\left(\fint_{Q_r^-(\Lambda z_0)}\big(|D_{y^{k}}\breve{\mathfrak u}(z;\Lambda z_0)-\mathbf q^{k}|^{\frac{1}{2}}+|\mathcal{A}^{d\beta}D_{y^\beta}\breve{\mathfrak u}(z;\Lambda z_0)-\tilde {\mathfrak f}^d-\mathbf Q|^{\frac{1}{2}}\big)\,dz\right)^{2}.
\end{align*}
The following result holds.

\begin{proposition}\label{lemma iteravarp}
For any $0<\rho\leq r\leq 1/4$, we have
\begin{align*}
\varphi(\Lambda z_{0},\rho)\leq N\Big(\frac{\rho}{r}\Big)^{\mu'}\varphi(\Lambda z_{0},r/2)+N\rho^{\mu'}\mathcal{C}_1,
\end{align*}
where  $\mu'=\min\big\{\frac{1}{2},\mu\big\}$, $N$ depends on $n,d,m,p,\nu,|A|_{(s+\delta)/2,s+\delta;(-1,0)\times\overline{\cD_{j}}}$, and the  $C^{s+1+\mu}$ norm of $h_j$, and
\begin{align}\label{defC1}
\mathcal{C}_1:=\sum_{j=1}^{m+1}[D^su]_{1/2,1;Q_{r}^-(z_0)\cap((-1+\varepsilon,0)\times\overline\cD_j)}+\sum_{j=1}^{M}|f|_{(s+\delta)/2,s+\delta;(-1,0)\times\overline{\cD_{j}}}+\|Du\|_{L_{1}(\cQ)}.
\end{align}
\end{proposition}

The proof of Proposition \ref{lemma iteravarp} is similar to that of Proposition \ref{lemma iteraphi}. We shall provide the main ingredients of the proof. We first rewrite \eqref{def-brevefalpha} as $\breve f^\alpha(t,x;z_0)=\breve f_1^\alpha(t,x;z_0)+\breve f_2^\alpha(t,x)$, where
\begin{align*}
\breve f_1^\alpha(t,x;z_0)
&=A^{\alpha\beta}\Big(D_\beta(\ell_{i_1}\ell_{i_2}\cdots\ell_{i_{_{s}}})D_{i_1}D_{i_2}\cdots D_{i_{_{s}}}u+D_\beta(R(u))-\tilde F_\beta\Big),
\end{align*}
and
\begin{align*}
&\breve f_2^\alpha(t,x)=\ell_{i_1}\ell_{i_2}\cdots\ell_{i_{_{s}}}D_{i_1}D_{i_2}\cdots D_{i_{_{s}}}f^\alpha+\delta_{\alpha d}\sum_{j=1}^m\mathbbm{1}_{x^d>h_j(x')} (n^d_j(x'))^{-1}\breve h_j(t,x')\nonumber\\
&\qquad-\ell_{i_1}\ell_{i_2}\cdots\ell_{i_{_{s}}}\big(D_{i_1}A^{\alpha\beta}D_\beta D_{i_2}\cdots D_{i_{_{s}}}u+\sum_{\tau=1}^{s-1}D_{i_1}\cdots D_{i_{\tau}}(D_{i_{\tau+1}}A^{\alpha\beta}D_\beta D_{i_{\tau+2}}\cdots D_{i_{_{s}}}u)\big).
\end{align*}
Here $R(u)$ is defined in \eqref{defLu}. Set
\begin{equation*}
\breve {\mathfrak f}_1^\alpha(t,y;\Lambda z_0)=\Lambda^{\alpha k}\breve f_1^k(t,x;z_0),\quad \breve {\mathfrak f}_2^\alpha(t,y)=\Lambda^{\alpha k}\breve f_2^k(t,x).
\end{equation*}
Then we  choose $F^\alpha$ and $G$ as follows:
\begin{equation}\label{def-Falphahigh}
\begin{split}
F^\alpha:=F^\alpha(t,y;\Lambda z_0)&=(\overline{{\mathcal A}^{\alpha\beta}}(y^{d})-\mathcal{A}^{\alpha\beta}(t,y))D_{y^\beta}\breve{\mathfrak u}(t,y;\Lambda z_0)+\breve {\mathfrak f}_1^\alpha(t,y;\Lambda z_0)\\
&\quad+\breve {\mathfrak f}_2^\alpha(t,y)-\bar{\mathfrak f}_2^\alpha(y^d),\quad F=(F^1,\cdots,F^d),
\end{split}
\end{equation}
and
\begin{equation}\label{def-Ghigh}
G=\breve{\mathfrak g}(t,y)=\breve g(t,x),
\end{equation}
where $\bar{\mathfrak f}_2^\alpha(y^d)$ is the piecewise constant function corresponding to $\breve {\mathfrak f}_2^\alpha(t,y)$.

\begin{lemma}\label{lemmaFGhigh}
Let $F$ and $G$ be defined  in \eqref{def-Falphahigh} and \eqref{def-Ghigh}, respectively. Then we have
\begin{align}\label{estFhigh}
\|F\|_{L_1(Q_{r}^-(\Lambda z_0))}\leq Nr^{d+2+\mu'}\mathcal{C}_1
\end{align}
and
\begin{align}\label{estGL1high}
&\|G\|_{L_{1}(Q_{r}^-(\Lambda z_0))}\notag\\
&\leq N r^{d+\frac{3}{2}}\Big(\sum_{j=1}^{m+1}\|D^{s+1}u\|_{L_\infty(Q_{r}^-(z_0)
\cap((-1+\varepsilon,0)\times\cD_j))}\notag\\
&\qquad\qquad+\sum_{j=1}^{M}|f|_{(s+\delta)/2,s+\delta;(-1,0)\times\overline{\cD_{j}}}
+\|Du\|_{L_{1}(\cQ)}\Big),
\end{align}
where $\mathcal{C}_1$ is defined in \eqref{defC1}, $\mu'=\min\big\{\frac{1}{2},\mu\big\}$, $N$ depends on  $|A|_{(s+\delta)/2,s+\delta;(-1,0)\times\overline{\cD_{j}}}$,  $n,d,m,p,\nu$, and the $C^{s+1+\mu}$ norm of $h_j$.
\end{lemma}

\begin{proof}
We first note that $\breve {\mathfrak f}_2^\alpha$ is piecewise $C^{\mu'/2,\mu'}$, and similar to \eqref{estDf}, we have
\begin{align*}
\int_{Q_{r}^-(\Lambda z_0)}\big|\breve {\mathfrak f}_2^\alpha(t,y)-\bar{\mathfrak f}_2^\alpha(y^d)\big|
\leq Nr^{d+2+\mu'}\bigg( \sum_{j=1}^{M}|f|_{(s+\delta)/2,s+\delta;(-1,0)\times\overline{\cD_{j}}}+\|Du\|_{L_{1}(\cQ)}\bigg),
\end{align*}
where $\mu'=\min\big\{\frac{1}{2},\mu\big\}$.  By using the same argument  in deriving  \eqref{Af1}, we obtain
\begin{align*}
&(\overline{{\mathcal A}^{\alpha\beta}}(y^{d})-\mathcal{A}^{\alpha\beta}(t,y))D_{y^\beta}\breve{\mathfrak u}(t,y;\Lambda z_0)+\breve {\mathfrak f}_1^\alpha(t,y;\Lambda z_0)\nonumber\\
&=(\overline{{\mathcal A}^{\alpha\beta}}(y^{d})-\mathcal{A}^{\alpha\beta}(t,y))\Gamma^{\beta k}(\ell_{i_1}\ell_{i_2}\cdots\ell_{i_{_{s}}}D_kD_{i_1}D_{i_2}\cdots D_{i_{_{s}}}u-D_{k}{\mathsf u}-F_k)\nonumber\\
&\quad+\overline{{\mathcal A}^{\alpha\beta}}(y^{d})\Gamma^{\beta k}\big(D_k(\ell_{i_1}\ell_{i_2}\cdots\ell_{i_{_{s}}}) D_{i_1}D_{i_2}\cdots D_{i_{_{s}}}u+D_k(R(u))-\tilde F_k\big),
\end{align*}
where $F_k$ and $\tilde F_k$ are defined in \eqref{defmathF} and \eqref{deftildeF}, respectively.
Then similar to \eqref{estADtildeu}, using \eqref{est-Dmathu} and  Lemma \ref{lemell}(iii), we have
\begin{align*}
\left\|(\overline{{\mathcal A}^{\alpha\beta}}(y^{d})-\mathcal{A}^{\alpha\beta}(\cdot))D_{y^\beta}\breve{\mathfrak u}(\cdot;\Lambda z_0)+\breve {\mathfrak f}_1^\alpha(\cdot;\Lambda z_0)\right\|_{L_1(Q_{r}^-(\Lambda z_0))}\leq Nr^{d+5/2}\mathcal{C}_1,
\end{align*}
where $\mathcal{C}_1$ is defined in \eqref{defC1}. Thus \eqref{estFhigh} is proved. By \eqref{estDellk0}, we have \eqref{estGL1high}. The proof of Lemma \ref{lemmaFGhigh} is finished.
\end{proof}

With the above preparations, applying Lemmas \ref{weak est barv} and \ref{lemmaFGhigh}, and following the process in the proof of  Proposition \ref{lemma iteraphi}, we obtain Proposition \ref{lemma iteravarp}. With Proposition \ref{lemma iteravarp} at hand,  replicating  the proof Lemma \ref{lemma itera}, we have the following decay estimate of $\Psi(z_{0},r)$.

\begin{lemma}\label{lemma iterab}
For any $0<\rho\leq r\leq 1/4$, we have
\begin{align*}
\Psi(z_{0},\rho)\leq N\Big(\frac{\rho}{r}\Big)^{\mu'}\Psi(z_{0},r/2)+N\rho^{\mu'}\mathcal{C}_1,
\end{align*}
where $\mathcal{C}_1$ is defined in \eqref{defC1}, $\mu'=\min\big\{\frac{1}{2},\mu\big\}$, $N$ is a constant depending on $n$, $d$, $m$, $p$, $\nu$,  $\varepsilon$, $|A|_{(s+\delta)/2,s+\delta;(-1,0)\times\overline{\cD_{j}}}$, and the  $C^{s+1+\mu}$ norm of $h_j$.

\end{lemma}

We shall further establish the estimates of $[D^su]_{t,(1+\delta)/2}$, $|u_t|_{(s-1+\mu')/2,s-1+\mu'}$, and $\|D^{s+1}u\|_{L_\infty}$.

\begin{lemma}\label{lemmaDDDu}
Under the same assumptions as in Proposition \ref{proptbreveu}, we have
\begin{align*}
&\sum_{j=1}^{m+1}[D^su]_{t,(1+\delta)/2;(-1+3\varepsilon,0)\times (B_{1-2\varepsilon}\cap\overline{\cD_{j}})}+\sum_{j=1}^{m+1}|u_t|_{(s-1+\mu')/2,s-1+\mu';(-1+\varepsilon,0)\times (B_{1-\varepsilon}\cap\overline{\cD_{j}})}\nonumber\\
&\quad+\sum_{j=1}^{m+1}\|D^{s+1}u\|_{L_\infty(Q_{1/4}^-\cap((-1+\varepsilon,0)\times\cD_j))}\\
&\leq N\|Du\|_{L_{1}(Q_{3/4}^{-})}+N\sum_{j=1}^{M}|f|_{(s+\delta)/2,s+\delta;(-1,0)\times\overline{\cD_{j}}},
\end{align*}
where $N>0$ is a constant depending only on $n,d,p,m,\nu,\varepsilon$, $|A|_{(s+\delta)/2,s+\delta;(-1,0)\times\overline{\cD_{j}}}$, and the $C^{s+1+\mu}$ norm of $h_j$.
\end{lemma}

\begin{proof}
We start with the proof of the estimate of $[D^su]_{t,1/2}$. Since
$$
A^{\alpha\beta},f^\alpha\in C^{(s+\delta)/2,s+\delta}((-1+\varepsilon,0)\times(B_{1-\varepsilon}\cap\overline{\cD}_j)),
$$
one can verify that
$$\delta_h^{\gamma}A^{\alpha\beta}, \delta_h^{\gamma}f^\alpha\in C^{\frac{s+\delta}{2}-\gamma,s+\delta-2\gamma}((-1+2\varepsilon,0)\times(B_{1-\varepsilon}\cap\overline{{\cD}_{j}})),$$
where $\gamma\in \big(0,\frac{1+\delta}{2}\big)$, $h\in(0,\varepsilon)$, and $\delta_h^\gamma$ is defined in \eqref{def-deltaf}. Applying the inductive assumption to \eqref{systemut}, we obtain $\delta_h^{\gamma}u\in C^{(s+\gamma')/2,s+\gamma'}((-1+3\varepsilon,0)\times (B_{1-2\varepsilon}\cap\overline{{\cD}_{j}}))$ with $\gamma':=\min\{1+\delta-2\gamma,\frac{1}{2},\mu\}>0$ and
\begin{align*}
|\delta_h^{\gamma}u|_{(s+\gamma')/2,s+\gamma';(-1+3\varepsilon,0)\times (B_{1-2\varepsilon}\cap\overline{{\cD}_{j}})}\leq N\big(\sum_{j=1}^{M}|f|_{(s+\delta)/2,s+\delta;(-1,0)\times\overline{\cD_{j}}}+\|Du\|_{L_{1}(\cQ)}\big).
\end{align*}
Taking $\gamma$ to be sufficiently close to $\frac{1+\delta}{2}$, we have $\gamma'=1+\delta-2\gamma$, $\gamma+\frac{\gamma'}{2}=\frac{1+\delta}{2}$, and
\begin{align*}
[D^su]_{t,(1+\delta)/2;(-1+3\varepsilon,0)\times (B_{1-2\varepsilon}\cap\overline{{\cD}_{j}})}\leq N\big(\sum_{j=1}^{M}|f|_{(s+\delta)/2,s+\delta;(-1,0)\times\overline{\cD_{j}}}+\|Du\|_{L_{1}(\cQ)}\big).
\end{align*}

Next we prove the higher regularity of $u_t$. By differentiating the equation \eqref{systems} with respect to $t$, we get
\begin{equation*}
-(u_t)_t+D_\alpha(A^{\alpha\beta}D_\beta u_t)=D_\alpha \partial_tf^\alpha-D_\alpha(\partial_tA^{\alpha\beta}D_\beta u).
\end{equation*}
By the inductive assumption, we have
\begin{equation}\label{Dl-u}
u\in C^{(s+\mu')/2,s+\mu'}((-1+\varepsilon,0)\times (B_{1-\varepsilon}\cap\overline{{\cD}_{j}})),
\end{equation}
where $\mu'=\min\big\{\frac{1}{2},\mu\big\}$, and thus  $D_\beta u\in C^{(s-1+\mu')/2,s-1+\mu'}((-1+\varepsilon,0)\times (B_{1-\varepsilon}\cap\overline{{\cD}_{j}}))$. Then combining
$$
\partial_tf^\alpha,~\partial_tA^{\alpha\beta}\in C^{(s-2+\delta)/2,s-2+\delta}((-1+\varepsilon,0)\times (B_{1-\varepsilon}\cap\overline{{\cD}_{j}}))
$$
and the inductive assumption \eqref{pieceu}, we obtain
\begin{equation}\label{estDDtu}
u_t\in C^{(s-1+\mu')/2,s-1+\mu'}((-1+\varepsilon,0)\times (B_{1-\varepsilon}\cap\overline{{\cD}_{j}}))
\end{equation}
and
\begin{align*}
&|u_t|_{(s-1+\mu')/2,s-1+\mu';(-1+\varepsilon,0)\times (B_{1-\varepsilon}\cap\overline{{\cD}_{j}})}\notag\\
&\leq N\Big(\|Du\|_{L_{1}((-1,0)\times \cD)}+\sum_{j=1}^{M}|f|_{(s-1+\delta)/2,s+\delta;(-1,0)\times \overline{\cD_{j}}}\Big).
\end{align*}

Finally, we estimate  $\sum_{j=1}^{m+1}\|D^{s+1}u\|_{L_\infty(Q_{1/4}^-\cap((-1+\varepsilon,0)\times\cD_j)))}$.
Similar to \eqref{estDtildeuU}, we have
\begin{align*}
&|D_{\ell^{k}}\breve u(z_0;z_0)|+|\breve U(z_0;z_0)|\nonumber\\
&\leq Nr^{\mu'}\sum_{j=1}^{m+1}\|D^{s+1}u\|_{L_\infty(Q_{2r}^-(z_1)\cap((-1+\varepsilon,0)\times\cD_j))}\\
&\quad +Nr^{-1}\big(\|Du\|_{L_{1}(\cQ)}+\sum_{j=1}^{M}|f|_{(1+\delta)/2,1+\delta;(-1,0)\times\overline{\cD_{j}}}\big)
\end{align*}
for any  $k=1,\ldots,d-1$, $z_0\in (-1+\varepsilon,0)\times (\cD_{\varepsilon}\cap{{\cD}_{j}})$, and $r\in(0,1/4)$. By using the definition of $\breve u$ in  \eqref{defbreveu}, we have
\begin{align}\label{Dbreveu}
D_{\ell^{k}}\breve u(z;z_0)
&=D_{\ell^{k}}D_{\ell}\cdots D_{\ell}u-D_{\ell^{k}} u_0-D_{\ell^{k}}{\mathsf u}
\end{align}
and
\begin{align*}
A^{\alpha\beta}D_\beta \breve u
&=A^{\alpha\beta}D_\beta(D_{\ell}\cdots D_{\ell}u-u_0-{\mathsf u}),
\end{align*}
where $u_0$ and ${\mathsf u}$ are defined in \eqref{defu-0}  and \eqref{defsfu}, respectively.
Then combining \eqref{Dlu}  and \eqref{def-brevefalpha}, we obtain
\begin{align}\label{breveU}
&\breve U(z;z_0)=n^\alpha\big(A^{\alpha\beta}D_\beta \breve u-\breve f^\alpha\big)\nonumber\\
&=n^\alpha\big(A^{\alpha\beta}\ell_{i_1}\ell_{i_2}\cdots\ell_{i_{_{s}}}D_\beta D_{i_1}D_{i_2}\cdots D_{i_{_{s}}}u-A^{\alpha\beta}D_\beta{\mathsf u}-\ell_{i_1}\ell_{i_2}\cdots\ell_{i_{_{s}}}D_{i_1}D_{i_2}\cdots D_{i_{_{s}}}f^\alpha\nonumber\\
&\quad+\ell_{i_1}\ell_{i_2}\cdots\ell_{i_{_{s}}}\big(D_{i_1}A^{\alpha\beta}D_\beta D_{i_2}\cdots D_{i_{_{s}}}u+\sum_{\tau=1}^{s-1}D_{i_1}\cdots D_{i_{\tau}}(D_{i_{\tau+1}}A^{\alpha\beta}D_\beta D_{i_{\tau+2}}\cdots D_{i_{_{s}}}u)\big)\nonumber\\
&\quad-\delta_{\alpha d}\sum_{j=1}^m\mathbbm{1}_{x^d>h_j(x')} (n^d_j(x'))^{-1}\breve h_j(t,x')-A^{\alpha\beta}F_\beta\big),
\end{align}
where $\breve h_j(t,x')$ and $F_\beta$ are defined in  \eqref{brevehj} and \eqref{defmathF}, respectively.
Note that there are  $n\tbinom{d+s}{s+1}$ components in $D^{s+1}u$. From \eqref{Dbreveu} and \eqref{breveU}, we have $n\tbinom{d+s-1}{s+1}+n\tbinom{d+s-2}{s}$ equations. Since $$
\tbinom{d+s}{s+1}-\tbinom{d+s-1}{s+1}-\tbinom{d+s-2}{s}=\tbinom{d+s-2}{s-1},$$
we need another $n\tbinom{d+s-2}{s-1}$ equations to solve for $D^{s+1}u$.	For this, by taking the $(s-1)$-th derivative of the equation \eqref{systems}  with respect to $x$ in each subdomain,  we get the following $n\tbinom{d+s-2}{s-1}$ equations
\begin{align}\label{eqDDDu}
&A^{\alpha\beta}D_{\alpha\beta} D^{s-1}u\notag\\
&=D^{s-1}D_\alpha f^{\alpha}+D^{s-1}u_t-\sum_{i=1}^{s-1}\tbinom{s-1}{i}D^i A^{\alpha\beta}D^{s-1-i}D_{\alpha\beta} u-D^{s-1}(D_\alpha A^{\alpha\beta}D_\beta u).
\end{align}
It follows from the assumption $A^{\alpha\beta}\in C^{(s+\delta)/2,s+\delta}((-1+\varepsilon,0)\times (B_{1-\varepsilon}\cap\overline{{\cD}_{j}}))$  and \eqref{Dl-u} that
$$\sum_{i=1}^{s-1}\tbinom{s-1}{i}D^i A^{\alpha\beta}D^{s-1-i}D_{\alpha\beta} u,~D^{s-1}(D_\alpha A^{\alpha\beta}D_\beta u)\in C^{\mu'/2,\mu'}((-1+\varepsilon,0)\times (B_{1-\varepsilon}\cap\overline{{\cD}_{j}})).$$
Combining \eqref{estDDtu} and the condition on $f$, one can see that the right-hand side of \eqref{eqDDDu} is of class piecewise $C^{\mu'/2,\mu'}$. By mimicking the argument in the proof of Lemma \ref{lemma Dtildeu} (see Step 2), we can solve for $D^{s+1}u$ by using Cramer's rule and \eqref{Dbreveu}--\eqref{eqDDDu}. Moreover,
\begin{align*}
&\sum_{j=1}^{m+1}\|D^{s+1}u\|_{L_{\infty}(Q_r^-(z_1))}\nonumber\\
&\leq Nr^{\mu'}\sum_{j=1}^{m+1}\|D^{s+1}u\|_{L_\infty(Q_{2r}^-(z_1)\cap((-1+\varepsilon,0)\times\cD_j))}\notag\\
&\quad +Nr^{-1}\big(\|Du\|_{L_{1}(\cQ)}+\sum_{j=1}^{M}|f|_{(1+\delta)/2,1+\delta;(-1,0)\times\overline{\cD_{j}}}\big).
\end{align*}
With this estimate, we obtain our desired estimate of
$$
\sum_{j=1}^{m+1}\|D^{s+1}u\|_{L_\infty(Q_{1/4}^-\cap((-1+\varepsilon,0)\times\cD_j)))}
$$
by using a standard iteration argument. See, for instance, \cite[Lemma 3.4]{dx2019}. The lemma is proved.
\end{proof}

Finally, we  finish the proof of Proposition \ref{proptbreveu}.
\begin{proof}[{\bf Proof of Proposition \ref{proptbreveu}}]
(a) The proof of \eqref{holderbreveuU} closely follows  the argument in Section \ref{prfprop} by using Lemmas \ref{lemma iterab} and \ref{lemmaDDDu}, and thus is omitted.
	
(b) Note that the directional derivatives of $\ell$ are in general not piecewise H\"older continuous. However, for any $z_0=(t_0,x_0)\in(-1+\varepsilon,0)\times (B_{1-\varepsilon}\cap\overline{\cD}_{j_0})$, by using \eqref{Dlu} and \eqref{defu-0}, one can verify that the terms containing (directional) derivatives of $\ell$ at $x_0$ in  \eqref{Dbreveu} and \eqref{breveU}  are cancelled. Using \eqref{Dbreveu}--\eqref{eqDDDu} with $z=z_0$ and Cramer's rule, we find that $D^{s+1}u(z_0)$ is expressed by
\begin{equation}\label{combination}
\begin{split}
&D_{\ell^{k}}\breve u(z_{0};z_0),~ \breve U(z_{0};z_0),~D^{s-1}u_t(z_0),~D{\mathsf u}(z_0),\\
&\delta_{\alpha d}\sum_{j=1}^m\mathbbm{1}_{x^d>h_j(x')} (n^d_j(x'_0))^{-1}\breve h_j(t_0,x'_0),~n^\alpha(x_0),~D^{s}f^\alpha(z_0),\\
&A^{\alpha\beta}(z_0),~D^{i}A^{\alpha\beta}(z_0),~D^iu(z_0),~D^iu(t_0,P_jx_0),\\
&\mbox{and (directional)~derivatives~of}~\tilde\ell_{,j}(x_0)~\mbox{for~}j\neq j_0,
\end{split}
\end{equation}
where $i=1,\ldots,s$, and $\tilde{\ell}_{,j}:=(\tilde{\ell}_{1,j},\dots,\tilde{\ell}_{d,j})$ is the smooth extension  of $\ell|_{\cD_j}$ to $\cup_{k=1,k\neq j}^{m+1}\cD_k$. Similarly, for any $\tilde z_0=(\tilde t_0,\tilde z_0)\in(-1+\varepsilon,0)\times (B_{1-\varepsilon}\cap\overline{\cD}_{j_0})$, $D^{s+1}u(\tilde z_0)$ is expressed by \eqref{combination} with $z_0$ replaced with $\tilde z_0$. Thus, combining \eqref{pieceu}, \eqref{defsfu}, \eqref{holderbreveuU}, Assumption \ref{assump} (b), and Lemma \ref{lemmaDDDu}, we obtain
\begin{align*}
&[D^{s+1}u]_{\mu'/2,\mu';(-1+\varepsilon,0)\times (B_{1-\varepsilon}\cap\overline{\cD}_{j_0})}\\
& \leq N\Big(\|Du\|_{L_{1}((-1,0)\times \cD)}+\sum_{j=1}^{M}|f|_{(s+\delta)/2,s+\delta;(-1,0)\times \overline{\cD_{j}}}\Big)
\end{align*}
 for $j_0=1,\ldots,m+1$.
The proof of Proposition \ref{proptbreveu} is complete.
\end{proof}

\section*{Acknowledgement}

The authors would like to thank Prof. Yanyan Li and Youchan Kim for very helpful discussions on the subject.

\appendix

\section{}\label{Append}
Since the local boundedness and piecewise H\"{o}lder-regularity of $Du$,  and some auxiliary estimates proved in \cite{dx2021} are crucially used in the proofs above, we present them here for reader's convenience.

\begin{lemma}\cite[Lemma 4.5, Theorem 2.3]{dx2021}
\label{lemlocbdd}
Let $Q_1^-=(-1,0)\times B_1$, $\varepsilon\in (0,1)$, $p\in(1,\infty)$,  $A^{\alpha\beta}$ and $f^{\alpha}$ satisfy Assumption \ref{assump} with $s=0$. Let $u\in \mathcal{H}_{p}^{1}(Q_1^-)$ be a weak solution to \eqref{systems} in $Q_1^-$. Then $u\in C^{(1+\mu')/2,1+\mu'}((-1+\varepsilon,0)\times (B_{1-\varepsilon}\cap\overline{{\cD}_{j_0}}))$ and it holds that
\begin{align*}
&\|Du\|_{L_{\infty}(Q_{1/4}^{-})}+|u|_{(1+\mu')/2,1+\mu';(-1+\varepsilon,0)\times (B_{1-\varepsilon}\cap\overline{{\cD}_{j_0}})}\\
&\leq N\big(\|Du\|_{L_{1}(Q_1^-)}+\sum_{j=1}^{M}|f|_{\delta/2,\delta;(-1,0)\times\overline{\cD_{j}}}\big),
\end{align*}
where $j_0=1,\ldots,m+1$, $\mu'=\min\{\frac{1}{2},\mu\}$, $N>0$ is a constant depending only on $n,d,m,p,\nu,\varepsilon$, $|A|_{\delta/2,\delta;(-1,0)\times\overline{\cD_{j}}}$, and the $C^{1+\mu}$ norm of $h_j$.
\end{lemma}

Let us recall the $\mathcal{H}_{p}^{1}$-estimate for parabolic equations with variably partially small bounded mean oscillation coefficients, i.e., there exists a small enough constant $\gamma_{0}=\gamma_{0}(d,n,p,\nu)\in (0,1/2)$ and a constant $r_{0}\in(0,1)$ such that for any $r\in(0,r_{0})$ and $(t_{0},x_{0})\in Q_{1}^{-}$ with $B_{r}(x_{0})\subset B_{1}$, in a coordinate system depending on $(t_{0},x_{0})$ and $r$, one can find a $\bar{A}=\bar{A}(x^{d})$ satisfying
\begin{equation}\label{BMO}
\fint_{Q_{r}^{-}(t_{0},x_{0})}|A(t,x)-\bar A(x^{d})|\ dx\,dt\leq\gamma_{0}.
\end{equation}

\begin{lemma}\label{lem loc lq}\cite[Lemma 3.4]{dx2021}
Let $1<q<\infty$. Assume that $A$ satisfies \eqref{BMO} with a sufficiently small constant $\gamma_{0}=\gamma_{0}(d,n,q,\nu)\in (0,1/2)$ and $u\in \mathcal{H}_{1,\text{loc}}^{1}$ satisfies
$$-u_{t}+D_\alpha(A^{\alpha\beta}D_\beta u)=g+ \Div f\quad\mbox{in}~Q_1^-,$$
where $f,g\in L_{q}(Q_{1}^{-})$. Then
$$\|u\|_{\mathcal{H}_{q}^{1}(Q_{1/2}^{-})}\leq N(\|u\|_{L_{1}(Q_{1}^{-})}+\|g\|_{L_{q}(Q_{1}^{-})}+\|f\|_{L_{q}(Q_{1}^{-})}),$$
where $N$ depends on $n,d,\nu,q$, and $r_{0}$.
\end{lemma}

The $\mathcal{H}_{p}^{1}$-solvability for parabolic systems with coefficients which satisfy \eqref{BMO} in $Q_{1}^{-}$ is also used in the proof above.

\begin{lemma}\label{solvability}\cite[Lemma 3.5]{dx2021}
For any $p\in(1,\infty)$, $f\in L_{p}(Q_{1}^{-})$, the following hold.
\begin{enumerate}
\item
For any $u\in \mathcal{\mathring{H}}_{p}^{1}(Q_{1}^{-})$ satisfying
\begin{align}\label{approxi sol}
-u_t+D_{\alpha}(\boldsymbol{A}^{\alpha\beta}D_\beta u)=\Div f&\quad\mbox{in}~Q_{1}^{-}
\end{align}
and $u(-1,\cdot)\equiv0$ in $B_{1}$,
we have
\begin{align}\label{est H u}
\|u\|_{\mathcal{H}_{p}^{1}(Q_{1}^{-})}\leq N\|f\|_{L_{p}(Q_{1}^{-})},
\end{align}
where $\boldsymbol{A}^{\alpha\beta}$ is defined in \eqref{mathcalA}, $N$ depends on $d,n,p,\nu$, and $r_{0}$.
		
\item
For any $f\in L_{p}(Q_{1}^{-})$, there exists a unique solution $u\in \mathcal{\mathring{H}}_{p}^{1}(Q_{1}^{-})$ of \eqref{approxi sol} with the initial data $u(-1,\cdot)\equiv0$ in $B_{1}$. Furthermore, $u$ satisfies \eqref{est H u}.
\end{enumerate}
\end{lemma}

Denote
$$\mathcal{P}_{0}u:=-u_{t}+D_{\alpha}(\bar{A}^{\alpha\beta}(x^{d})D_{\beta}u).$$
For parabolic systems with coefficients depending only on $x^{d}$, we have the following result.

\begin{lemma}\label{lemma xn}\cite[Lemma 3.6]{dx2021}
Let $p\in(0,\infty)$. Assume that $u\in C_{\text{loc}}^{0,1}$ satisfies $\mathcal{P}_{0}u=0$ in $Q_{1}^{-}$. Then there exists a constant $N=N(n,d,p,\nu)$ such that
\begin{align*}
[D_{x'}u]_{1/2,1;Q_{1/2}^{-}}\leq N\|D_{x'}u\|_{L_{p}(Q_{1}^{-})},\quad [\bar{A}^{d\beta}(x^{d})D_{\beta}u]_{1/2,1;Q_{1/2}^{-}}\leq N\|Du\|_{L_{p}(Q_{1}^{-})}.
\end{align*}
\end{lemma}

\bibliographystyle{plain}

\end{document}